\documentclass[english,reqno,11pt,empty]{amsart}
\usepackage{amsmath, amsthm, amsfonts}
\usepackage{hyperref}
\hypersetup{
	colorlinks=true,
		citecolor=blue!60!black,
		linkcolor=red!60!black,
		urlcolor=green!40!black,
		filecolor=yellow!50!black,
	breaklinks=true,
	pdfpagemode=UseNone,
	bookmarksopen=false,
}
\usepackage{amsmath,amsthm,amssymb}
\usepackage{amscd,indentfirst,epsfig}
\usepackage{latexsym}
\usepackage{times}
\usepackage{enumerate}
\usepackage{mathrsfs}
\usepackage{stmaryrd}
\usepackage{amsopn}
\usepackage{amsmath}
\usepackage{amssymb,dsfont,mathtools}
\usepackage{amsfonts,bm}
\usepackage{amsbsy,amsmath}
\usepackage{amscd}
\usepackage{xcolor}
\usepackage{mathtools}%\usepackage{fourier,setspace,graphicx,xcolor}
\usepackage{cancel}
\usepackage{hyperref} 
\usepackage{mathrsfs}
\usepackage{amsmath,amsthm,amssymb}
\usepackage{latexsym}
\usepackage{latexsym}
\usepackage{bm}
\usepackage{enumerate}
\usepackage{mathrsfs}
\usepackage{stmaryrd}
\usepackage{amsopn}
\usepackage{amsmath}
\usepackage{amssymb}
\usepackage{amsfonts}
\usepackage{amsbsy}
\usepackage{amscd,indentfirst}
\usepackage{hyperref}
\usepackage{amsfonts,amsmath,latexsym,amssymb,verbatim,amsbsy}
\usepackage{amsthm}
\usepackage{colordvi}
\usepackage{dsfont}
\usepackage{hyperref}
\hypersetup{
	colorlinks=true,
		citecolor=blue!60!black,
		linkcolor=red!60!black,
		urlcolor=green!40!black,
		filecolor=yellow!50!black,
	breaklinks=true,
	pdfpagemode=UseNone,
	bookmarksopen=false,
}
\usepackage{amsmath,amsthm,amssymb}
\usepackage{amscd,indentfirst,epsfig}
\usepackage{latexsym}
\usepackage{times}
\usepackage{enumerate}
\usepackage{mathrsfs}
\usepackage{stmaryrd}
\usepackage{amsopn}
\usepackage{amsmath}
\usepackage{amssymb,dsfont,mathtools}
\usepackage{amsfonts,bm}
\usepackage{amsbsy,amsmath}
\usepackage{amscd}
\usepackage{xcolor}

\usepackage[mono=false]{libertine}
\usepackage[T1]{fontenc}
\usepackage{amsthm}
\usepackage{cancel}
\usepackage[cal=euler, scr=boondoxo]{}
\usepackage{microtype}

\usepackage{numprint}

\makeatletter
\def \leq {\leqslant}
\def \le {\leq}
\def \geq {\geqslant}

\def \ge {\geq}

\def\R{\mathbb R}

\def\Z{\mathbb Z}
\def\N{\mathbb N}

\def\D{\mathcal D}

\def\g{\gamma}
\def \ds {\displaystyle}
\def\Gg  {\bm{G}_{\g}}

\def \d {\mathrm{d}}

\def \Q {\mathcal{Q}}

 \def \w {\bm{w}}

\def \X {\mathbb{X}}

\def \ind {\mathbf{1}}
\numberwithin{equation}{section}

\newcommand{\emptylabel}[1]{}

\def\:{\colon}

\def\N{\mathbb{N}}
\def\P{\mathbb{P}}
\def\R{\mathbb{R}}

\def\Z{\mathbb{Z}}

\def\D{\mathcal{D}}

\DeclareMathOperator{\sgn}{sgn}

\def\d{\,\mathrm{d}}
\def\dx{\d x}

\def\dy{\d y}

\def\P{\mathcal P}
\def\p{\partial}

\numberwithin{equation}{section}

\usepackage{natbib}
\newtheorem{theo}{Theorem}[section]
\newtheorem{cor}[theo]{Corollary}
\newtheorem{rmk}[theo]{Remark}
\newtheorem{lem}[theo]{Lemma}
\newtheorem{prp}[theo]{Proposition}

\newtheorem{rem}[theo]{Remark}
\newtheorem{defi}[theo]{Definition}
\theoremstyle{example}

\newcommand{\vertiii}[1]{{\left\vert\kern-0.25ex\left\vert\kern-0.25ex\left\vert #1  %%
    \right\vert\kern-0.25ex\right\vert\kern-0.25ex\right\vert}}                      %%
\newcommand{\verti}[1]{{\left\vert\kern-0.25ex\left\vert\kern-0.25ex\left\vert #1    %%
    \right\vert\kern-0.25ex\right\vert\kern-0.25ex\right\vert}}						 %%	

\title[One-dimensional inelastic Boltzmann equation]{One-dimensional
  inelastic Boltzmann equation: Stability and uniqueness of
  self-similar $L^{1}$-profiles for moderately hard potentials}

\def\theauthor{R. Alonso, V. Bagland, J. A. Ca\~{n}izo, B. Lods, S. Throm}

\hypersetup{pdfauthor={\theauthor}}

\vspace{1cm}

\author{R. Alonso}
\address{Division of Arts \& Sciences, Texas A\&M University at Qatar, Education City, Doha, Qatar.}
\email{ricardo.alonso@qatar.tamu.edu}

\author{V. Bagland}
\address{Universit\'{e} Clermont Auvergne, LMBP, UMR 6620 - CNRS,  Campus des C\'ezeaux, 3, place Vasarely, TSA 60026, CS 60026, F-63178 Aubi\`ere Cedex, France.}
\email{Veronique.Bagland@uca.fr}

\author{J. A. Ca\~{n}izo}
\address{Departamento de Matem\'{a}tica Aplicada \& IMAG, Universidad de Granada, Avenida de Fuentenueva S/N, 18071 Granada, Spain.}
\email{canizo@ugr.es}

\author{B. Lods}
\address{Universit\`{a} degli Studi di Torino \& Collegio Carlo Alberto, Department of Economics, Social Sciences, Applied Mathematics and Statistics ``ESOMAS'', Corso Unione Sovietica, 218/bis, 10134 Torino, Italy.}
\email{bertrand.lods@unito.it}

\author{S. Throm}
\address{{Ume\aa} University,  Department of Mathematics and Mathematical Statistics,  901 87 Ume\aa, Sweden }
\email{sebastian.throm@umu.se}

\date{}

\begin{document}

\maketitle

\begin{abstract}
  We prove the stability of $L^{1}$ self-similar profiles under the hard-to-Maxwell potential limit
  for the one-dimensional inelastic Boltzmann equation with moderately hard potentials which, in
  turn, leads to the uniqueness of such profiles for hard potentials collision kernels of the form $|\cdot|^{\g}$
  with $\g >0$ sufficiently small (explicitly quantified). Our result provides the first
  uniqueness statement for self-similar profiles of inelastic
  Boltzmann models allowing for strong inelasticity besides the
  explicitly solvable case of Maxwell interactions (corresponding to
  $\g=0$). Our approach relies on a perturbation argument from the
  corresponding Maxwell model and a careful study of the
  associated linearized operator  recently derived in the companion paper \cite{maxwel}.  The results can be seen as a
  first step towards a complete proof, in the one-dimensional setting, of
  a conjecture in \cite{ernst} regarding the determination of the
  long-time behaviour of solutions to inelastic Boltzmann equation, at least, in a regime of moderately hard potentials.
\end{abstract}

%\tableofcontents

\section{Introduction}
\label{sec:intro}
 
We treat the one-dimensional inelastic Boltzmann equation for moderately hard
potentials, proving stability, regularity, and uniqueness of $L^{1}$ equilibrium
self-similar profiles.

\subsection{One-dimensional inelastic Boltzmann equation} Inelastic
models for granular matter are ubiquitous in nature and rapid granular
flows are usually described by a suitable modification of the
Boltzmann equation, see \cite{garzo,vill}. Inelastic interactions are
characterised, at the microscopic level, by the continuous dissipation
of the kinetic energy for the system. Typically, in the usual $3D$
physical situation, two particles with velocities
$(v,v_{\star}) \in \R^{3}\times \R^{3}$ interact and, due to inelastic
collision, their respective velocities $v'$ and $v_{\star}'$ after
collision are such that momentum is conserved
$$v+v_{\star}=v'+v_{\star}'$$
but kinetic energy is dissipated at the moment of impact:
$$|v'|^{2}+|v_{\star}'|^{2}  \leq |v|^{2}+|v_{\star}|^{2}.$$
Often the dissipation of kinetic energy is measured in terms of a
single parameter, usually called the \emph{restitution coefficient},
which is the ratio between the magnitude of the normal component of
the relative velocity after and before collision. This coefficient
$e \in [0,1]$ may depend on the relative velocity and encode all the
physical features. It holds then
$$\langle v'-v'_{\star},n\rangle=-e\,\langle v-v_{\star},n\rangle\,$$
where $n \in \mathbb{S}^{2}$ stands for the unit vector that points
from the $v$-particle center to the $v_{\star}$-particle center at the
moment of impact. Here above, $\langle\cdot,\cdot\rangle$ denotes the
Euclidean inner product in $\R^{3}$.

For one-dimensional interactions, we will rather denote by $x,y$ the
velocities before collision and $x',y'$ those after collision and the
collision mechanism is then described simply as
$$x'=\bar{a}x+(1-\bar{a})y, \qquad y'=(1-\bar{a})x+\bar{a}y, \qquad \bar{a} \in [\tfrac{1}{2},1]\,,$$
where  the parameter $\bar{a}$ describes the intensity of inelasticity.   One can check that $x'+y'=x+y$ whereas
\begin{equation}
  \label{eq:enxx}
  |x'|^{2}+|y'|^{2}-|x|^{2}-|y|^{2} = - 2\bar{a}\bar{b} |x-y|^{2} \leq 0,
  \qquad \bar{b}=1-\bar{a},
\end{equation}
where we used that $\bar{a}^{2}+\bar{b}^{2}-1=-2\bar{a}\bar{b}$ for $\bar{b}=1-\bar{a}$. In this case, the
inelastic Boltzmann equation is given by the following, as proposed
in \cite{BK}:
\begin{equation}\label{Intro-e1}
\partial_{s}f(s,x) = \Q_{\g}(f,f)(s,x),\qquad\qquad (s,x)\in (0,\infty)\times\R\,,
\end{equation}
with given initial condition $f(0,x)=f_0(x)\geq0$.  The interaction operator is defined as
\begin{equation}\label{Intro-e1.5}
\Q_{\g}(f,f)(x) = \int_{\R}f(x-\bar{a}y)f(x+(1-\bar{a})y)|y|^{\gamma}\d y - f(x)\int_{\R}f(x-y)|y|^{\gamma}\d y
\end{equation}
for fixed $\gamma\geq0$ and $\bar{a}\in[\frac{1}{2},1]$.  Notice that the model \eqref{Intro-e1} conserves mass and momentum
$$\ds\int_{\R}\Q_\g(f,f)(x)\dx =\int_{\R}\Q_{\g}(f,f)(x)\,x\dx= 0\,, 
$$
but dissipates energy since
\begin{multline}\label{eq:ener}\int_{\R}\Q_{\g}(f,f)(x)|x|^{2}\d x\\
=\frac{1}{2}\int_{\R^{2}}f(u)f(v)|u'-v'|^{\g}\left[|u'|^{2}+|v'|^{2}-|u|^{2}-|v|^{2}\right]\d u\d v		\\
=-\bar{a}\bar{b}\int_{\R \times \R}f(u)f(v)|u-v|^{\g+2}\d u\d v\leq0, \quad \quad \bar{b}=1-\bar{a},
\end{multline}
where we used a change of variable $u=x-\bar{a}y,v=x+\bar{b}y$ and a symmetry argument to get the first identity while we used \eqref{eq:enxx} to establish the second one. This implies that, for any nonnegative initial datum $f_{0}$ and  any solution $f(s,z)$ to \eqref{Intro-e1}, it holds
$$\dfrac{\d}{\d s}\int_{\R}f(s,z)\left(\begin{array}{cc}1 \\ z \end{array}\right)\d z=\left(\begin{array}{cc}0 \\ 0\end{array}\right)\,,$$
while
\begin{equation}\label{eq:dissiab}
\dfrac{\d}{\d s}\int_{\R}f(s,z)z^{2}\d z=-\bar{a}\bar{b}\int_{\R\times \R}f(s,u)f(s,v)|u-v|^{\g+2}\d u\d v \leq 0.
\end{equation}
Throughout the document, we shall consider without loss of generality a renormalized initial condition $f_0$ satisfying 
$$\int_{\R} f_0(x)\dx=1 \qquad \mbox{ and }\qquad \int_{\R}f_0(x) x\dx =0.$$
Moreover, one sees from \eqref{eq:dissiab} that the single parameter $\bar{a}$ (through the product $\bar{a}\bar{b}=\bar{a}(1-\bar{a})$) measures the strength of energy dissipation. The case $\bar{a}=1$ represents a purely elastic interaction which, in one dimension, is described by no interaction at all, or simply $\Q_\gamma(f,f)=0$.  The other case $\bar{a}=\frac{1}{2}$ is the case of extreme inelasticity or the sticky particle case; that is, after interaction the particles remain attached yet considered distinct so that no global mass is lost. From now, throughout this manuscript, we consider this latter case
$$\bar{a}=\frac{1}{2}$$ 
but wish to point out that the general case $\bar{a}\in\left(\frac{1}{2},1\right)$ can be treated in a similar fashion.\\

The above dissipation of kinetic energy (together with mass and
momentum conservation) leads to a natural equilibrium given by the
distribution that accumulates all the initial mass, say $m_0$, at the
initial system's bulk momentum $z_0$:
$$\Q_{\g}(F,F)= 0 \implies \exists \ m_{0} \geq0,\, z_{0} \in \R \text{ such that } F=m_0\,\delta_{z_0}.$$ 
Such a degenerate solution is of course expected to attract all solutions to \eqref{Intro-e1} but, as for the multi-dimensional model, one expects that, before reaching the degenerate state, solutions resemble some universal profile as an intermediate asymptotic. 

More precisely, a more detailed description of solutions can be derived introducing a rescaling of the form 
\begin{equation*}
V_{\g}(s)\,g(t(s),x) = f(s,z)\,,\qquad \qquad x=V_{\g}(s)\,z\,,
\end{equation*}      
where the rescaling functions are given by 
\begin{equation*}
V_{\g}(s) =\begin{cases}
(1+c \,\gamma\, s)^{\frac1\gamma}\,\qquad &\text{ if } \g \in (0,1),\\
e^{cs}\,\qquad &\text{ if } \g=0\,,\end{cases} \quad \text{and} \quad  t(s)=t_{\g}(s)=\begin{cases} \frac{1}{c\gamma}\log(1+c\,\gamma\, s)\,, \qquad &\text{ if } \g \in (0,1),\\
s \,,&\text{ if } \g=0\,.
\end{cases}
\end{equation*}
We refer to \cite{bobcerc1} for a study of the Maxwell model $(\g=0)$
and \cite{ABCL} for the hard potential model $(\g >0)$. For
$\gamma >0$, this rescaling is useful for any $c > 0$ and we are free
to choose this parameter as we see fit; while for $\gamma = 0$, the
only useful choice is $c=\frac{1}{4}$ (assuming unitary mass). We will come back to this later. Straightforward computations show then that, if $f(s,z)$ is a
solution to \eqref{Intro-e1}, the function $g(t,x)$ satisfies
\begin{equation}\label{Intro-e2}
\partial_{t}g(t,x) + c\,\partial_{x}\big(x\,g(t,x)\big) = \Q_{\gamma}(g,g)(t,x),\qquad\qquad (t,x)\in (0,\infty)\times\R\,,
\end{equation} 
with $g(0,x) = f_0(x)$ so that
\begin{equation}\label{eq:conse}
  \int_{\R} g(t,x)\d x=\int_{\R}f_{0}(x)\d x=1, \qquad \int_{\R}xg(t,x)\d x=\int_{\R}xf_{0}(x)\d x=0, \quad \forall\, t \geq0
  \end{equation}
due to the conservation of mass and momentum induced by both the drift term and the collision operator $\Q_{\g}$. 
Since, formally,
$$ \int_{\R}\partial_{x}\big(x\,g(t,x)\big) |x|^{2}\d x = -2\int_{\R}g(t,x)\,|x|^{2}\d x\,,$$    
one can interpret the rescaling as an artificial way to add energy into the system, the bigger the $c>0$ the more energy per time unit is added.  Thus, the rescaling has the same effect of adding a background linear ``anti-friction'' with constant $c>0$.  However, unless in the special case $\g=0$, evolution of the kinetic energy along solutions to \eqref{Intro-e2} is not given in closed form. Namely, if 
$$M_{2}(g(t))=\int_{\R}g(t,x)x^{2}\dx$$
one sees from \eqref{Intro-e2} and \eqref{eq:ener} that
\begin{equation}\label{eq:evolenerg}
\dfrac{\d}{\d t}M_{2}(g(t))-2cM_{2}(g(t))=-\frac{1}{4}\int_{\R\times\R}g(t,x)g(t,y)|x-y|^{\g+2}\d x\d y\end{equation}
so that the evolution of the second moment of $g(t)$ depends on the evolution of moments of order $\g+2$. The situation is very different in the case $\g=0$ and this basic observation will play a crucial role in our analysis.\medskip

It is important to observe that problems \eqref{Intro-e1} and
\eqref{Intro-e2} are related by a simple rescaling, so that knowledge
of properties of one of them is transferable to the other. Equation
\eqref{Intro-e2} is referred to as the \emph{self-similar
  equation}. For $\gamma > 0$ and any $c > 0$, it has at least one
non-trivial equilibrium with positive energy \citep{ABCL}, satisfying the
equation
\begin{equation}\label{Intro-e3}
c\,\partial_{x}\big(x\,\bm{G}(x)\big) = \Q_{\gamma}(\bm{G},\bm{G})(x)\,.
\end{equation} 
For $\gamma = 0$, there is a non-trivial equilibrium with positive
energy only if $c = \frac{1}{4}$. The equilibria are known as
\emph{self-similar profiles}. Of course, $\bm{G}$ depends on the
choice of $c>0$; however, they are all related by a simple
rescaling. Indeed, if $\bm{G}$ solves \eqref{Intro-e3} then 
\begin{equation}\label{eq:sca}
\bm{G}_{\lambda,\mu}(x)=\lambda \bm{G}(\mu x) \;\;\text{ solves \eqref{Intro-e3} with } \;\; c\,\lambda\, \mu^{-\g-1}\;\; \text{ instead of }  \;\; c . 
\end{equation}
Moreover, the fact that $\bm{G}$ is a regular (smooth)
function is helpful for the technical analysis, for example to have a
standard linearisation referent.  Indeed, in this document we prove
regularity properties for $\bm{G}$ and answer the
uniqueness question for the problem \eqref{Intro-e3}, at least in the
context of moderately hard potentials, that is, our results will be
valid for relatively small positive $\gamma$.  \medskip

The question of uniqueness for self-similar profiles of the model
\eqref{Intro-e3} is notoriously difficult for $\gamma > 0$. Since the
model conserves mass and momentum, a uniqueness result has to take
into account this fact. In other words, steady states should be unique
in a space with fixed mass and momentum. In the case of Maxwell
interactions ($\gamma = 0$ and $c=\frac{1}{4}$), energy is additionally
conserved, and it is known that self similar profiles are unique when
mass, momentum and energy are fixed \citep{bobcerc1}. This case is
less technical, and somehow critical, since the self-similar rescaling
is uniquely determined (by the initial mass and energy) as opposed to
the case $\gamma>0$ where one can choose any $c>0$ to perform the
rescaling.  For the Maxwell case the rescaling leads to the
conservation of energy which is an important help in the analysis,
together with a computable spectral gap for the linearized equation.
We refer to \cite{MR2355628} for a good account of the theory of the
Maxwell model in one and multiple dimensions.  \medskip

There is another type of uniqueness result.  In the context of
$3D$-dissipative particles it is possible to define a weakly inelastic
regime.  A big difference between the one-dimensional problem and the
three-dimensional problem is that in the latter the elastic limit of
the model is the classical Boltzmann equation whereas in the
one-dimensional problem the elastic limit $a=1$ is simply $\partial_t
f =0$.  This is the reason one can study weakly inelastic systems as a
perturbation of the classical Boltzmann equation in several dimensions
with powerful tools such as entropy-entropy dissipation methods
leading to a uniqueness result in this context, see
\cite{MM09,CMS,AL}.  And yet, the same strategy completely fails in
the $1D$-dissipative model where such tools are not available.
\medskip

Our analysis for small positive $\gamma$ will be also perturbative
taking as reference the one-dimensional Maxwell sticky particle model;
that is, our result covers the most extreme case of inelasticity
providing a strong indication that the steady inelastic self-similar
profiles should be unique in full generality, for all collision
kernels and degrees of inelasticity.  This perturbation is highly
singular in two respects: first, the Maxwell model conserves energy
{in self-similar variables} which is not the case for $\gamma>0$.
This is a major difficulty since the spectral gap of the linearized
Maxwell model depends crucially on this conservation law.  Second, the
tail of the self-similar profiles are completely different, for
Maxwell models the profile enjoys some few statistical moments only,
whereas for hard potentials the profile has exponential tails.
Fortunately, steady states will enjoy regularity for all
$\gamma >0$, a property that will be also proved in this paper.

 \subsection{The problem at stake}
 
The main concern of the present paper is, as said, the uniqueness of the steady solutions $\Gg$ to the equation
\begin{equation}
\label{eq:steadyg}
c\partial_{x}(x\Gg ) = \Q_{\g}(\Gg ,\Gg )\,,
\end{equation}
with unit mass and zero momentum where, for $\g \in [0,1]$ ,
$\Q_{\g}(f,g)$ reads in its weak form as 
\begin{equation}\label{eq:weakgamma}
\int_{\R}\Q_{\g}(f,g)(x)\varphi(x)\d x=\frac{1}{2}\int_{\R^{2}}f(x)g(y)\Delta\varphi(x,y)|x-y|^{\g}\d x\d y\end{equation}
with
$$\Delta \varphi(x,y)=2\varphi\left(\frac{x+y}{2}\right)-\varphi(x)-\varphi(y)$$
for any suitable test function $\varphi$.  {We can split $\Q_{\g}(f,g)$ into positive and negative parts
$$\Q_{\g}(f,g)=\Q^{+}_{\g}(f,g)-\Q^{-}_{\g}(f,g)$$
 where, in weak form,
$$\int_{\R}\Q_{\g}^{+}(f,g)(x)\varphi(x)\d x=\int_{\R^{2}}f(x)g(y)\varphi\left(\frac{x+y}{2}\right)|x-y|^{\g}\d x\d y$$
and
$$\int_{\R}\Q^{-}_{\g}(f,g)\varphi(x)\d x=\frac{1}{2}\int_{\R^{2}}f(x)g(y)\left(\varphi(x)+\varphi(y)\right)|x-y|^{\g}\d x\d y.$$}
The existence of a suitable solution $\Gg$ to \eqref{eq:steadyg} with
finite moments up to third order has been obtained in \cite{ABCL}. {We
  will always assume here that $\Gg \in L^1_3(\R)$ is nonnegative and
  satisfies
\begin{equation}\label{eq:mom}
\int_{\R}\Gg(x)\dx=1, \qquad \int_{\R}\Gg(x)x \dx=0, \qquad \g\in [0,1].\end{equation}
Notice that the energy 
$$\int_{\R}x^{2}\Gg(x)\d x=M_{2}(\Gg)$$
is not known \emph{a priori} since the non-conservation of kinetic
energy precludes any simple selection mechanism to determine it at
equilibrium.  %{\color{magenta}In particular, it is not even clear whether $\inf_{\g\in (0,1)}M_{2}(\Gg)>0$ or vanishes. ST: I suggest to remove this since it is proven now.}

A crucial point of our analysis lies in the fact that this problem
has a well-understood answer in the Maxwell case in which
$\g=0$. Indeed, in such a case, many computations are explicit and,
for instance, the evolution of kinetic energy for equation
\eqref{Intro-e2} is given in closed form as, according to
\eqref{eq:evolenerg}
\begin{equation*}
\dfrac{\d}{\d t}M_{2}(g(t))-2cM_{2}(g(t))=-\frac{1}{4}\int_{\R\times\R}g(t,x)g(t,y)|x-y|^{2}\d x\d y=-\frac{1}{2}M_{2}(g(t))
\end{equation*}
where we used \eqref{eq:conse} to compute the contribution of the collision operator. In particular, for $\g=0$, energy is conserved if and only if $c=\frac{1}{4}$ and, in such a case, we can prescribe the kinetic energy of the steady state $\bm{G}_{0}.$ \medskip

For this reason, in the sequel we will always assume that
$$c=\frac{1}{4}.$$

\medskip

Another important property of the Maxwell molecules case is that, due to explicit computations in Fourier variables, solutions to \eqref{eq:steadyg} are actually explicit in this case.  More precisely, one has the following theorem.
\begin{theo}[\cite{bobcerc1}]\label{theo:bob} Let  {$\P_{2}(\R)$ denote the set of probability  measures} on $\R$ with finite second order moment. Any {$\mu \in \P_{2}(\R)$} such that 
\begin{equation}\label{eq:weakmes}
-\frac{1}{4}\int_{\R}x\varphi'(x)\mu(\d x)=\frac{1}{2}\int_{\R\times\R}\Delta \varphi(x,y)\mu(\!\dx)\mu(\dy) \qquad \forall \varphi \in \mathcal{C}_{b}^{1}(\R)\end{equation}
and satisfying 
$$\int_{\R}\mu(\!\dx)=1,\qquad \int_{\R}xg(x)\mu(\!\dx)=0, \qquad \int_{\R}x^{2}\mu(\d x)=\frac{1}{\lambda^{2}} >0$$
is of the form
\begin{equation}\label{eq:HH}
\mu(\dx)=H_{\lambda}(x)\d x=\lambda\bm{H}(\lambda x)\d x \qquad \text{ with } \quad
\bm{H}(x) = \frac{2}{\pi (1+x^2)^2}\,\qquad x \in \R.
\end{equation}
\end{theo}
Here above of course \eqref{eq:weakmes} is a measure-valued version of the steady equation \eqref{eq:steadyg} and in particular one sees that $\bm{H}$ is the unique solution to 
$$\frac{1}{4}\partial_{x}\left(x\bm{H}(x)\right)=\Q_{0}(\bm{H},\bm{H})(x)$$
with unit mass and energy and zero momentum. The existence and uniqueness has been obtained, through Fourier transform methods, in \cite{bobcerc1} and extended to measure solutions in \cite{MR2355628}. 

We introduce the following set of equilibrium solutions
$$\mathscr{E}_{\g}=\left\{{\Gg \in \bigcap_{s<3
}L^{1}(\w_s)}\;;\;\Gg \text{ satisfying \eqref{eq:steadyg} and \eqref{eq:mom}}\,\right\}$$
for any $\g \in [0,1)$ where $L^{1}(\w_s)$ is defined in Section \ref{sec:notation}. The above theorem ensures that elements of $\mathscr{E}_{0}$ are entirely described by their kinetic energy, i.e., given $E >0$, 
$$\left\{g \in \mathscr{E}_{0} \text{ and } M_{2}(g)=E\right\} \text{ reduces to a singleton}.$$The main objective of the present contribution is to prove that, for moderately hard potentials, the situation is similar and more precisely, our main result can be summarized as follows.
\begin{theo}\label{theo:mainUnique} There exists some explicit $\g^{\dagger} \in (0,1)$ such that, if $\g \in (0,\g^{\dagger})$ then $\mathscr{E}_{\g}$ reduces to a singleton.
\end{theo} 
Notice here the contrast between the case $\g >0$ where
$\mathscr{E}_{\g}$ is a singleton whereas, for $\g=0$,
$\mathscr{E}_{0}$ is an infinite one-dimensional set parametrised by the energy
of the steady solution. As we will see, it happens that, in the limit
$\g \to 0$, the steady equation \eqref{eq:steadyg} (with
$c=\frac{1}{4}$) \emph{selects} the energy.

Before describing in details the main steps behind the proof of Theorem \ref{theo:mainUnique}, we need to introduce the notations that will be used in all the sequel.
\subsection{Notations}
\label{sec:notation}
For  any weight function $\varpi : \R \to \R^{+}$, we define, for any $p \geq 1$,
$$L^{p}(\varpi) :=\Big\{f : \R \to \R\;;\,\|f\|_{L^{p}(\varpi)}^{p}:=\int_{\R}\big|f(x)\big|^{p}\,\varpi(x)^p\,\d x  < \infty\Big\}\,.$$
A frequent choice will be the weight function
\begin{equation}\label{eq:weight}
\w_{k}(x)=\left(1+|x|\right)^{k}, \qquad k\in \R, \qquad x \in \R. 
\end{equation}
For general $s\ge 0$, the Sobolev space $H^s(\R)$ is defined thanks to the Fourier transform. 
$$H^s(\R):=\Big\{f \in L^{2}(\R)\;;\; \|f\|_{H^{s}}^{2}:=\int_{\R} (1+|\xi|^2)^s |\hat{f}(\xi)|^2\d \xi<\infty\Big\}\,, $$
where 
$$ \hat{f}(\xi) =\int_{\R} f(x) e^{-ix\xi}\d x, \qquad \xi \in\R. $$
We shall also use an important shorthand for the moments of order $s \in \R$ of  measurable mapping $f : \R\to {\R}$, 
$$M_{s}(f) :=\int_{\R}f(x)\,  |x|^s \d x. $$
Finally, for any $k \geq0$, we denote with ${\mathcal M}_{k}(\R)$  the set of real Borel measures $\mu$ on $\R$ with finite total variation of order $k$ that are satisfying $\ds\int_{\R}  \w_{k}(x) \,|\mu|(\dx)<\infty$.
\subsection{Strategy and main intermediate
  results}\label{Sec:strategy} The idea to prove our main result
Theorem \ref{theo:mainUnique} is to adopt a \emph{perturbative
  approach} and to fully exploit the knowledge of the limiting case
$\g=0$.  This explains in particular why our result is valid for
\emph{moderately hard potentials} $\g \gtrsim 0.$ More precisely,
inspired by similar ideas developed in the context of the Smoluchowski
equation (see \cite{CanizoThrom} for a recent account and a source of
inspiration for the present work), the first step in our proof is to
show that, in some weak sense to be determined,
\begin{equation}\label{eq:conv}
\lim_{\g \to 0}\Gg=\bm{G}_{0}, \qquad \bm{G}_{0} \in \mathscr{E}_{0}
\end{equation}
where $\bm{G}_{\gamma}$ is any steady solution in $\mathscr{E}_{\gamma}$ and $\bm{G}_{0}$ is a steady solution in the Maxwell case, i.e. a solution to \eqref{eq:weakmes} and, consequently, $\bm{G}_0\in \{H_{\lambda}, \lambda >0\}$.  
 
Such a limiting process is quite singular. For instance, it can be shown, see \cite[Proposition 1.3]{long}, that for any $\g \in (0,1)$ and $\Gg \in \mathscr{E}_{\g}$, one has
$$M_{k}(\Gg):=\int_{\R}\Gg(x)|x|^{k}\d x < \infty \qquad \text{ for any } k \geq0$$
whereas, for $\g=0$,
$$M_{k}(\bm{G}_{0}) < \infty \iff k \in (-1,3).$$
Now, since solutions to \eqref{eq:weakmes} exist with arbitrary energy, a first important step is to derive the correct limiting temperature
$$\lim_{\g\to0}M_{2}(\Gg)=E_{0}=M_{2}(\bm{G}_{0}) >0,$$
since that single parameter, thanks to Theorem \ref{theo:bob}, characterizes $\bm{G}_{0}$.  In particular, to rule out the degenerate case in which $\bm{G}_{0}$ is a Dirac mass, a crucial point is to prove that necessarily $E_{0} >0$ which is done thanks to the following \emph{a priori} estimate on $M_{2}(\Gg)$.
\begin{prp}\label{theo:energy0} For any  {$k_{0} \in (0,\frac{1}{2})$,} there exists $\bm{\alpha}_{2} >0$ such that {for any $\g\in(0,k_{0})$, }
$$\bm{\alpha}_{2} \leq M_{2}(\Gg) \leq \frac{1}{2} \qquad \forall \, \Gg \in \mathscr{E}_{\g}.$$
\end{prp}
Once this is established, and admitting for the moment that \eqref{eq:conv} holds in some suitable sense, we can deduce that 
$$\bm{G}_{0}(x)=\lambda_{0}\bm{H}(\lambda_{0} x), \qquad \lambda_{0} >0\,.$$
This allows to exploit the integrability of $\bm{G}_{0}$ to provide additional $L^{2}$-bounds  -- and thus compactness -- in the limiting process $\g \to0$.  The first step in our proof can be summarized in the following theorem which characterizes the limit temperature for $\g\to 0$. Its proof will be given in
Section~\ref{Sec:limit:temp}. We recall that we are always setting
$c=\frac{1}{4}$ in Eq. \eqref{eq:steadyg}.
\begin{theo}\phantomsection\label{theo:Unique}
For any $\varphi \in \mathcal{C}_{b}(\R)$ and choice of steady states $\Gg \in \mathscr{E}_{\g}$ (for $\gamma > 0$) one has that
\begin{equation}\label{eq:convstar}
\lim_{\g\to0^{+}}\int_{\R}\Gg(x)\varphi(x)\dx=\int_{\R}\bm{G}_{0}(x)\varphi(x)\dx,\end{equation}
where, for $x\in\R$,  
$$\bm{G}_{0}(x)= {H_{\lambda_0}(x)=}\frac{2\lambda_{0}}{\pi\left(1+\lambda_{0}^{2}x^{2}\right)^{2}}, \qquad \mbox{ with } \qquad  {
 \lambda_{0}:=\exp\left(\frac12\mathscr{I}_{0}(\bm{H},\bm{H})\right),}$$ and 
\begin{equation*}
 {\mathscr{I}_{0}(\bm{H},\bm{H})}=\int_{\R}\int_{\R}\bm{H}(x)\bm{H}(y)|x-y|^{2}\log|x-y|\dx\dy >0.
\end{equation*}
 Moreover, for any  {$\delta \in \left(0,\frac12\right)$} there exist $C_{0} >0$ and $\g_{\star}  \in (0,1)$, both depending on $\delta$, such that
\begin{equation}\label{eq:estimGg}
\|\Gg\|_{L^2} \leq C_{0}, \quad \quad M_{k}(\Gg) \leq C_{0},  \qquad \forall \, \g \in [0,\g_{\star} ), \qquad k \in (0,3-\delta).\end{equation}
\end{theo}
\begin{rmk} {In fact it has been shown in \cite[Lemma~4.9]{maxwel} that}  {$\mathscr{I}_{0}(\bm{H},\bm{H})=2\log2+1$ } and thus $\lambda_0=2\sqrt{e}$. The convergence in the aforementioned theorem is a weak-$\star$ convergence of the probability measures $\{\Gg(x)\d x\}_{\g}$ towards $\bm{G}_{0}(x)\d x$.  We point out that the bound \eqref{eq:estimGg} actually implies that the family $\{\Gg\}_{\g\in (0,\g_{\star})}$ is \emph{weakly compact} in $L^{1}(\R)$ which improves the convergence and, in particular, allows to assume $\varphi \in L^{\infty}(\R)$ in \eqref{eq:convstar}. Since the bounds \eqref{eq:estimGg} are actually deduced from the weak-$\star$ convergence, we adopted this presentation of the result.\end{rmk}
Notice that, as documented in \cite{ABCL}, the derivation of $L^{2}$
bounds for solutions to the $1D$-Boltzmann equation is not an easy
task. This comes from the lack of regularizing effects induced by the
collision operator $\Q_{\g}$ in dimension $d=1$. Recall indeed that a
celebrated result in \cite{lions} asserts that, for very smooth
collision kernels, the Boltzmann collision operator (for elastic
interactions and in dimension $d$) maps, roughly speaking,
$L^{2}(\R^{d}) \times L^{2}(\R^{d})$ in the Sobolev space
$H^{\frac{d-1}{2}}(\R^{d})$, i.e. the collision operator induces a
gain of $\frac{d-1}{2}$ (fractional) derivative. One sees therefore
that no regularisation effect is expected in dimension $d=1$ whereas
gain of regularity is the fundamental tool for the derivation of
$L^{p}$-estimates for solutions to the Boltzmann equation (see
\cite{MoVi}). This simple heuristic consideration is also confirmed in
the related case of the Smoluchowski equation for which derivation of
suitable $L^{p}$-estimates $p >1$ is a notoriously difficult problem
(see \cite{bana, CanizoThrom}).

In the present paper, the derivation of $L^{2}$-bounds (uniformly with
respect to $\g > 0$) is deduced from quite technical arguments,
specific to the study of equilibrium solutions, and crucially exploits
the convergence of $\Gg$ towards $\bm{G}_{0}$ together with the fact
that $\bm{G}_{0}$ is explicit. In particular, our argument
does not seem to work for time dependent solutions to \eqref{Intro-e2}.

A second step in our proof is then to be able to quantify the above convergence of $\Gg$ towards $\bm{G}_{0}$ and, in particular, to exploit the fact that $\bm{G}_{0}$ is a \emph{stable equilibrium solution} to \eqref{eq:steadyg} for $\g=0.$ This is made possible thanks to a good knowledge on the long time behaviour of the solutions to the evolution problem \eqref{Intro-e2} corresponding to Maxwell molecules $\g=0$:
\begin{equation}\label{introMax}
\partial_{t}g(t,x) + \frac{1}{4}\partial_{x}\left(xg(t,x)\right)=\Q_{0}(g,g)(t,x), \qquad x \in \R, t \geq 0\,.\end{equation}
New results about the exponential convergence towards equilibrium $\bm{G}_{0}$ as well as propagation of regularity for solutions to \eqref{introMax} have been recently obtained in  the companion paper \cite{maxwel} (which revisits and generalises results obtained in \cite{MR2355628} and \cite{FPTT}).  We refer to those contributions for more details about this equation. A toolbox needed specifically for the problem at stake in the present contribution is recalled in Section \ref{sec:tool}. 

Such results and a new regularity bound for the self-similar profile $\Gg$ allow to derive  the following stability estimate for self-similar profiles which is also a main ingredient in the proof of Theorem~\ref{theo:mainUnique} and whose proof will be given in Section~\ref{sec:upgrade}.
\begin{theo}[\textit{\textbf{Stability of profiles}}]\phantomsection\label{theo:sta} Let $2<a<3$. There exist $\g_\star\in(0,1)$ and an \emph{explicit} function $\overline{\eta}=\overline{\eta}(\gamma)$ depending on $a$, with $\lim_{\g\to0^+}\overline{\eta}(\gamma)= 0$, such that, for any $\g\in(0,\g_\star)$ and any $\Gg\in\mathscr{E}_{\g}$,   
$$ \|\Gg-\bm{G}_{0}\|_{L^1(\w_{a})} \le \overline{\eta}( \gamma).$$
\end{theo} 
We point out here that some suitable smoothness estimates for $\Gg$ in Sobolev spaces (uniformly with respect to $\g$) are required for the proof of Theorem \ref{theo:sta} (see Lemma \ref{lem:stabil}).  As explained already, the lack of regularization effect for the operator $\Q_{\g}$ produces technical obstacles in the proof of such estimates for $\Gg$.

A final important tool for the proof of Theorem \ref{theo:mainUnique} is the \emph{quantitative stability} of the steady state $\bm{G}_{0}$ of \eqref{introMax} in the space $L^{1}(\w_{a})$ which is a consequence of a careful spectral analysis of the linearized operator $\mathscr{L}_{0}$ associated to Maxwell molecules $\g=0$ performed in the companion paper \cite{maxwel} (see Definition \ref{def:spaces} for details). Roughly speaking, the idea is to apply some suitable quantitative estimate of the spectral gap $\mu_{0} >0$ of the linearized operator $\mathscr{L}_{0}$ to the difference $g_{\g}=\Gg^{1}-\Gg^{2}$ of two elements of $\mathscr{E}_{\g}$. Namely, for any $a\in(2,3)$, there is $\mu_{0} >0$ (see Prop. \ref{restrict0} for a precise statement) such that
\begin{equation}\label{eq:invert0-in} 
 \|\mathscr{L}_{0}h\|_{L^{1}(\w_{a})}\ge \mu_{0} \|h\|_{L^{1}(\w_{a})}\end{equation}
for all $h \in \D(\mathscr{L}_{0}) \cap L^{1}(\w_{a})$ with 
\begin{equation}\label{eq:0mass}\int_{\R}h(x)\d x=\int_{\R}h(x)x\d x=\int_{\R}h(x)x^{2}\d x=0.\end{equation}
In the simplified situation in which two profiles $\Gg^{1}, \Gg^{2} \in \mathscr{E}_{\g}$ share the same energy, the difference $g_{\g}=\Gg^{1}-\Gg^{2}$ satisfies \eqref{eq:0mass}.  Furthermore, the action of $\mathscr{L}_{0}$ on $\mathscr{E}_{\g}$ is such that there exists a mapping $\tilde{\eta}\;:\;[0,1] \to \R^{+}$ with 
$\lim_{\g\to0^{+}}\tilde{\eta}(\g)=0$ and such that
\begin{equation*}\label{eq:L0dif-intro}
\|\mathscr{L}_{0}\left(\Gg^{1}-\Gg^{2}\right)\|_{L^{1}(\w_{a})} \leq \tilde{\eta}(\g)\left\|\Gg^{1}-\Gg^{2}\right\|_{L^{1}(\w_{a})}, \qquad \g >0\end{equation*}
which combined with \eqref{eq:invert0-in} leads to 
$$\mu_{0} \|g_{\g}\|_{L^{1}(\w_{a})} \leq \|\mathscr{L}_{0}g_{\g}\|_{L^{1}(\w_{a})} \leq \tilde{\eta}(\g)\left\|g_{\g}\right\|_{L^{1}(\w_{a})}\,.$$
Because $\lim_{{\g\to0}}\tilde{\eta}(\g)=0$, one observes that it is possible to choose $\g^{\dagger} \in (0,1)$ such that
$$\|g_{\g}\|_{L^{1}(\w_{a})} < \|g_{\g}\|_{L^{1}(\w_{a})} \qquad \forall \g \in (0,\g^{\dagger}).$$
This shows that $g_{\g}=0$ for all $\g \in (0,\g^{\dagger})$ and gives a simplified version of Theorem \ref{theo:mainUnique} in the special case in which $\Gg^{1}$ and $\Gg^{2}$ share the same energy. To prove the uniqueness result without any restriction on the energy, we need therefore, in some rough sense, to be able to control the fluctuation of kinetic energy introducing a kind of selection principle which allows to compensate the discrepancy of energies to apply a variant of \eqref{eq:invert0-in}. This is done in Section \ref{sec:uniqueness} to which we refer for technical details regarding such a procedure.

\subsection{Main features of our contribution} 
A first important novelty and main interest of the present contribution is to present a uniqueness result for self-similar profiles associated to an inelastic Boltzmann equation for hard potentials in a regime of \textit{\textbf{strong inelasticity}}.  Previously, the only uniqueness result available in the literature is given, for the $3D$ case, in \cite{MM} in a \emph{weakly inelastic regime} corresponding to a restitution coefficient $e \simeq 1$. Our analysis here is the first one dealing with strong inelastic interactions (in fact, the strongest corresponding to sticky particles). 

Second, our analysis provides a first step towards the analogous of the so-called scaling
hypothesis which, in the study of Smoluchowski's equation, asserts
that self-similar profiles are unique and attract all solutions to the
associated evolution equation, refer to \cite{CanizoThrom} for a first
proof of the scaling hypothesis for non-explicitly solvable
kernels. In the current contribution we prove, as in
\cite{CanizoThrom}, that for a singular perturbation of the explicitly
solvable case of Maxwell molecules (i.e. for $\g \simeq 0$) the
self-similar profile $\Gg$ is unique. Some additional work should be
carried on to prove that such a unique solution attracts all
$L^1$ solutions to \eqref{Intro-e2} with some explicit, perhaps exponential, rate.
Indeed, exploiting the fact that convergence in Fourier norm $\vertiii{\cdot}_{k}$ is
 exponential in the limit case $\g=0$ (see Theorem
\ref{k-norm-cvgce}) and combining it with a careful spectral analysis
of the linearization of $\Q_{\g}$ around the self-similar profile
$\Gg$ should provide an insight on this important question and
pave the way to a full mathematical justification of a conjecture in
\cite{ernst} about the long-time behaviour of granular gases with strong inelasticity
with moderate hard potentials.

 \subsection{Organisation of the paper} The rest of the paper is organised as follows. Sections \ref{sec:moment}  and \ref{sec:apost} derive the main a posteriori estimates on the self-similar profile $\Gg \in \mathscr{E}_{\g}$, focusing mainly on estimates which are uniform with respect to the parameter $\g \gtrsim 0$. In particular, Section \ref{sec:moment} is devoted to uniform moment estimates with in particular a proof of Proposition \ref{theo:energy0}, whereas Section \ref{sec:apost} is devoted to $L^{2}$ integrability estimates as well as a proof of Theorem \ref{theo:Unique}. The two main results Theorem \ref{theo:sta} and Theorem \ref{theo:mainUnique} are proven in Section \ref{sec:sec3} in which we first recall the results regarding the Maxwell equation \eqref{introMax} obtained in \cite{maxwel} that are needed for the present contribution. The paper ends with a technical appendix in which several results of independent interests are derived.  In particular,  Appendix \ref{app:QgQ0} is devoted to some functional estimates of the collision operator $\Q_{\g}$ and its linearized counterpart. 

\subsection*{Acknowledgments}
R.~Alonso gratefully acknowledges the support from O Conselho Nacional de Desenvolvimento Cient\'ifico e Tecnol\'ogico, Bolsa de Produtividade em Pesquisa - CNPq (303325/2019-4). J.~Ca\~nizo acknowledges support from grant
PID2020-117846GB-I00, the research network RED2018-102650-T, and the
Mar\'ia de Maeztu grant CEX2020-001105-M from the Spanish
government. B.~Lods gratefully acknowledges the financial support from the Italian Ministry of Education, University and Research (MIUR), “Dipartimenti di Eccellenza” grant 2022-2027, as well as the support from the de Castro Statistics Initiative, Collegio Carlo Alberto (Torino). B.L. was also partially supported by PRIN2022 \textit{(project ID: BEMMLZ) Stochastic control and games and the role of information.} The authors would like to acknowledge the support of the
Hausdorff Institute for Mathematics where this work started during
their stay at the 2019 Junior Trimester Program on Kinetic Theory.\\

\emph{Data sharing not applicable to this article as no datasets were generated or analysed during the current study.}

\section{Uniform moments estimates}\label{sec:moment}

As explained in the Introduction, our proof of the uniqueness of
solutions to \eqref{eq:steadyg} is based upon a perturbative approach
around the pivot case $\g=0$ corresponding to Maxwell molecules
interactions. To undertake this perturbative approach we begin
with the control of the energy $M_{2}(\Gg)$.

\subsection{Uniform energy control}
As far as the energy is concerned, one has the following estimate.
\begin{lem}\phantomsection\label{lem:energy}
  The following universal bound holds true: for any $\g \in (0,1)$ and
  any $\Gg \in \mathscr{E}_{\g}$ one has
  \begin{equation*}
    M_2(\Gg)  \leq \frac{1}{2}\,\,.
  \end{equation*}
  As a consequence, there exists some universal constant $\tilde{C} >
  0$ such that, for any $\g \in (0,1)$ and
  any $\Gg \in \mathscr{E}_{\g}$,
  \begin{equation}\label{eq:tildeC}
    M_s(\Gg) \leq \tilde{C}
    \qquad \forall s \in [0,2].
  \end{equation}
\end{lem}

\begin{proof}
  Multiplying the equation \eqref{eq:steadyg} by $|x|^{2}$ and
  integrating in $x$, one obtains formally
  \begin{equation}\label{eq1}
    \frac{1}{4}\int_{\R}\int_{\R}\Gg(x)\Gg(y)|x-y|^{2}\big(|x-y|^{\gamma}-1\big)\dx\dy =0\,.
  \end{equation}Since $r\mapsto r^{\frac{\g+2}{2}}$ is convex, one has 
$$|x-y|^{\g+2} =\left(|x-y|^2\right)^{\frac{\g+2}{2}}\ge 1 +\frac{2+\g}{2}(|x-y|^2-1)=\frac{2+\g}{2}|x-y|^2-\frac{\g}{2}.$$
Consequently, \eqref{eq1} leads to 
\begin{equation*}\begin{split}
0&\ge \frac14 \int_{\R}\int_{\R} \Gg(x)\Gg(y) \left( \frac{2+\g}{2}|x-y|^2-\frac{\g}{2} - |x-y|^2\right) \dx \dy\\
& =  \frac{\g}{8} \int_{\R}\int_{\R} \Gg(x)\Gg(y) \left( |x-y|^2-1\right) \dx \dy\,.
\end{split}\end{equation*}
Since  
$$\int_{\R}\int_{\R}\Gg(x)\Gg(y)|x-y|^{2} \dx \dy=2M_{2}(\Gg), \qquad \int_{\R}\int_{\R}\Gg(x)\Gg(y)\dx\dy=1\,,$$
we deduce the result.
\end{proof} 

We provide also a \emph{lower bound} for $M_{s}(\Gg)$ not uniform with respect to $\g$.
\begin{lem} For any $\g \in [0,1]$ and $\Gg \in \mathscr{E}_{\g}$ one has 
\begin{equation}\label{lower-bound}M_{s}(\Gg) \geq  \left( \frac{1}{{8}\,\cdot 2^{\gamma}} \right)^{\frac{s}{\gamma}}\,\qquad \forall s \geq \g.\end{equation}
\end{lem}
\begin{proof} First, a simple use of Jensen's inequality shows that
$$M_{s}(\Gg) \geq M_{\g}(\Gg)^{\frac{s}{\g}} \qquad \forall \,s \geq \g$$
and therefore it suffices to estimate $M_{\g}(\Gg)$ from below. Given $\Gg \in \mathscr{E}_{\g}$, the first order moment of $\Gg$ satisfies
\begin{equation}\label{moment1}
 -\frac{1}{2}\int_{\R^{2}} \Gg(x)\Gg(y)\bigg( 2\Big|\frac{x + y}{2} \Big| - |x| - |y| \bigg)|x - y|^{\gamma}\d x\d y = \frac{1}{4}M_{1}(\Gg)\,.
\end{equation}
Now, we observe that, for any $\theta \in [-1,1]$, it follows that
$$
\big(|1+\theta| - 1 - |\theta|) = {2}\min\{\theta,0\}\geq -{ 2}|\theta|\, \quad \text{ and } \quad  |1-\theta|^{\gamma}\leq 2^{\gamma}$$
so that
$$\left(2\left|\frac{1+\theta}{2}\right|-1-|\theta|\right)|1-\theta|^{\g} \geq { -2^{\g+1}}|\theta|.$$
If $|x| \geq |y|$, applying this with $\theta=\frac{y}{x}$ 
$$-\bigg( 2\Big| \frac{x + y}{2} \Big| - |x| - |y|\bigg)|x-y|^\gamma\leq { 2^{\gamma+1}}|x|^{\g+1}|\theta|={ 2^{\g+1}}|x|^{\g}|y|.$$
Reversing the role of $x$ and $y$, we deduce that the following inequality is valid for any $x ,y \in \mathbb{R}$,
\begin{align}\label{ine2.3}
-\bigg( 2\Big| \frac{x + y}{2} \Big| - |x| - |y|\bigg)|x-y|^\gamma\leq { 2^{\gamma+1}}\max\{|x|^{\gamma},|y|^{\gamma}\}\min\{|x|,|y|\}\,. 
\end{align}
Multiplying this inequality with $\Gg(x)\Gg(y)$ and integrating over $\R^{2}$, we deduce, after  dividing the set of integration according to $|x| \geq |y|$ and its complement, that
$$
-\frac{1}{2}\int_{\R^{2}}\Gg(x)\Gg(y)\bigg( 2\Big|\frac{x + y}{2} \Big| - |x| - |y| \bigg)|x - y|^{\gamma}\d x\d y \leq  {2^{\gamma+1}}M_{1}(\Gg)M_{\gamma}(\Gg).
$$
Therefore,  since $M_{1}(\Gg) >0$ for any $\g \in (0,1)$ one deduces from \eqref{moment1} that
$M_{\g}(\Gg) \geq \frac{1}{{8}\cdot 2^{\g}}.$ This proves the result.\end{proof}

\subsection{Uniform lower bound for the energy}
To deduce a lower bound for the energy $M_{2}(\Gg)$ which is \emph{uniform} with respect to $\g$, we introduce a suitable scaling
\begin{subequations}\label{eq:Scale}
\begin{equation}\label{eq:scalea}
\widetilde{\Gg}(x) = \varrho\,\Gg(\varrho x), \qquad  \varrho>0,\end{equation}
which is such that
\begin{equation}\label{eq:momentSc}
\varrho^{k}\,M_k(\widetilde{\Gg}) = M_{k}(\Gg)\,,\qquad k\in\mathbb{R}\,,
\end{equation}
and satisfies
\begin{equation}\label{re-maineq}
\frac{1}{4\varrho^\gamma}\partial_{x}\big(x \widetilde{\Gg}\big) = \Q(\widetilde{\Gg},\widetilde{\Gg})\,.
\end{equation}
\end{subequations}
The following lemma holds.
\begin{lem}\label{size-a-final} For any $\,\g\in [0,1]$, { $\varrho\leq 1$} and $k \geq \g$,
$$0\leq\Big(\frac{1}{\varrho^{\gamma}}-1\Big)M_{2}(\widetilde{\Gg}) \leq \frac{2^{k+1}\gamma}{k}M_{2+k}(\widetilde{\Gg})\,,$$
holds for any $\Gg \in \mathscr{E}_{\g}$ with $\widetilde{\Gg}$ given by \eqref{eq:scalea} where
\begin{multline}\label{eq:m2+k}
\frac{k+2}{2\varrho^{\g}}M_{2+k}(\widetilde{\Gg})\\
=-\int_{\R^{2}}\widetilde{\Gg}(x)\widetilde{\Gg}(y)\left(2\left|\frac{x+y}{2}\right|^{2+k}-|x|^{2+k}-|y|^{2+k}\right)|x-y|^{\g}\d x\d y.\end{multline}
\end{lem}
\begin{proof}  As in the proof of \eqref{eq1}, one deduces the energy estimate for $\widetilde{\Gg}$
\begin{equation}\label{size-a}
0\leq\Big(\frac{1}{\varrho^{\gamma}}-1\Big)M_{2}(\widetilde{\Gg}) = \frac{1}{2}\int_{\R^{2}}\widetilde{\Gg}(x) \widetilde{\Gg}(y)|x-y|^{2}\big(|x-y|^{\gamma}-1\big)\d x\d y\,.
\end{equation}
Since, for $x \geq 1$ and $k \geq \g$,
$$x^{\g}-1=\g \int_{1}^{x}r^{\g-1}\d r  \leq \g\int_{1}^{x}r^{k-1}\d r=\frac{\g}{k}\left(x^{k}-1\right)\,,$$
it follows that
\begin{equation*}\begin{split}
\frac{1}{2}\int_{\R^{2}}\widetilde{\Gg}(x) \widetilde{\Gg}(y)&|x-y|^{2}\big(|x-y|^{\gamma}-1\big)\d x\d y \\
&\leq \frac{1}{2}\int_{|x-y| \geq 1}\widetilde{\Gg}(x) \widetilde{\Gg}(y)|x-y|^{2}\big(|x-y|^{\gamma}-1\big)\d x\d y\\
&\leq \frac{\g}{2k}\int_{|x-y| \geq 1}\widetilde{\Gg}(x) \widetilde{\Gg}(y)|x-y|^{2}\big(|x-y|^{k}-1\big)\d x\d y \\
&\leq \frac{\g}{2k}\int_{\R^{2}}\widetilde{\Gg}(x) \widetilde{\Gg}(y)|x-y|^{2+k}\d x\d y.\end{split}\end{equation*}
Using that $M_{0}(\widetilde{\Gg})=1$ we deduce the result from \eqref{size-a}. The proof of \eqref{eq:m2+k} is deduced easily after multiplying the rescaled equation \eqref{re-maineq} by $|x|^{2+k}$ and integrating over $\R$ as in the proof of \eqref{eq1}.
\end{proof}
To prove the lower bound we choose the scaling parameter $\varrho >0$ as 
$$\varrho=\varrho_{\g}:=M_{2}(\Gg)^{\frac{\xi}{2}}, \qquad \g \in \left(0,1\right)$$
where $\xi \in (0,1)$ is suitably chosen in the sequel. With this choice of $\varrho>0$ one sees that
$$0< M_2(\widetilde{\Gg}) = \big(M_2(\Gg)\big)^{1-\xi}\leq 1$$
and
\begin{equation}\label{lower_bound}
  \varrho^{\g}=\left(M_{2}(\Gg)\right)^{\frac{\xi\g}{2}} \geq  \Big( \frac{1}{{8}\,\cdot 2^{\gamma}} \Big)^{\xi}\,,\qquad \xi\in(0,1]\,,
\end{equation}
according to \eqref{lower-bound}. %In particular, for any $\varepsilon >0$ one can find $\xi \in (0,1)$ such that
%$$\frac{1}{\varrho^{\g}} \leq 1+\varepsilon \qquad \forall\, \g \in \big[0,\tfrac{1}{2}\big).$$

    %
%
In order to prove the lower bound on $M_{2}(\Gg)$ in Proposition \ref{theo:energy0}, we prove that $M_{2}(\widetilde{\Gg})$ is controlling $M_{k+\g}(\widetilde{\Gg})$ with some $k > 2$. To do so, we control the integral in \eqref{eq:m2+k} using the following elementary inequality.
{\begin{lem}\label{lem:elem}
For any $k_{0} \in (0,\frac{1}{2})$  there exists $\alpha_{0}$ such that for any $\alpha\in(0,\alpha_0)$ there exists  $\delta >0$ such that
\begin{equation}\label{eq:alpha-delta}
-\left(2\left|\frac{x+y}{2}\right|^{2+k}-|x|^{2+k}-|y|^{2+k}\right) \geq \left(1-2^{-1-k}-\alpha\right)\max\left(|x|^{{ k}},|y|^{{ k}}\right)|x-y|^{2}
\end{equation}
holds true for any $(x,y) \in \mathcal{A}=\left\{-\delta \leq \dfrac{x}{y} \leq 1 \text{ or } -\delta \leq \dfrac{y}{x} \leq 1\right\}$ and  $k \in (k_{0},1-k_{0})$.
\end{lem}
\begin{proof} For simplicity of notation we set $m=2+k$.  Notice that introducing $u=\frac{x}{y}$ or $u=\frac{y}{x}$ it is enough to prove that there is $\alpha_{0} >0$ such that for any $\alpha\in(0,\alpha_0)$ there exists $\delta >0$ such that 
\begin{multline}\label{eq:alphadeltau}
  -\left[2^{1-m}(1+u)^{m}-|u|^{m}-1\right] \geq \left(1-2^{1-m}-\alpha\right)(u-1)^{2}, \quad \\
  \forall u \in (-\delta,1),\, m \in (2+k_{0},3-k_{0}).\end{multline}
Let us set
$$\Lambda:=\inf_{m \in (2+k_{0},3-k_{0})}\{ m(m-1)2^{1-m}+2(1-2^{1-m}), 2-2^{1-m}(2+m), -m+4(1-2^{1-m})\}.$$
Since we have $\Lambda>0,$ there exists $\alpha_{0}>0$ such that $\Lambda -4\alpha_0 >0$. Consequently, for any $\alpha\in(0,\alpha_{0})$, for any  $m \in (2+k_{0},3-k_{0})$
\begin{equation}\label{second_derivative}
 -m(m-1)2^{1-m}-2(1-2^{1-m})+2\alpha<0,
\end{equation}
\begin{equation}\label{first_derivative}
  2-2^{1-m}(2+m)-2\alpha>0 \qquad \mbox{ and } \qquad -m+4(1-2^{1-m})-4\alpha>0.
\end{equation}
For any $m \in (2+k_{0},3-k_{0})$ and $\alpha\in(0,\alpha_{0})$, we introduce the families of mappings
$$g_{m}(u)=1+|u|^{m}-2^{1-m}(1+u)^{m}-\left(1-2^{1-m}\right)(u-1)^{2},\;\;\quad u \in [-1,1].$$
and
$$f_{m}(u)=g_{m}(u)+\alpha(u-1)^{2},\;\;\quad u \in [-1,1].$$
Computing the first three derivatives $f_{m}',f_{m}''$ and $f_{m}^{(3)}$, direct inspection shows that
$f_{m}^{(3)}(u) < 0$ for $u < 0$, $f_{m}^{(3)}(u) >0$ for $u>0$, while $f_{m}''(0)=-m(m-1)2^{1-m}-2(1-2^{1-m}) +2\alpha < 0$ by \eqref{second_derivative}, 
$$f_{m}''(-1) > f_{m}''(1)=\frac{m(m-1)}{2}-2(1-2^{1-m})+2\alpha>0, \qquad  \forall \,m \in (2+k_{0},3-k_{0})$$ and $f_{m}'(0)=2-2^{1-m}(2+m)-2\alpha> 0$, $f_{m}'(-1)=-m+4(1-2^{1-m})-4\alpha >0$, $f_{m}'(1)=0$ for all $m \in (2+k_{0},3-k_{0})$. From this,  elementary considerations show that there exists a unique $\theta_{m} \in (0,1)$ such that 
$$f_{m}'(\theta_{m})=0, \qquad f'_{m}(u) > 0 \quad u \in (-1,\theta_{m}), \quad f'_{m}(u) < 0 \quad u \in (\theta_{m},1),$$
hold true for any $m \in (2+k_{0},3-k_{0}).$ Since $f_{m}(0)=\alpha>0$, $f_{m}(1)=0$ and $f_{m}(-1)=2(2^{2-m}-1)+4\alpha < 2-m<0$ by \eqref{first_derivative}, there is a unique $\delta_{m}\in(0,1)$ such that $f_{m}(-\delta_m)=0$ and $f_{m}(u) > 0$ for all $u \in (-\delta_{m},1)$.  Notice that  $g_{m}(-\delta_{m})=-\alpha(\delta_{m}+1)^{2} <-\alpha$ so that $\inf_{m\in (2+k_0,3-k_0)}\delta_{m}=\delta >0$ (otherwise, $\inf_{m}\delta_{m}=0$ would contradict $g_{m}(0)=0$). This proves \eqref{eq:alphadeltau} for such choice of $\alpha_0,\delta$.
\end{proof}}
We deduce from this the following lemma.
\begin{lem}\label{lem:M2+k} For {any $k_0 \in (0,\frac12)$}, there exists $\beta_{k_0} >0$ such that,  for any $\g \in \left(0,k_0\right)$ and any $k\in(k_0,1-k_0)$,
$$M_{2+k+\g}(\widetilde{\Gg}) \leq \beta_{k_0}M_{2}(\widetilde{\Gg}).$$
\end{lem}
\begin{proof} For  {$k_0 \in \left(0,\frac{1}{2}\right)$ and $\g \in (0,k_0)$}, one applies \eqref{eq:m2+k} and \eqref{eq:alpha-delta} to deduce that {there exists $\alpha_0>0$ such that for any $\alpha\in(0,\alpha_0)$ and any $k\in(k_0,1-k_0)$,  }
\begin{multline*}
\frac{2+k}{2\varrho^{\g}}M_{2+k}(\widetilde{\Gg}) \geq -\int_{\mathcal{A}}\widetilde{\Gg}(x)\widetilde{\Gg}(y)\left(2\left|\frac{x+y}{2}\right|^{2+k}-|x|^{2+k}-|y|^{2+k}\right)|x-y|^{\g}\d x\d y \\
\geq \left(1-2^{-1-k}-\alpha\right)\int_{\mathcal{A}}\widetilde{\Gg}(x)\widetilde{\Gg}(y)\max\left(|x|^{{ k}},|y|^{{ k}}\right)|x-y|^{2+\g}\d x\d y.\end{multline*}
Using the elementary inequality
$$|x-y|^{2+\g}\geq |x|^{2+\g}+|y|^{2+\g}-{(2+\g)}\big( | x |^{1+\gamma}|y| + |y|^{1+\gamma}|x| \big)$$
valid for any $(x,y) \in \R^{2}$ it holds that
\begin{equation}\label{eq:main}\frac{2+k}{2\varrho^{\g}}M_{2+k}(\widetilde{\Gg}) \geq \left(1-2^{-1-k}-\alpha\right)\left(I_{1}-I_{2}-I_{3}\right)\end{equation}
where 
\begin{multline*}
I_{1}=\int_{\R^{2}}\widetilde{\Gg}(x)\widetilde{\Gg}(y)\max\left(|x|^{{ k}},|y|^{{ k}}\right)\left(|x|^{2+\g}+|y|^{2+\g}\right)\d x\d y\\
I_{2}=\int_{\mathcal{A}^{c}}\widetilde{\Gg}(x)\widetilde{\Gg}(y)\max\left(|x|^{{ k}},|y|^{{ k}}\right)\left(|x|^{2+\g}+|y|^{2+\g}\right)\d x\d y\end{multline*}
and
$$I_{3}={(2+\g)}\int_{\mathcal{A}}\widetilde{\Gg}(x)\widetilde{\Gg}(y)\max\left(|x|^{{ k}},|y|^{{ k}}\right)\left(|x|^{1+\g}|y|+|y|^{1+\g}|x|\right)\d x\d y\,.$$
Using that $M_{0}(\widetilde{\Gg})=1$ and splitting the integral according to the regions $|x| \geq |y|$ or $|y| \geq |x|$ one sees that
$$I_{1} \geq 2\int_{\R^{2}}\widetilde{\Gg}(x)|x|^{2+k+\g}\d x \int_{\R}\widetilde{\Gg}(y)\d y=2M_{2+k+\g}(\widetilde{\Gg}).$$
For the third integral, in the same way
\begin{equation*}\begin{split}
I_{3} &\leq 2 \,{(2+\g)}\int_{\R^{2}}\widetilde{\Gg}(x)\widetilde{\Gg}(y)\left(|x|^{k+1+\g}|y|+|y|^{1+\g}|x|^{k+1}\right)\d x\d y\\
&=2\,{(2+\g)}\left(M_{k+1+\g}(\widetilde{\Gg})M_{1}(\widetilde{\Gg})+M_{1+\g}(\widetilde{\Gg})M_{1+k}(\widetilde{\Gg})\right)\leq 4\,{(2+\g)}\left(M_{2}(\widetilde{\Gg})\right)^{\frac{k+2+\g}{2}}\end{split}\end{equation*}
where the last inequality follows from a simple interpolation recalling that {$k \in (k_0,1-k_0)$ has been chosen such that $k+1+\g < 2+\g-k_0<2$.} For the second term, we use the definition of $\mathcal{A}$ to deduce that
$$\min\{|x|,|y|\} \geq \delta\max\left\{|x|,|y|\right\}, \qquad \forall (x,y) \in \mathcal{A}^{c}$$
so that
\begin{equation*}\begin{split}
I_{2} &\leq \int_{\mathcal{A}^{c}}\widetilde{\Gg}(x)\widetilde{\Gg}(y){ \max(|x|,|y|)\left(|x|^{k+1+\g}+|y|^{k+1+\g}\right)}\d x\d y\\
&\leq \frac{2}{\delta}\int_{\R^{2}}\widetilde{\Gg}(x)\widetilde{\Gg}(y)\min(|x|,|y|) |x|^{k+1+\g} \d x\d y\leq \frac{2}{\delta}M_{k+1+\g}(\widetilde{\Gg})M_{1}(\widetilde{\Gg}),\end{split}\end{equation*}
and consequently,
$I_{2}  \leq \frac{2}{\delta}\left(M_{2}(\widetilde{\Gg}\right)^{\frac{k+2+\g}{2}}.$
Plugging these estimates in \eqref{eq:main} it follows that
\begin{multline*}
2\left(1-2^{-1-k}-\alpha\right)M_{2+k+\g}(\widetilde{\Gg}) 
\leq \frac{2+k}{2\varrho^{\g}}M_{2+k}(\widetilde{\Gg})\\+\left(\frac{2}{\delta}+4\, {(2+\g)}\right)\left(1-2^{-1-k}-\alpha\right)\left(M_{2}(\widetilde{\Gg})\right)^{\frac{2+k+\g}{2}}.\end{multline*}
%Recalling that $\frac{1}{\varrho^{\g}} \leq 1+\varepsilon$ and
Since $\gamma\in (0,\frac12)$, and since the definition of $\alpha_0$ implies that for any $\alpha\in(0,\alpha_0)$, 
$$\sup_{k\in(k_0,1-k_0)}\left\{\frac{2+k}{4}- (1-2^{-1-k}-\alpha)\right\} <0,$$ we deduce from \eqref{lower_bound} that there exists $\xi \in (0,1)$ independent of $\gamma$ such that
  $$\frac{2+k}{4\varrho^{\g}}\left(1-2^{-1-k}-\alpha\right)^{-1} \le (8\sqrt{2})^\xi\sup_{k\in(k_0,1-k_0)}\left\{\frac{2+k}{4}\left(1-2^{-1-k}-\alpha\right)^{-1}\right\} :=b_{k_{0}}<  1$$ 
  %$$b=b_{\g,k}=\frac{2+k}{4\varrho^{\g}}\left(1-2^{-1-k}-\alpha\right)^{-1}< 1$$
%and 
Then, 
$$M_{2+k+\g}(\widetilde{\Gg}) \leq b_{k_{0}}\,M_{2+k}(\widetilde{\Gg}) + C\left(M_{2}(\widetilde{\Gg})\right)^{\frac{2+k+\g}{2}}$$
with $C=2\, {(2+\g)}+\frac{1}{\delta}.$ Simple moments interpolation shows that
$$M_{2+k}(\widetilde{\Gg}) \leq M_{2}(\widetilde{\Gg})^{1-\theta}M_{2+k+\g}(\widetilde{\Gg})^{\theta} \leq (1-\theta)M_{2}(\widetilde{\Gg})+\theta\,M_{2+k+\g}(\widetilde{\Gg}), \quad \theta=\frac{k}{k+\g},$$
and therefore,
\begin{equation*}
M_{2+k+\g}(\widetilde{\Gg}) \leq \frac{b_{k_{0}}(1-\theta)}{1-b_{k_{0}}\theta}M_{2}(\widetilde{\Gg})+\frac{C}{1-b_{k_{0}}\theta}\left(M_{2}(\widetilde{\Gg})\right)^{\frac{2+k+\g}{2}}.
\end{equation*}
Since $M_{2}(\widetilde{\Gg}) \leq 1$ and $\sup_{\g \in [0,k_{0}]}\frac{1}{1-b_{k_{0}}\theta}\left(b_{k_{0}}(1-\theta)+C_{k}\right) < \infty$ we get the result.\end{proof}
The Lemma allows to prove Proposition \ref{theo:energy0}.
\begin{proof}[Proof of Proposition \ref{theo:energy0}] The upper bound on $M_{2}(\Gg)$ has been derived in Lemma \ref{lem:energy}. For the lower bound, one uses Lemma \ref{size-a-final} and Lemma  \ref{lem:M2+k} with  {$k_{0}\in(0,\frac12)$ and $\g\in(0,k_{0})$. For any $k\in(k_0,1-k_0)$, we obtain }
$$0\leq\Big(\frac{1}{\varrho^{\gamma}}-1\Big)M_{2}(\widetilde{\Gg}) \leq \frac{2^{k+\g+1}\gamma}{k+\g}M_{2+k+\g}(\widetilde{\Gg}) \leq \g\frac{\beta_{k_{0}}2^{k+\g+1}}{k+\g}M_{2}(\widetilde{\Gg}).$$
Since $M_{2}(\widetilde{\Gg}) >0$ we deduce that
$$0 \leq \Big(\frac{1}{\varrho^{\g}}-1\Big) \leq \g\frac{\beta_{k_{0}}2^{k+\g+1}}{k+\g} \leq c_{k}\g \qquad \forall \g \in \big(0,k_{0}\big).$$
This proves that $\frac{1}{\varrho^{\g}} \leq 1+c_{k}\g \leq \exp\left(c_{k}\g\right)$ and recalling the definition of $\varrho=\left(M_{2}(\Gg)\right)^{\frac{\xi}{2}}$ gives the result with $\bm{\alpha}_{2}=\exp\left(-\frac{2c_{k}}{\xi}\right) >0$.
\end{proof}
One expects $\Gg \to \bm{G}_{0}$ where $\bm{G}_{0}$ is a steady state to \eqref{introMax} with $\bm{G}_{0} \in L^{1}(\w_s)$ for any $s<3$. Thus, for $k >3$ one should expect that
$$\limsup_{\g\to0}M_{k}(\Gg)=\infty$$
whereas, for $2 < k < 3$,
$$\limsup_{\g\to0}M_{k+\g}(\Gg) < \infty.$$
This can be made rigorous thanks to Lemma \ref{lem:momentsk}.
\begin{lem}\label{lem:momentsk}  Let $k_0\in(0,\frac12)$.  There exists $C >0$ such that, for all  $\g \in \left(0,k_{0}\right)$, it holds
$$M_{p+\g}(\Gg) \leq C, \qquad \forall p \in (0,3-k_{0}), \qquad \Gg \in \mathscr{E}_{\g}.$$
\end{lem}
\begin{proof}  {Let $k_0\in(0,\frac12)$},  $\g \in \left(0,k_{0}\right)$ and $\Gg \in \mathscr{E}_{\g}.$ It suffices to consider of course $p > 2$ so we write $p=2+k,$ $k \in (0,1-k_{0})$.  We use here the scaling \eqref{eq:Scale} with $\varrho=\varrho_{\g}$ as before and deduce from Lemma \ref{lem:M2+k} that  {for any $k\in(k_{0},1-k_{0})$,} 
\begin{equation*}\begin{split}
M_{2+k+\g}(\Gg)&=\varrho^{2+k+\g}M_{2+k+\g}(\widetilde{\Gg}) \leq \beta_{k_{0}}\varrho^{2+k+\g}M_{2}(\widetilde{\Gg})\\
&\leq \beta_{k_{0}}\varrho^{k+\g}M_{2}(\Gg).\end{split}\end{equation*}
Recalling that $\varrho=M_{2}(\Gg)^{\frac{\xi}{2}}$ with $M_{2}(\Gg) < 1$, we deduce that
$$M_{2+k+\g}(\Gg) \leq \beta_{k_{0}} \qquad \forall \g \in  {\left(0,k_{0}\right)}, \qquad k \in  {(k_{0},1-k_{0})}.$$
This proves the result.%, simply observing that $\sup_{k \in (0,1)}\beta_{k_{0}}=C < \infty.$
\end{proof}
 \subsection{Weak convergence}\label{sec:cvgce}
Notice that the results of the previous subsections are valid for any $\Gg \in \mathscr{E}_{\g}$ and provide some weak-compactness features for the family $\{\Gg\}_{\g}.$ Indeed, a first consequence of the above energy estimate \eqref{eq:tildeC} is
that, for any choice of equilibria $\Gg \in \mathscr{E}_{\g}$,
$$\left\{\Gg(x)\d x\right\}_{\gamma\in(0,1)} \quad \text{ is a tight set of probability measures}.$$ 
From Prokhorov's compactness Theorem (see \cite[Theorem 1.7.6, p. 41]{kolo}),  there exist some probability measure $\mu(\d x)$ and some sequence $(\gamma_n)_{n\in\N}$ tending to $0$ such that $\left\{\bm{G}_{\gamma_n}(x)\dx\right\}_{n\in\N}$ converges narrowly to $\mu(\!\dx)$.  Owing to Lemma \ref{lem:energy}, the set $\{\bm{G}_{\g}(x)\dx\}_{\g\in(0,1)}$ is also tight in $\mathcal{M}_k(\R)$ for $k\in[0,2]$. Consequently, $\mu(\!\dx) \in \mathcal{M}_{2}(\R)$ and the narrow convergence still holds with a weight $x^{k}$ with $k\in[0,2]$. This means that for any $k\in[0,2]$, and any $\varphi \in \mathcal{C}_{b}(\R)$,
  \begin{equation}\label{narrow_cvge}\int_{\R} \varphi(x)\, x^{k} \,\bm{G}_{\gamma_n}(x) \dx \underset{n\to\infty}{\longrightarrow} \int_{\R}\varphi(x) x^{k} \mu(\!\dx).
  \end{equation}
    Let $\phi\in\mathcal{C}_{b}^{1}(\R)$. Set $\psi_n(x,y)=|x-y|^{\gamma_n} \left(2\phi\left(\frac{x+y}{2}\right)-\phi(x)-\phi(y)\right)$ and $\psi(x,y)=2\phi\left(\frac{x+y}{2}\right)-\phi(x)-\phi(y)$, one can check without much difficulty (see \cite{long}) that 
$$\lim_{n\to \infty} \int_{\R^2} \psi_n(x,y) \bm{G}_{\gamma_n}(x) \bm{G}_{\gamma_n}(y)\d x \d y =  \int_{\R^2} \psi(x,y) \mu(\!\dx) \mu(\dy).$$
We thereby obtain that $\mu(\d x)$ is a steady
solution to \eqref{introMax}. By virtue of Proposition \ref{theo:energy0}, 
$$\int_{\R}x^{2}\mu(\d x)=\frac{1}{\lambda^{2}} > 0$$
for some $\lambda >0$ and then, according to Theorem \ref{theo:bob}, 
$$\mu(\!\dx)=H_{\lambda}(x)\dx=\lambda\bm{H}(\lambda x)\dx.$$

\section{The limiting process $\g \to 0$ and \emph{a posteriori} regularity estimates}
\label{sec:apost}
 We complement the uniform estimates on $\Gg$ for moments as derived in the previous Section with additional regularity estimates -- still uniform with respect to $\g$ -- and make the limiting process explicit. We begin with $L^{2}$-integrability  estimates.

\subsection{$L^{2}$-estimates on the profile}
\label{sec:L2est}

Before going with $L^{2}$-estimates on $\Gg$, we begin with the following uniform \emph{pointwise} upper
bound.
\begin{lem}\phantomsection\label{lem:Linfty} There exists $C_{\infty} >0$ \emph{independent of $\g$} such that
\begin{equation}\label{eq:pointX}
\Gg(x) \leq \frac{C_{\infty}}{|x|} \qquad  {\mbox{ for  a.e. } x \in \R},\end{equation}
holds true for any $\g\in[0,1]$ and any $\Gg \in \mathscr{E}_{\g}$.
\end{lem}
\begin{proof} Formally, integrating the equation for $\Gg$ in $(0,x)$, with $x>0$, one has
\begin{equation*}
x\,\Gg(x) = 4\int^{x}_{0} \Q_{\g}(\Gg,\Gg) \dy \leq 4\| \Q^{+}_{\g}(\Gg,\Gg) \|_{L^1}\leq C\| \Gg \|_{L^{1}(\w_{\g})} \| \Gg \|_{L^1}\,.
\end{equation*}
Similarly, for $x<0$, we obtain
\begin{equation*}
- x\,\Gg(x) = 4\int^{0}_{x} \Q_{\g}(\Gg,\Gg) \dy \leq 4\| \Q^{+}_{\g}(\Gg,\Gg) \|_{L^1}\leq C\| \Gg \|_{L^{1}(\w_{\g})} \| \Gg \|_{L^1}\,.
\end{equation*}
The uniform control provided by \eqref{eq:tildeC}  yields the result.  We refer to \cite{long} for a rigorous justification of the estimate.
\end{proof}

We deduce from this the following technical estimate regarding the control of $L^{2}$-norms.
\begin{lem}\phantomsection \label{lem:boundL2L} There exists some
  universal numerical constant $C_0 >0$ such that the inequality
  \begin{equation}
    \label{eq:boundL2L}
    \|\Gg\|_{L^2}^{2} \leq
    C_{0} \|\Gg\|_{L^2}^{2}\int_{-\ell}^{\ell}|x|^{\g}\Gg(x)\dx
    + C_0\,\ell^{\g-1}
  \end{equation}
  holds true for any  {$\g \in(0,1)$}, $\Gg \in \mathscr{E}_{\g}$ and any
  $\ell >0$.
\end{lem}
 
\begin{proof}
  We provide here a formal proof which provides the main ideas
  underlying the result. We refer to \cite[Appendix C.1]{long} for
  a rigorous justification of the formal argument that follows.  For
  any generic solution $\Gg$ to \eqref{eq:steadyg}, after multiplying
  \eqref{eq:steadyg} with $\Gg$ and integrating over $\R$ one sees
  that
\begin{equation}\label{eq:L2Gg}
\frac{1}{8}\|\Gg \|^{2}_{L^2} \leq \int_{\R}\Q^{+}_{\g}(\Gg,\Gg)\,\Gg\dx.\end{equation}
One sees that
\begin{equation*}\begin{split}
\int_{\R}\Q^{+}_{\g}(\Gg,\Gg)\,\Gg\dx&=\int_{\R}\int_{\R}|x-y|^{\g}\Gg(x)\Gg(y)\Gg\left(\frac{x+y}{2}\right)\dx\dy\\
&\leq \int_{\R}\int_{\R}\left(|x|^{\g}+|y|^{\g}\right)\Gg(x)\Gg(y)\Gg\left(\frac{x+y}{2}\right)\dx\dy\\
&=2\int_{\R}|x|^{\g}\Gg(x)\dx\int_{\R}\Gg(y)\Gg\left(\frac{x+y}{2}\right)\dy.
\end{split}\end{equation*}
Notice that such an inequality actually means that
\begin{equation}\label{eq:QgQ0}
\int_{\R}\Q^{+}_{\g}(\Gg,\Gg)\,\Gg\dx \leq 2\int_{\R}\Q^{+}_{0}(|\cdot|^{\g}\Gg,\Gg)\Gg\d x.\end{equation}
Given $\ell >0$, splitting the integral with respect to $x$ according to $|x| >\ell$ and $|x| \leq \ell$, one has
\begin{equation*}\begin{split}
\int_{\R}\Q^{+}_{\g}(\Gg,\Gg)\,\Gg\dx &\leq  2\int_{-\ell}^{\ell}|x|^{\g}\Gg(x)\dx\int_{\R}\Gg(y)\Gg\left(\frac{x+y}{2}\right)\dy\\
&\phantom{++++} + 2C_{\infty}\ell^{\g-1}\int_{\R}\dx\int_{\R}\Gg(y)\Gg\left(\frac{x+y}{2}\right)\dy\,,
\end{split}\end{equation*}
where we used \eqref{eq:pointX} in the last step. Clearly, the last integral can be estimated as
$$\int_{\R}\dx\int_{\R}\Gg(y)\Gg\left(\frac{x+y}{2}\right)\dy=2\|\Gg\|_{L^1}^2=2$$
whereas, for any given $x \in [-\ell,\ell]$, one has from Cauchy-Schwarz inequality
$$\int_{\R}\Gg(y)\Gg\left(\frac{x+y}{2}\right)\dy \leq \|\Gg\|_{L^2}\,\left\|\Gg\left(\frac{x+\cdot}{2}\right)\right\|_{L^2}=\sqrt{2}\|\Gg\|_{L^2}^{2}.$$
Combining these estimates, we deduce that
$$\int_{\R}\Q^{+}_{\g}(\Gg,\Gg)\,\Gg\dx\leq 2\sqrt{2}\|\Gg\|_{L^2}^{2}\,\int_{-\ell}^{\ell}\,|x|^{\g}\Gg(x)\dx+4C_{\infty}\ell^{\g-1}.$$
This gives the desired result thanks to \eqref{eq:L2Gg}.
\end{proof}
A trivial bound for the integral is the following
$\ds\int_{-\ell}^{\ell}|x|^{\g}\Gg(x)\dx \leq \ell^\g\,\|\Gg\|_{L^1}=\ell^{\g}$
which gives a bound like
$$\|\Gg\|_{L^2}^2 \leq C_0 \ell^{\g}\|\Gg\|_{L^2}^{2} + C_0 \ell^{\g-1}$$
and cannot provide a bound on $\|\Gg\|_{L^2}$ uniform with respect to $\g$. If one assumes say 
\begin{equation}\label{eq:controlL}
  C_0\int_{-\ell}^{\ell}|x|^{\g}\Gg(x)\dx
  \leq
  C_1\ell^{\g+1}
\end{equation}
for some universal (independent of $\g$) constant $C_1$, then picking
$\ell$ small enough would yield a bound on $\|\Gg\|_{L^{2}}$
\emph{uniform with respect to $\g$ small enough}.  We are actually not
able to establish the bound \eqref{eq:controlL} for any
$\Gg \in \mathscr{E}_\g$ but will provide a similar estimate for any
sequence $\{\bm{G}_{\g_n}\}$ converging weakly-$\star$. Namely, one has the following lemma.
\begin{lem}\phantomsection\label{lem:L2bound}
  Let $\left(\g_{n}\right)_{n}$ be a sequence going to zero,
  $\bm{G}_{\g_n} \in \mathscr{E}_{\g_n}$ an equilibrium for each $n$,
  and $\lambda >0$ such that
  \begin{equation}\label{eq:weakCC}
  \lim_{n\to\infty}\int_{\R}\bm{G}_{\g_{n}}(x)\varphi(x)\dx=\int_{\R}H_{\lambda}(x)\varphi(x)\dx, \qquad \forall \varphi \in \mathcal{C}_{b}(\R).\end{equation}Then there exists $C=C(\lambda)$ depending only on $\lambda$ and
  $N \geq 1$ such that
  $$\sup_{n \geq N}\|\bm{G}_{\g_{n}}\|_{L^2} \leq C.$$
\end{lem}

\begin{proof}
  From the weak-$\star$ convergence, for any $\ell >0$, one can choose
  a smooth cutoff function $\varphi_{\ell}\geq0$ equal to one in
  $[-\ell,\ell]$ and vanishing on $\R \setminus [-2\ell,2\ell]$ to
  deduce that there exists $N >1$ such that
$$\int_{-\ell}^{\ell}\bm{G}_{\g_{n}}(x)\dx \leq 2\int_{-2\ell}^{2\ell}H_{\lambda}(x)\dx=2\int_{-2\lambda\,\ell}^{2\lambda\,\ell}\bm{H}(x)\dx \qquad \forall n \geq N.$$
Direct computations show that
$$\int_{-2\lambda\,\ell}^{2\lambda\,\ell}\bm{H}(x)\dx=\frac{2}{\pi}\left[\arctan\left(2\lambda\,\ell\right)+\frac{{2}\lambda\,\ell}{1+4\lambda^{2}\ell^{2}}\right] \leq \frac{8}{\pi}\lambda\,\ell \qquad \forall \ell >0, \lambda >0.$$
Thus, for any $n \geq N$, one has
$$\int_{-\ell}^{\ell}|x|^{\g_n}\bm{G}_{\g_{n}}(x)\dx \leq \ell^{\g_n}\,\int_{-\ell}^{\ell}\bm{G}_{\g_{n}}(x)\dx \leq \frac{16\lambda}{\pi}\ell^{\g_n+1}.$$
Arguing as described previously, plugging this into \eqref{eq:boundL2L} we get
$$\|\bm{G}_{\g_n}\|_{L^2}^{2} \leq \frac{16\,C_0\,\lambda}{\pi}\ell^{\g_n+1}\|\bm{G}_{\g_n}\|_{L^2}^{2} + C_0\,\ell^{\g_n-1} \qquad \forall n \geq N, \qquad \forall \ell >0.$$
Picking then $\ell\leq 1$ (depending on $\lambda$) such that
$$ \frac{16\,C_0\,\lambda}{\pi}\ell^{\g_n+1}\leq \frac{16\,C_0\,\lambda}{\pi}\ell \leq\frac{1}{2}, \qquad \text{ i.e. } \qquad \ell=\min\left\{\left(\frac{\pi}{{32}\lambda\,C_0}\right),1\right\}$$
we deduce that
$$\|\bm{G}_{\g_n}\|_{L^2}^{2} \leq 2C_0\,\ell^{\g_n-1}$$
and, since $\ell \leq 1$ and $\g_{n}\ge0$, $\ell^{\g_{n}} \leq 1$ so that
$$\|\bm{G}_{\g_{n}}\|_{L^{2}}^{2} \leq 2\frac{C_{0}}{\ell}=2C_{0}\max\left\{\frac{32}{\pi}C_{0}\lambda,1\right\}$$
which gives the result.\end{proof}
\subsection{Limiting temperature and proof of Theorem~\ref{theo:Unique}}
\label{Sec:limit:temp}
We prove here that  the parameter $\lambda$ is actually uniquely determined, yielding the uniqueness of the possible limit point different from the Dirac mass $\delta_{0}$.  Namely, we prove the following lemma.
\begin{lem}\phantomsection\phantomsection\label{lem:unique} Under the assumptions of Lemma \ref{lem:L2bound} and under the convergence \eqref{eq:weakCC}, one has 
$$\lambda=\lambda_{0}:=\exp\left(\frac12 \mathscr{I}_{0}(\bm{H},\bm{H})\right)$$
where
\begin{equation}\label{eq:A0}
 \mathscr{I}_{0}(\bm{H},\bm{H}):=\int_{\R}\int_{\R}\bm{H}(x)\bm{H}(y)|x-y|^{2}\log|x-y|\dx\dy >0.
\end{equation}
\end{lem} 
{\begin{rmk} It turns out that  {$\mathscr{I}_{0}(\bm{H},\bm{H})$} can be made explicit and, according to Lemma \ref{rmk:varphi0},
 {$\mathscr{I}_{0}(\bm{H},\bm{H})=2\log2+1$} from which
$\lambda_{0}=2\sqrt{e}.$\end{rmk}
}
\begin{proof} We consider a sequence $\{\bm{G}_{\g_{n}}\}_{n}$ and $\lambda >0$ such that \eqref{eq:weakCC} holds. Let  {$\delta\in(0,\frac12)$}. Let us fix $k\in(2,3)$ and $s>0$ small enough such that $k+s\in(2+\delta,3-\delta)$. We consider $N\in\N$ large enough such that the conclusions of Lemma \ref{lem:L2bound}  and Lemma \ref{lem:momentsk} hold true and  {$\g_{n}<\delta$} for any  $n\ge N$. Then, Lemma \ref{lem:momentsk} and Young's inequality imply that
\begin{equation}\label{Mk+s} 
M_{k+s}(\bm{G}_{\g_{n}}) \le M_{k+s+\gamma_n}(\bm{G}_{\g_{n}})+M_{0}(\bm{G}_{\g_{n}}) \le C+1=:\bar{C},
\end {equation}
 for any $n \geq N$. Introducing 
$$\Lambda_{\g}(r)=\frac{r^{\g}-1}{\g}, \qquad \forall r >0, \qquad \g >0$$
we recall from \eqref{eq1} that
\begin{equation}\label{eq1-gn}
\int_{\R}\int_{\R}\bm{G}_{\g_{n}}(x)\bm{G}_{\g_{n}}(y)|x-y|^{2}\Lambda_{\g_{n}}(|x-y|)\dx\dy=0\,\quad \quad \forall \,n\geq N.
\end{equation}
The weak convergence of $\bm{G}_{\g_{n}}$ towards $H_{\lambda}$ together with the control of moments in \eqref{Mk+s} and the fact that 
$$\lim_{\g\to0}\Lambda_{\g}(r)=\log r, \qquad \forall r >0$$ 
allow to prove  that, at the limit as $n \to \infty$, \eqref{eq1-gn} yields
\begin{equation}\label{eq1Hl}
\int_{\R}\int_{\R}H_{\lambda}(x)H_{\lambda}(y)|x-y|^{2}\log|x-y|\dx\dy=0.
\end{equation}
Details are provided in \cite[Lemma 3.4]{long}.  Now, recalling that $H_{\lambda}(x)=\lambda\bm{H}(\lambda x)$ for any $x \in \R$, with the change of variables $u=\lambda\,x$, $v=\lambda\,y$, \eqref{eq1Hl} becomes
$$\frac{1}{\lambda^{2}}\int_{\R}\int_{\R}\bm{H}(u)\bm{H}(v)|u-v|^{2}\log\left(\frac{|u-v|}{\lambda}\right)\d u\d v=0$$
from which
\begin{multline}
\log \lambda \int_{\R}\int_{\R}\bm{H}(u)\bm{H}(v)|u-v|^{2}\d u\d v \\
= \int_{\R}\int_{\R}\bm{H}(u)\bm{H}(v)|u-v|^{2}\log |u-v|\d u\d v= \mathscr{I}_{0}(\bm{H},\bm{H}).
\end{multline}
Since 
$$\int_{\R}\int_{\R}\bm{H}(u)\bm{H}(v)|u-v|^{2}\d u\d v=2\int_{\R}|u|^{2}\bm{H}(u)\bm{H}(v)\d u\d v=2$$
we deduce the result. 
\end{proof}
The aforementioned lemma gives a full proof of Theorem \ref{theo:Unique} in the Introduction.
\begin{proof}[Proof of Theorem \ref{theo:Unique}]  
Lemma~\ref{lem:unique} proves that the weakly-$\star$ compact family $\left\{\Gg\right\}_{\g\in (0,1)}$ admits a \emph{unique} possible limit (as $\g \to0$) given by
$$\bm{G}_{0}(x):=\lambda_{0}\bm{H}(\lambda_{0}x), \qquad \lambda_{0}=\exp\left( {\frac12 \mathscr{I}_{0}(\bm{H},\bm{H})}\right)$$
with {$\mathscr{I}_{0}(\bm{H},\bm{H})$} defined {in \eqref{eq:A0}}. In particular, the whole net $\left\{\Gg\right\}_{\g \in (0,1)}$ is converging (in the weak-$\star$ topology) towards $\bm{G}_{0}$. We can then resume the arguments of  
Lemma \ref{lem:L2bound} to deduce the $L^{2}$-bound in \eqref{eq:estimGg} while Lemma \ref{lem:momentsk} gives the moments estimates in  \eqref{eq:estimGg}. \end{proof}
 
\subsection{Uniform control of weighted $L^{2}$-norms}\label{sec:weighL2} We can complement the estimates \eqref{eq:estimGg} in
Theorem~\ref{theo:Unique} with $L^{2}$-moments estimates.  Notice that, combining the  estimates in \eqref{eq:estimGg} together with the pointwise bound in Lemma \ref{lem:Linfty} allows directly to provide a control of $\|\Gg\|_{L^{2}(\w_{k})}$ for $k \in (0,2-\delta)$. Indeed, one observes that 
$$\int_{\R}\Gg^{2}(x)|x|^{2k}\d x \leq C_{\infty}\int_{\R}\Gg(x)|x|^{2k-1}\d x$$
according to Lemma \ref{lem:Linfty} and then, as long as {$\frac12<k < 2-\frac{\delta}{2}$}, one can apply \eqref{eq:estimGg} to deduce that
$$\int_{\R}\Gg^{2}(x)|x|^{2k}\d x \leq \bm{C}_{0}$$
which is enough, using the $L^{2}$-estimate in \eqref{eq:estimGg}, to bound $\|\Gg\|_{L^{2}(\bm{w}_{k})}$. We however wish to extend the range of the parameter $k$ up to $3$ which is done in the following corollary where a lower bound on $\Q^{-}_{\g}(\Gg,\Gg)$ is used and resort to Lemma \ref{lem:Sigmag} in Appendix \ref{app:tech}
\begin{cor}\label{L2-weighted}
  For any $\delta\in(0,\frac12)$ there exists $\g_{\star} \in (0,1)$ and $C >0$ such that
  \begin{equation}\label{eq:L2-weighted}
    \|\Gg \|_{L^2(\w_{k})} \leq C
  \end{equation}
  for all $\g \in [0,\g_{\star})$ with $k+\g \in (0,3-\delta)$ and all $\Gg \in \mathscr{E}_\gamma$.
\end{cor}

\begin{proof}
  We give a formal proof here which presents the argument to obtain a
  uniform bound. A complete justification can be found in \cite[Appendix C.1]{long}. Let $\g_\star\in(0,1)$ be such that
  the conclusion of Theorem \ref{theo:Unique} holds true. For any
  $k>0$, setting
$$G_{k}(x)=\Gg(x)\,|x|^{k}$$
one notices that
\begin{equation}\label{eq:Gk}
\int_{\R}\Q_{\g}(\Gg,\Gg)\Gg\,|x|^{2k}\dx=\frac{1}{4}\int_{\R}\partial_x (x\Gg)\, \Gg\,|x|^{2k} \dx = \Big( \frac{1}{8} - \frac{k}{4}\Big)\|G_{k} \|^{2}_{L^2}\,.
\end{equation}
Also, thanks to Lemma \ref{lem:Sigmag} 
$$\int_{\R}\Q^{-}_{\gamma}(\Gg,\Gg)\,\Gg\,|x|^{2k} \dx =  \int_{\R}G^{2}_{k}(x)\Sigma_{\g}(x)\,\dx  \geq  \kappa_{\gamma}\|G_{k}\w_{\frac{\g}{2}}\|^2_{L^2}\geq \kappa_{\g}\|G_{k}\|_{L^{2}}^{2}\,.$$
For the positive part, using that 
$$|x+y|^{k}\leq 2^{k-1}\big(|x|^{k} + |y|^{k}\big) \;\text{ while } \;(|x|^{k}+|y|^{k})|x-y|^{\g} \leq 2\big(|x|^{k+\g} + |y|^{k+\g}\big),$$
we can argue as in the derivation of \eqref{eq:QgQ0} to conclude that
\begin{align*}
\int_{\R}\Q^{+}_{\gamma}(\Gg,\Gg)\,\Gg\,|x|^{2k} \dx &\leq 2\int_{\R}\Q^{+}_{0}(\Gg |x|^{k+\gamma},\Gg)\, G_k \dx \\
& \leq  {2\sqrt{2}}M_{k+\gamma}(\Gg)\| \Gg\|_{L^2} \|G_{k}\|_{L^2}\,.
\end{align*}
Therefore,  one deduces from Theorem \ref{theo:Unique}  that there exists some positive constant $C_{0}$ depending neither on $k$, nor on $\g$ such that for $k+\g\in(0,3-\delta)$, 
$$ \int_{\R}\Q^{+}_{\gamma}(\Gg,\Gg)\,\Gg\,|x|^{2k} \dx \leq C_{0}\|G_{k}\|_{L^2}.$$
Gathering these estimates with \eqref{eq:Gk}, one deduces that
\begin{equation*}
\Big(\kappa_\gamma  + \frac{1}{8} - \frac{k}{4}\Big)\| G_{k} \|^{2}_{L^2} \leq C_{0}\| G_{k}\|_{L^2}\,.
\end{equation*}
Since $\kappa_\gamma\rightarrow1$ as $\gamma \to 0^{+}$, one easily concludes that for some explicit $\gamma_{\star} >0$ (independent of $k$), it holds $\kappa_\gamma  + \frac{1}{8} - \frac{k}{4}   \geq \frac{1}{8}$ for any $\gamma\in[0,\gamma_{\star})$ which proves the result since then $\|G_{k}\|_{L^2} \leq 8C_{0}$.
\end{proof}
\subsection{Higher regularity}\label{sec:weighted}
In this section we  pursue our analysis of the behaviour of steady states $\Gg \in \mathscr{E}_{\g}$ and prove  Sobolev regularity  uniformly with respect to $\g$. Parts of the arguments are formal while a full justification is given in \cite[Appendix~C]{long}.   From equation \eqref{eq:steadyg} we write
\begin{equation*}
x\partial_{x}\Gg = 4\Q_{\g}(\Gg ,\Gg ) - \Gg\,.
\end{equation*}
Consequently, taking the $L^{2}(\w_{k})$ norm with $k\in (0,3)$,   one has
\begin{align*}
\begin{split}
\| x\partial_{x}\Gg \|_{L^{2}(\w_{k})}  &\leq  4\| \Q_{\g}(\Gg ,\Gg )\|_{ L^{2}(\w_{k}) } + \| \Gg \|_{L^{2}(\w_{k})}\\
&\leq  C \left(\|\Gg\w_{k+\g}\|_{L^2}\, \|\Gg\w_{\g}\|_{L^1}+\|\Gg\|_{L^2} \|\Gg\w_{k+\g}\|_{L^1}\right)+\| \Gg \|_{L^{2}(\w_{k})}\,,
\end{split}
\end{align*}
thanks to Proposition \ref{prop:QgL2}. {Let $\delta\in\left(0,\frac12\right)$  be such that $k<3-\delta$}. Using now the uniform estimates obtained in Theorem \ref{theo:Unique} and Corollary \ref{L2-weighted}, we see that
\begin{equation}\label{eq:init-gradient-regularity}
  \sup_{\g \in (0,\g_{\star})}\|x\partial_{x}\Gg\|_{L^{2}(\w_{k})}=C_{1} < \infty, \qquad \forall   k+2\g <3-\delta.\end{equation}
In order to deduce from this some $L^{2}$-estimate for $\partial_{x}\Gg$, we need to handle the small values of $x$. Introducing now
$\Gg'(x)=\partial_{x}\Gg(x)$
one differentiates \eqref{eq:steadyg} to obtain that
 \begin{equation}\label{gradient-equation}
\frac{1}{4}\partial_x(x\Gg') + \frac{1}{4}\Gg' =2 \Q_{\g}(\Gg,\Gg').
\end{equation}
Let us estimate the gain and the loss term on the right side separately in the following lemmata.
\begin{lem}[\textit{\textbf{Gain collisional estimate}}]\phantomsection\label{collision+} 
 Let {$\delta\in(0,\frac12)$} and $\g_\star\in(0,1)$ given by Corollary \ref{L2-weighted}. For any $\tilde{\delta}>0$, $\g \in(0,\g_\star)$ and  {$0\le k<3-\frac{5\gamma}{2}-\delta$} it holds that
 {\begin{equation}\label{eq:L2-weighted-q+}
\int_{\R}\Q^{+}_{\g}(\Gg,|\Gg'|) \,|\Gg'|\,\w_{2k} \dx \leq C\sqrt{\tilde{\delta}}\| \Gg'\w_{k+\frac{\g}{2}}\|^2_{L^2} + \frac{C}{\tilde{\delta}^{\frac{5}{2}}}\,,
\end{equation}}
for some explicit $C>0$.
 \end{lem}
\begin{proof}  As in the proof of \eqref{eq:QgQ0}, one first observes that
\begin{multline*}
\int_{\R}\Q^{+}_{\g}(\Gg,|\Gg'|) \,|\Gg'|\,\w_{2k} \dx\\
\leq \int_{\R^{2}}\left(|x|^{\g}+|y|^{\g}\right)\Gg(x)|\Gg'(y)|\,\left|\Gg'\left(\frac{x+y}{2}\right)\right|\w_{2k}\left(\frac{x+y}{2}\right)\dx\dy\\
\leq \int_{\R^{2}}\left(|x|^{\g}+|y|^{\g}\right)\Gg(x)|\Gg'(y)|\,\left|\Gg'\left(\frac{x+y}{2}\right)\right|\w_{k}\left(\frac{x+y}{2}\right) \w_{k}(x)\w_{k}(y) \dx\dy
\end{multline*}
where we used first 
 that $\w_{2k}(\cdot) {=} \w_{k}(\cdot)^{2}$ and then that 
$\w_{k}\left(\frac{x+y}{2}\right)\leq \w_{k}(x)\w_{k}(y).$ Consequently, using the fact that {$r \mapsto r^{\g}$} is concave, one checks that 
\begin{equation*} 
|x|^{\g}+|y|^{\g}\le 2\left|\frac{x+y}{2}\right|^{\g} \le 2^{1-\frac{\g}{2}}\, (|x|+|y|)^{\frac{\g}{2}}  \, \left(\frac{|x+y|}{2}\right)^{\frac{\g}{2}} 
 \le 2 \,\w_{\frac{\g}{2}}(x)\,\w_{\frac{\g}{2}}(y)\,\w_{{\frac{\g}{2}}}\left(\frac{x+y}{2}\right)
\end{equation*}
from which we deduce that
\begin{multline*}
\int_{\R}\Q^{+}_{\g}(\Gg,|\Gg'|) \,|\Gg'|\,\w_{2k} \dx\\
\leq 2\int_{\R^2}\left(\w_{k+\frac{\g}{2}}(x)\Gg(x)\right)\left(\w_{k+\frac{\g}{2}}(y)|\Gg'(y)|\right)\w_{k+\frac{\g}{2}}\left(\frac{x+y}{2}\right)\left|\Gg'\left(\frac{x+y}{2}\right)\right|\dx\dy\\
={2}\int_{\R}\Q_{0}^{+}\left(\w_{k+\frac{\g}{2}}\Gg,\w_{k+\frac{\g}{2}}|\Gg'|\right)\w_{k+\frac{\g}{2}}|\Gg'|\dx.
\end{multline*}
Therefore, for any $\tilde{\delta} >0$, one has
\begin{align*}
\int_{\R}\Q^{+}_{\g}(\Gg,|\Gg'|) &\,|\Gg'|\,\w_{2k} \dx\\
&\leq {2}\int_{\R}\Q^{+}_{0}(\Gg\w_{k+\frac{\g}{2}},\big[\ind_{[-\tilde{\delta},\tilde{\delta}]} + \ind_{|x|>\tilde{\delta}}\big]|\Gg'|\w_{k+\frac{\g}{2}}) \,|\Gg'|\,\w_{k+\frac{\g}{2}} \dx \\
&\leq {4} \Big(\| \Gg\w_{k+\frac{\g}{2}}\|_{L^2}\|\ind_{[-\tilde{\delta},\tilde{\delta}]} \Gg'\w_{k+\frac{\g}{2}} \|_{L^1} \\
&\qquad\qquad + \| \Gg\w_{k+\frac{\g}{2}}\|_{L^1}\|\ind_{|x| > \tilde{\delta}} \Gg'\w_{k+\frac{\g}{2}} \|_{L^2} \Big)\| \Gg'\,\w_{k+\frac{\g}{2}}\|_{L^2}\,\,
\end{align*}
where we used the known estimates for $\Q^{+}_{0}$ (see Lemma \ref{lem:Q+0}). Consequently, using again Theorem \ref{theo:Unique} and Corollary \ref{L2-weighted} one deduces that there exists $C >0$ such that
$$\int_{\R}\Q^{+}_{\g}(\Gg,|\Gg'|) \,|\Gg'|\,\w_{2k} \dx \leq C\| \Gg'\,\w_{k+\frac{\g}{2}}\|_{L^2}\left(\|\ind_{[-\tilde{\delta},\tilde{\delta}]} \Gg'\w_{k+\frac{\g}{2}} \|_{L^1}  +\|\ind_{|x| > \tilde{\delta}} \Gg'\w_{k+\frac{\g}{2}} \|_{L^2}\right)$$
as soon as $\g \in (0,\g_{\star})$ and $k+\frac{3\g}{2} < 3-\delta$ where we applied Corollary \ref{L2-weighted} to $\| \Gg\w_{k+\frac{\g}{2}}\|_{L^2}$. Now, one has
\begin{equation*}
\|\ind_{[-\tilde{\delta},\tilde{\delta}]} \Gg'\w_{k+\frac{\g}{2}} \|_{L^1} \leq {\sqrt{2\tilde{\delta}}}\,\| \Gg'\w_{k+\frac{\g}{2}}\|_{L^2}\,,
\end{equation*}
whereas, thanks to \eqref{eq:init-gradient-regularity}
\begin{equation*}
  \|\ind_{|x|>\tilde{\delta}} \Gg'\w_{k+\frac{\g}{2}} \|_{L^2} \leq \frac{1}{\tilde{\delta}}\,\| x\Gg'\|_{L^{2}(\w_{k+\frac{\g}{2}})} \leq \frac{C_{1}}{\tilde{\delta}},\qquad  0 \le k<3-\frac{5\gamma}{2}-\delta\,.
\end{equation*}
Thus
\begin{equation*}
\int_{\R}\Q^{+}_{\g}(\Gg,|\Gg'|) \,|\Gg'|\,\w_{2k} \dx\leq C\sqrt{\tilde{\delta}}\| \Gg'\w_{k+\frac{\g}{2}}\|^2_{L^2} + \frac{C_{1}}{\tilde{\delta}}\| \Gg'\w_{k+\frac{\g}{2}}\|_{L^2}\,.
\end{equation*}
The result follows from here using Young's inequality.
\end{proof}
The loss operator is estimated in the following lemma.
\begin{lem}[\textit{\textbf{Loss collisional estimate}}]\phantomsection\label{collision-} {Let $\delta\in(0,\frac12)$} and $\g_\star\in(0,1)$ given by Corollary \ref{L2-weighted}. There exists some positive constant $C >0$, such that, for any {$\tilde{\delta}\in(0,1)$}, $\g \in(0,\g_{\star})$ and  $0\le k <3-\gamma-\delta$ it holds that
\begin{equation}\label{eq:L2-weighted-q-}
2\int_{\R}\Q_{\g}^{-}(\Gg,\Gg')\,\Gg'\,\w_{2k}\dx 
\geq \left(\kappa_{\g}-C\sqrt{\tilde{\delta}}\right)\|\Gg'\w_{k+\frac{\g}{2}}\|_{L^2}^{2}-\frac{C}{\tilde{\delta}^{\frac{5}{2}}} \,\end{equation}
where $\kappa_{\g}$ is defined in Lemma \ref{lem:Sigmag}.
\end{lem}
\begin{proof} One has that
\begin{multline*}
\mathcal{J}:=2\int_{\R}\Q_{\g}^{-}(\Gg,\Gg')\,\Gg'\,\w_{2k}\dx=\int_{\R}\left(\Gg'(x)\right)^{2}\Sigma_{\g}(x)\w_{2k}(x)\dx\\
+\int_{\R^{2}}\Gg'(y)\Gg(x)|x-y|^{\g}\Gg'(x)\w_{2k}(x)\dx\dy=\mathcal{J}_{1}+\mathcal{J}_{2},
\end{multline*}
where the collision frequency $\Sigma_\g$ is defined in \eqref{eq:coll:freq}. The first term $\mathcal{J}_{1}$ is easily estimated using Lemma \ref{lem:Sigmag}, for all $\g \in (0,\g_{\star})$ it holds
$$\mathcal{J}_{1} \geq  \kappa_{\gamma}\| \Gg' \w_{k+\frac{\g}{2}}\|^2_{L^2}.$$
In order to estimate $\mathcal{J}_{2}$, we introduce a smooth cutoff function $0 \leq \chi(x) \leq 1$ with support in the unitary interval $[-1,1]$ and set $\chi_{\tilde{\delta}}(x)=\chi(\tilde{\delta}^{-1}x)$. For any $x \in \R$, it follows then
\begin{multline*}
\int_{\R}\Gg'(y) |x-y|^{\gamma}\dy = \int_{\R}\Gg'(y) \chi_{\tilde{\delta}}(x-y)|x-y|^{\gamma}\dy +   \int_{\R}\Gg'(y) (1-\chi_{\tilde{\delta}}(x-y))|x-y|^{\gamma}\dy\\
=\int_{\R}\Gg'(y) \chi_{\tilde{\delta}}(x-y)|x-y|^{\gamma}\dy  -  \int_{\R} \Gg(y) \partial_y\big[ (1-\chi_{\tilde{\delta}}(x-y))|x-y|^{\gamma} \big]\dy.
\end{multline*}
Notice that, for $\tilde{\delta} \in (0,1)$,
$$\left|\int_{\R}\Gg'(y) \chi_{\tilde{\delta}}(x-y)|x-y|^{\gamma}\dy\right| \leq \tilde{\delta}^{\g}\int_{|x-y|<\tilde{\delta}}|\Gg'(y)|\dy \leq {\sqrt{2\tilde{\delta}}}\|\Gg'\|_{L^2}\,,$$
while
$$\left|\int_{\R} \Gg(y) \partial_y\big[ (1-\chi_{\tilde{\delta}}(x-y))|x-y|^{\gamma} \big]\dy\right| \leq \frac{C}{\tilde{\delta}}\|\Gg\|_{L^1}=\frac{C}{\tilde{\delta}}$$
where we used the fact that
\begin{align*}
\left|\partial_y\big[ (1-\chi_{\tilde{\delta}}(x-y))|x-y|^{\gamma}\big]\right|&=\left|\chi'_{\tilde{\delta}}(x-y)|x-y|^{\g} {-}\g\left(1-\chi_{\tilde{\delta}}(x-y)\right)(x-y)|x-y|^{\g-2}\right|\\
&\leq \|\chi_{\tilde{\delta}}'\|_{L^\infty}|x-y|^{\g}{\ind_{|x-y|\leq \tilde{\delta}}}+\g\left(1-\chi_{\tilde{\delta}}(x-y)\right)|x-y|^{\g-1}\\
&\leq \left(\|\chi'\|_{L^\infty}+\g\right)\tilde{\delta}^{\g-1}, \qquad \tilde{\delta} \in (0,1).\end{align*}
Consequently,
$$\left|\int_{\R}\Gg'(y) |x-y|^{\gamma}\dy\right| \leq  \sqrt{2\tilde{\delta}}\|\Gg'\w_{k}\|_{L^2}+\frac{C}{\tilde{\delta}}\,,$$
and
\begin{equation*}\begin{split}
|\mathcal{J}_{2}| &\leq \left( {\sqrt{2\tilde{\delta}}\|\Gg'\w_{k}\|_{L^2}}+\frac{C}{\tilde{\delta}}\right)\int_{\R}\Gg(x)|\Gg'(x)|\w_{2k}(x)\dx \\
&\leq \left( {\sqrt{2\tilde{\delta}}}\|\Gg'\w_{k}\|_{L^2}+\frac{C}{\tilde{\delta}}\right)\|\Gg \w_{k}\|_{L^2}\,\|\Gg'\w_{k}\|_{L^2}.\end{split}\end{equation*}
Using again Corollary \ref{L2-weighted} to estimate $\|\Gg\w_{k}\|_{L^2}$ for $k+\g<3-\delta$, we deduce using Young's inequality that there exists $C >0$ such that
$$|\mathcal{J}_{2}| \leq C\sqrt{\tilde{\delta}}\| \Gg'\w_{k}\|^2_{L^2} + \frac{C}{\tilde{\delta}^{\frac{5}{2}}}.$$
Since $\|\Gg'\w_{k}\|_{L^2}^{2} \leq \| \Gg'\w_{k+\frac{\g}{2}}\|^2_{L^2}$, the Lemma is deduced from the bounds on $\mathcal{J}_{1}$ and $|\mathcal{J}_{2}|$.
\end{proof} 
We have all in hands, starting from \eqref{gradient-equation} to deduce the following theorem where we introduce the notations
$$ \|f\|_{W^{1,1}} :=\|f\|_{L^{1}}+\|f'\|_{L^{1}}, \qquad 
\|f\|_{H^{1}(\w_k)}^{2}=\|f\|_{L^{2}(\w_{k})}^{2}+\|f'\|_{L^{2}(\w_{k})}^{2}, \quad k \in \R$$
for suitable smooth function $f=f(x)$ with derivative $f'$.
\begin{theo}\label{theo:gradient}
Let  {$\delta\in(0,\frac12)$}. There exists $\g_\star\in(0,1)$  such that, for any  $\g \in(0,\g_\star)$ and {$0\leq k<3-\frac{5\gamma}{2}-\delta$} it holds that \begin{align}\label{eq:gradient}
\| \Gg\|_{H^{1}(\w_k)} + \| \Gg\|_{W^{1,1}}\leq C\,,
\end{align}
for some explicit $C>0$ depending on $\g_{\star}$ but not $\g$ and $k$.  In particular, it holds that
\begin{equation}\label{eq:pointwisebound}
\big| \widehat\Gg(\xi)\big|\leq \frac{C}{\sqrt{1+|\xi|^2}}\,, \qquad \forall \,\xi \in \R.
\end{equation}
%\textcolor{blue}{I changed the right-hand side only for it to be closer to the assumption in Theorem 4.2. Actually, (3.19) readily follows from the equality
%$$\widehat{\Gg}(\xi)=\frac1{1+i\xi} \int_\R (\Gg'(x)+\Gg(x))e^{-ix\xi}\dx.$$}
\end{theo}
\begin{proof} Let us fix  $\g \in(0,\g_\star)$ and $0\leq k<3-\frac{5\gamma}{2}-\delta$, where $\g_\star\in(0,1)$ is given by Corollary \ref{L2-weighted}. {Up to reducing $\gamma_\star$, one may assume that $3-\frac{5\gamma_\star}{2}-\delta>\frac12$.}
Multiply equation \eqref{gradient-equation} by $\Gg'\w_{2k}$ and integrate to obtain 
 {\begin{equation*}
 {\frac38\|\Gg'\w_{k}\|^2_{L^2}-\frac k4 \int_\R|x| \w_{2k-1}(x)(\Gg'(x))^2\dx}
  = 2\int_{\R} \Q_{\g}(\Gg,\Gg') \Gg' w_{2k}\dx\,,
\end{equation*}}
where we used integration by parts and the fact  that $x\partial_x\w_{2k}(x)=2k|x|\w_{2k-1}(x)$ to show that
$$\int_{\R}\partial_x\left(x\Gg'(x)\right) {\Gg'(x)}\w_{2k}(x)\dx= {\frac{1}{2}\int_{\R}\left[\Gg'(x)\right]^{2}\w_{2k}(x)\dx -k\int_{\R}\left[\Gg'(x)\right]^{2}|x|\w_{2k-1}(x)\dx}.$$
Consequently, using Lemmata \ref{collision+} and \ref{collision-}, there exists $C >0$ such that, for any {$\tilde{\delta} \in(0,1)$},  it holds that
\begin{equation*}
\left( \frac38-\frac k4  \right) \big\| \Gg'\w_{k} \|^{2}_{L^2} \leq -\big(\kappa_\gamma - 3C\sqrt{\tilde{\delta}}\big)\| \Gg'\w_{k+\frac{\g}{2}}\|^2_{L^2} + \frac{3C}{\tilde{\delta}^{\frac{5}{2}}}\,.
\end{equation*}
Recalling that $\lim_{\g\to0}\kappa_{\g}=1$, we can fix $\g_{\star}$ sufficiently small, up to reducing our previous $\g_{\star}$, and {$\tilde{\delta} \in(0,1)$} such that $\kappa_\gamma - 3C\sqrt{\tilde{\delta}}>\frac{3}{4}$ for any $\g\in (0,\g_{\star})$. Then, using that $\| \Gg'\w_{k+\frac{\g}{2}}\|^2_{L^2}\ge  \big\| \Gg'\w_{k} \|^{2}_{L^2}$, we conclude that
$$\left( \frac98-\frac k4  \right) \big\| \Gg'\w_{k} \|^{2}_{L^2} \leq {\frac{3C}{\tilde{\delta}^{\frac{5}{2}}}}\,.$$
Finally, since $\frac98-\frac k4>\frac38 $, one deduces that
$$\sup_{\g\in [0,\g_{\star})}\|{\Gg'}\|_{L^{2}(\w_k)}\leq \bar{C}\,,$$
for some constant $\bar{C}$ independent of $k$.  For the $L^{1}$ estimate on the gradient, the H\"older inequality leads to 
$$\|\Gg'\|_{L^{1}}\le \|\Gg'\w_{k}\|_{L^2}\left\|\frac{1}{\w_{k}}\right\|_{L^2} \le C_0,$$
for any $\g\in(0,\g_{\star})$ and any $\frac12<k<3-\frac{5\g}{2}-\delta$. This proves the $W^{1,1}$-estimate. The decay of the Fourier transform of $\Gg$ is a direct consequence of the $W^{1,1}$ bound.
\end{proof}

Since, according to Theorem \ref{theo:gradient}, the family
$\{\Gg\}_{\g\in (0,\g_{\star})}$ is bounded in $H^1(\R)$, we get immediately the following corollary which complements Lemma \ref{lem:Linfty} for small values of $x$:

\begin{cor}\label{cor:hoelder}
 Under the assumption of Theorem~\ref{theo:gradient} there exist some positive constant $C >0$ such that
\begin{equation}\label{eq:Holder}
\sup_{\g \in (0,\g_{\star})}\|\Gg\|_{L^\infty} \leq C\, \qquad \quad \left|\Gg(x)-\Gg(y)\right| \leq C\,|x-y|^{\frac{1}{2}}, \qquad \forall x,y \in \R.\end{equation}
\end{cor}

 The previously obtained bounds on $\Gg$ enable to prove the following pointwise bound for $\widehat{\Gg}$.  
\begin{lem}\label{lem:assum}
  Under the assumption of Theorem~\ref{theo:gradient} there exists some positive constants $\beta, c>0$ such that, for every $\g\in(0,\g_\star)$ and every $\Gg\in\mathscr{E}_{\g}$,
  \begin{equation*}
    \left|\widehat{\Gg}(\xi)\right | \leq (1+c^2|\xi|^2)^{-\frac{\beta}2}  \qquad \xi \in \R.\end{equation*}
  \end{lem}
\begin{proof}
First, it follows from \eqref{eq:pointwisebound} that there exists $R>0$, depending only on $\g_\star$, such that 
\begin{equation}\label{eq:bd_large}
  \left|\widehat{\Gg}(\xi)\right | \leq (1+|\xi|^2)^{-\frac14}, \qquad |\xi|\ge R.
 \end{equation}
Following the same lines as in the proof of  \cite[Lemma 3]{ADVW}, we shall now show that there exists $C'>0$ depending only on $\g_\star$ such that
\begin{equation}\label{eq:advw}
    \left|\widehat{\Gg}(\xi) \right |\le 1-C'\min\{1,|\xi|^2\}, \qquad \xi\in\R.
\end{equation}
Let us fix $\xi\in\R_*$. There exists $\theta\in[0,2\pi]$ such that
$$1-\left|\widehat{\Gg}(\xi) \right |=\int_{\R} \Gg(x) (1-\cos(x\xi+\theta))\d x
= 2\int_{\R}\Gg(x)\sin\left(\frac{x\xi+\theta}2\right)\dx$$
Now for any $r>0$ and any $\varepsilon\in(0,\frac{\pi}2)$, we have
\begin{align*}
  1-\left|\widehat{\Gg}(\xi) \right |& \ge 2\sin^2\varepsilon \int_{|x|\le r,\forall k\in\Z,|x\xi+\theta-2k\pi|\ge 2\varepsilon}\Gg(x)\dx \\
  & \ge 2\sin^2\varepsilon \left\{\int_{\R}\Gg(x)\dx -\frac{M_2(\Gg)}{r^2}-\int_{K_{\varepsilon, r}}\Gg(x)\dx\right\}
  \end{align*}
where
$$K_{\varepsilon, r}:=\{x\in \R \;; \;|x|\le r \;\mbox{ and } \;\exists k\in\Z,|x\xi+\theta-2k\pi|\le 2\varepsilon\}.$$
We now deduce from \eqref{eq:mom}, Proposition \ref{theo:energy0} and Corollary \ref{eq:Holder} that
\begin{equation}\label{eq:below}
  1-\left|\widehat{\Gg}(\xi) \right |\ge 2\sin^2\varepsilon \left\{1 -\frac1{2r^2}- C |K_{\varepsilon, r}|\right\},\end{equation}
for some $C$ depending only on $\g_\star$. Now, since 
$$|K_{\varepsilon, r}| \le \left(\frac{|\xi|r}{\pi}+1\right)\frac{4\varepsilon}{|\xi|} =  4\varepsilon \left(\frac{r}{\pi}+\frac1{|\xi|}\right),$$
one sees that for $|\xi|\ge1$, choosing $r=1$ and
$\varepsilon= \frac{\pi}{16C(r+\pi)}$, one obtains 
\begin{equation}\label{eq:large}
  \left|\widehat{\Gg}(\xi) \right |\le 1- C'_1, \qquad |\xi|\ge 1,
  \end{equation}
with $C'_1=\frac12 \sin^2\varepsilon$. Next, for $|\xi|\le 1$, we set $\varepsilon =\mu |\xi|$ in \eqref{eq:below} for some $\mu\in(0,1]$ to be fixed. We thus obtain
$$ 1-\left|\widehat{\Gg}(\xi) \right |\ge \xi^2  2 \mu^2 \inf_{|\sigma|\le 1}\frac{\sin^2\sigma}{\sigma^2}  \left\{1 -\frac1{2r^2}- 4\mu C\left(\frac{r}{\pi}+1\right) \right\}.$$
Thus, choosing $r=1$ and $\mu=\min\left\{1, \frac{\pi}{16 C (r+\pi)}\right\}$, we deduce that
\begin{equation}\label{eq:small}
  \left|\widehat{\Gg}(\xi) \right |\le 1-C'_2 \xi^2, \qquad |\xi|\le 1,
  \end{equation}
with $C'_2=\frac12 \mu^2 \inf_{|\sigma|\le 1}\frac{\sin^2\sigma}{\sigma^2}.$ Finally, choosing $C'=\min\{C'_1,C'_2\}$, \eqref{eq:advw} follows from \eqref{eq:large} and \eqref{eq:small}. To complete the proof of Lemma \ref{lem:assum} it suffices to notice that \eqref{eq:advw} implies that
\begin{equation}\label{eq:bd_small}
  \left|\widehat{\Gg}(\xi) \right | \le 1-C' \xi^2 \le \frac1{1+C'\xi^2}, \qquad |\xi|\le 1,
  \end{equation}
and
\begin{equation}\label{eq:bd_medium}
\left|\widehat{\Gg}(\xi) \right | \le 1-C' \le \frac1{1+C'} \le \frac1{1+\frac{C'}{R^2} \xi^2}, \qquad 1\le |\xi|\le R.
\end{equation}
We then deduce from \eqref{eq:bd_large}, \eqref{eq:bd_small} and \eqref{eq:bd_medium} that choosing
$$\beta=\frac12 \qquad \mbox{ and }\qquad c=\min\left\{1,\sqrt{C'},\frac{\sqrt{C'}}{R}\right\}$$ leads to the desired estimate.
\end{proof}

One has the following estimate for differences of two equilibrium solutions.
\begin{lem}\phantomsection\label{lem:Diffmom} Let  {$\delta\in(0,\frac12)$} and $\g_\star\in(0,1)$ given by Corollary \ref{L2-weighted}. Let $\g \in (0,\g_{\star})$ and $\Gg^{1},\Gg^{2}\in \mathscr{E}_{\g}$ be given. For any  {$2< k<3-\gamma-\delta$}, there exists $\g_{\star}(k) >0$ and  $C_{k} >0$ such that
$$\|\Gg^{1}-\Gg^{2}\|_{L^{1}(\w_{k+\g})} \leq C_{k}\|\Gg^{1}-\Gg^{2}\|_{L^{1}(\w_{\g+\frac{2k}{3}})} \qquad \forall \g \in (0,\g_{\star}(k)).$$
\end{lem}
\begin{proof} Let $\g_\star\in(0,1)$ given in Theorem \ref{theo:Unique} and $\g\in(0,\g_\star)$. We introduce $g_{\g}=\Gg^{2}-\Gg^{1}$ and observe that
\begin{equation}\label{eq:diffe}
\frac{1}{4}\partial_{x}\left(xg_\g(x)\right)=\Q_{\g}(g_\g,\Gg^{2})+\Q_{\g}(\Gg^{1},g_{\g}).\end{equation}
We multiply then \eqref{eq:diffe} by $\mathrm{sign}(g_{\g})|x|^{k}$ and integrate over $\R$ to deduce
\begin{equation*}\begin{split}
-\frac{k}{4}M_{k}\left(|g_{\g}|\right)&=\int_{\R}\left[\Q_{\g}(g_{\g},\Gg^{2})+\Q_{\g}(\Gg^{1},g_{\g})\right]\mathrm{sign}(g_{\g}(x))|x|^{k}\dx\\
&=\int_{\R^{2}}g_{\g}(x)\bm{S}_{\g}(y)|x-y|^{\g}
\bigg[{2}\,\mathrm{sign}\left(g_{\g}\left(\frac{x+y}{2}\right)\right)\left|\frac{x+y}{2}\right|^{k}\\
&\phantom{+++++}-\mathrm{sign}(g_{\g}(x))|x|^{k}-\mathrm{sign}(g_{\g}(y))|y|^{k}\bigg]\dx\dy\\
&\leq -\int_{\R}\sigma_{\g}(x)|g_{\g}(x)|\,|x|^{k}\dx + \int_{\R^{2}}|g_{\g}(x)|\,\bm{S}_{\g}(y)|x-y|^{\g}|y|^{k}\dy\dx\\
&\phantom{+++}+ {2}\int_{\R^{2}}|g_{\g}(x)|\,\bm{S}_{\g}(y)|x-y|^{\g}\left|\frac{x+y}{2}\right|^{k}\dx\dy
\end{split}\end{equation*}
where
$$\bm{S}_{\g}=\frac{1}{2}\left(\Gg^{2}+\Gg^{1}\right), \qquad \sigma_{\g}(x)=\int_{\R}\bm{S}_{\g}(y)|x-y|^{\g}\dy, \qquad x \in \R.$$
 {Notice that, for $k<3$ it holds that
\begin{equation}\label{ineqxplusy}\begin{split}
\Big| \frac{x+y}{2} \Big|^{k}&= 2^{-k}| x^3 + 3x^{2}y + 3xy^2 + y^3 |^{\frac{k}{3}}\\
&\leq 2^{-k}\big( | x |^k + 3|x|^{\frac{2k}{3}}|y|^{\frac{k}{3}} + 3|x|^{\frac{k}{3}}|y|^{\frac{2k}{3}} + |y|^k\big) \,\qquad \forall (x,y) \in \R^{2}. 
\end{split}\end{equation}
Then, with this inequality, one deduces that}
\begin{multline*}
-\frac{k}{4}M_{k}\left(|g_{\g}|\right) \leq  {- {(1-2^{1-k})}}\int_{\R}\sigma_{\g}(x)|g_{\g}(x)|\,|x|^{k}\dx \\
+\left(1+ {2^{1-k}}\right)\,\left[M_{\g}(|g_{\g}|)M_{k}(\bm{S}_{\g})+M_{0}(|g_{\g}|)M_{k+\g}(\bm{S}_{\g})\right]\\
+ {6}\left[M_{\g+\frac{2k}{3}}(|g_{\g}|)M_{\frac{k}{3}}(\bm{S}_{\g})+M_{\frac{2k}{3}}(|g_{\g}|)M_{\frac{k}{3}+\g}(\bm{S}_{\g})\right.\\
\phantom{++++}\left.+M_{\frac{k}{3}+\g}(|g_{\g}|)M_{\frac{2k}{3}}(\bm{S}_{\g}) + M_{\frac{k}{3}}(|g_{\g}|)M_{\frac{2k}{3}+\g}(\bm{S}_{\g})\right].
\end{multline*}
Theorem \ref{theo:Unique} ensures that there exists $C>0$ independent of $k$ such that  $\|\bm{S}_{\g}\|_{L^{1}(\w_{k+\g})} \le C$ for any $\g\in (0,\g_{\star})$ and $k+\gamma\in(0,3-\delta)$. Using this bound and estimating every moment of $|g_{\g}|$ by $\|g_{\g}\|_{L^{1}(\w_{\frac{2k}{3}+\g})}$, yields
\begin{equation}\label{eq:Mk|g|}-\frac{k}{4}M_{k}\left(|g_{\g}|\right)
 + { {(1-2^{1-k})}}\int_{\R}\sigma_{\g}(x)|g_{\g}(x)|\,|x|^{k}\dx  \leq C'\|g_{\g}\|_{L^{1}(\w_{\g+\frac{2k}{3}})}\end{equation}
 for some suitable $C' >0$ depending neither on $k$ nor on $\g$. Since, according to Theorem \ref{theo:Unique}, $\sup_{\g\in (0,\g_{\star})}\|\bm{S}_{\g}\|_{L^{2}} \leq C$, one checks easily that $\sigma_{\g}$ satisfies a bound as in Lemma \ref{lem:Sigmag}, i.e. 
\begin{equation*}\sigma_{\g}(y) \geq \overline{\kappa}_{\g}\,\w_{\g}(y) \end{equation*}
for some explicit $\overline{\kappa}_{\g}$ with 
$\lim_{\g\to0}\overline{\kappa}_{\g}=1.$   We can then choose $\g_{\star}(k)$ small enough such that 
$$\overline{\kappa}_{\g}\left(1-{2^{1-k}}-\frac{k}{4\overline{\kappa}_{\g}} \right) \geq \frac{\sigma_{k}}{2} \qquad \forall \g \in \big(0,\g_{\star}(k)\big)$$ 
with ${\sigma_{k}:= {1-2^{1-k}-\frac{k}{4}} >0}$  {for any $k\in(2,3)$}. Inserting this in \eqref{eq:Mk|g|}, one obtains
{$$\int_{\R} |g_{\g}(x)| \w_{\g}(x) |x|^{k} \dx \leq \frac{2C'_{k}}{\sigma_{k}}\|g_{\g}\|_{L^{1}(\w_{\g+\frac{2k}{3}})} \quad \quad \forall\, \g \in \big(0,\g_{\star}(k)\big).$$
It readily follows that 
\begin{equation*}\begin{split}
\|g_{\g}\|_{L^{1}(\w_{k+\g})} & \le  2^{k} \int_{\R} |g_{\g}(x)| \w_{\g}(x) (1
+ |x|^{k} ) \dx \le  2^{k} \|g_{\g}\|_{L^{1}(\w_{\g})}+ \frac{2^{k+1} C'_{k}}{\sigma_{k}}\|g_{\g}\|_{L^{1}(\w_{\g+\frac{2k}{3}})} 
\end{split}\end{equation*}
which leads to the desired estimate with $C_{k}=2^{k}+\frac{2^{k+1}C'_{k}}{\sigma_{k}}.$}
\end{proof}

\section{Stability and uniqueness}\label{sec:sec3}
We are now in position to quantify first the stability of the profile $\bm{G}_{0}$ in the limit $\g \to 0^{+}$ and deduce the uniqueness of the steady profile $\Gg \in \mathscr{E}_{\g}$ for $\g>0$ sufficiently small.  Assumptions and notations of Theorem  \ref{theo:Unique} are in force along this section without further mention.
 
\subsection{A mathematical toolbox from the Maxwell case $\g=0$}\label{sec:tool} We recall in this section results from our companion paper \cite{maxwel} which is devoted to the study of the Maxwell equation 
\begin{equation}
 \label{eq:IB-selfsim}
  \p_t g = -\frac14 \p_x (xg) + \Q_{0}(g,g), \qquad g(0,\cdot)=f_{0} \in L^{1}_{2}(\R)
\end{equation}
where, as said in the introduction,  the parameter $c=\frac{1}{4}$ is the only which provides energy conservation. Thus, it holds at least formally that
\begin{equation}
  \label{eq:normalisation-g}
  \int_{\R} g(t,x)\left(\begin{array}{c}1 \\x \\x^{2}\end{array}\right) \dx = \int_{\R}f_{0}(x)\left(\begin{array}{c}1 \\x \\x^{2}\end{array}\right)\d x=\left(\begin{array}{c}1 \\0 \\ E\end{array}\right) \qquad \forall t \geq0
\end{equation}
for $E >0$. As recalled in Theorem \ref{theo:bob}, the unique steady solution with unit mass, zero momentum and energy $E >0$ is given by 
$$H_\lambda(x)=\lambda \bm{H}(\lambda x), \qquad \lambda=\frac{1}{\sqrt{E}} >0, \qquad\,x\in \R.$$
If we define the Fourier transform of
$g$ as
\begin{equation*}
  \varphi(t, \xi) := \int_{\R} g(t,x) e^{-ix \xi} \dx, \qquad \xi\in\R
\end{equation*}
then $\varphi(t,\xi)$ satisfies
\begin{equation} 
  \label{eq:selfsim-fourier}
  \p_t \varphi(t,\xi) = \frac14 \,\xi  \,\p_\xi \varphi(t,\xi)
  + \varphi\Big(t, \frac{\xi}{2} \Big)^2 - \varphi(t,\xi), \qquad t \geq0
\end{equation}
with initial condition $  \varphi(0,\xi)=\varphi_{0}(\xi)=\widehat{f}_{0}(\xi)$. Due to \eqref{eq:normalisation-g}, $\varphi$ satisfies for all
$t \geq 0$ that
\begin{equation}
  \label{eq:normalisation-varphi}
  \varphi(t,0) = 1,
  \qquad
  \p_\xi \varphi(t,0) = 0,
  \qquad
  \p_\xi^2 \varphi(t,0) = -E.
\end{equation}
The convergence to equilibrium for the solution $g(t)$ to $H_{\lambda}$ can be deduced from the convergence of $\varphi(t)$ towards $\widehat{H}_{\lambda}$ in a suitable Fourier norm (see Appendix \ref{app:fourier} for details) as established in \cite[Theorem 1.3 and Remark 2.1]{maxwel}.
 \begin{theo}\phantomsection\label{k-norm-cvgce}
  Assume that $g = g(t,x)$ is a nonnegative solution to
  \eqref{eq:IB-selfsim} with the normalisation
  \eqref{eq:normalisation-g}, and set $\varphi = \varphi(t,\xi)$ its
  Fourier transform in the $x$ variable. Then, for $0 \leq k < 3$, and $t \geq 0$,
  \begin{equation*}
   \vertiii{\varphi(t) - \widehat{H}_{\lambda}}_{k} \leq e^{-\sigma_{k} t} \vertiii{\varphi_0 -  \widehat{H}_{\lambda}}_{k}
    \qquad
    \text{with}  \qquad \sigma_{k} := 1 - \frac14 k - 2^{1-k}\,.
  \end{equation*}
  In particular, $g(t)$ converges exponentially to $H_{\lambda}$ in the $k$-Fourier
  norm for any $2 < k < 3$. 
 \end{theo}
 {The propagation of regularity for solution to \eqref{eq:IB-selfsim} has been obtained in \cite{maxwel} and \cite{FPTT}. We deduce from \cite[Theorem 1.4 and Remark 2.1]{maxwel} the following result.
\begin{theo}[\textit{\textbf{Sobolev norm propagation and relaxation}}]\phantomsection\label{theo:Sobolev}
Let $g(t)=g(t,x)$ be a solution to the Boltzmann problem \eqref{eq:IB-selfsim}-\eqref{eq:normalisation-g} with $E\in(0,1)$ and initial condition $g_{0}(x)=g(0,x)$ satisfying %\textcolor{red}{Define $\Psi$?} \textcolor{magenta}{I replaced $\Psi_\beta$ by the definition since it only appears here!}
% $$|\widehat{g_0}(\xi) | \leq \Psi_{\beta}(c\,|\xi|), \qquad \xi \in \R$$
$$|\widehat{g_0}(\xi) | \leq (1+c^2|\xi|^2)^{-\frac{\beta}{2}}, \qquad \xi \in \R$$
for some $\beta,\,c>0$ and $g_0 \in H^{\ell}(\mathbb{R})$ for $\ell\geq0$.  Then, for $\frac{5}{2} < k < {3}$ and any $0<\sigma < \frac98 - \frac14 k - 2^{\frac{3}{2} - k}$ one has that
\begin{equation}\label{H-relax}
  \| g(t) - H_{\lambda}\|_{H^{\ell}} \leq \lambda^{\ell+\frac12} e^{-\sigma t}\,\Big( \| g_0 - H_{\lambda}\|_{H^{\ell}} +  \lambda^k C(\sigma,\ell,k)\,\| g_0 - H_{\lambda}\|_{ L^{1}(\w_{k}) } \Big)\,,
\end{equation}
for some positive constant $C(\sigma,\ell,k) >0$ depending only on $\ell,k$ and $\sigma.$
\end{theo}}
%For simplicity, in the rest of this section, we will focus only in the  particular case in which $E=1$ for which the unique steady solution is given by $\bm{H}$.  The propagation of regularity for solution to \eqref{eq:IB-selfsim} has been obtained in \cite{maxwel} and \cite{FPTT}.
%\begin{theo}[\textit{\textbf{Sobolev norm propagation and relaxation}}]\phantomsection\label{theo:Sobolev}
%Let $g(t)=g(t,x)$ be a solution to the Boltzmann problem \eqref{eq:IB-selfsim}-\eqref{eq:normalisation-g} with $E=1$ and initial condition $g_{0}(x)=g(0,x)$ satisfying %\textcolor{red}{Define $\Psi$?} \textcolor{magenta}{I replaced $\Psi_\beta$ by the definition since it only appears here!}
% $$|\widehat{g_0}(\xi) | \leq \Psi_{\beta}(c\,|\xi|), \qquad \xi \in \R$$
%$$|\widehat{g_0}(\xi) | \leq (1+c^2|\xi|^2)^{-\frac{\beta}{2}}, \qquad \xi \in \R$$
%for some $\beta,\,c>0$ and $g_0 \in H^{\ell}(\mathbb{R})$ for $\ell\geq0$.  Then, for $\frac{5}{2} < k < {3}$ and any $0<\sigma < \frac98 - \frac14 k - 2^{\frac{3}{2} - k}$ one has that
%\begin{equation}\label{H-relax}
%  \| g(t) - \bm{H}\|_{H^{\ell}} \leq e^{-\sigma t}\,\Big( \| g_0 - \bm{H}\|_{H^{\ell}} +  C(\sigma,\ell,k)\,\| g_0 - \bm{H}\|_{ L^{1}(\w_{k}) } \Big)\,,
%\end{equation}
%for some positive constant $C(\sigma,\ell,k) >0$ depending only on $\ell,k$ and $\sigma.$
%\end{theo}
%
For the study of the stability of solution $\Gg$ around the steady state
$$\bm{G}_{0}(x)=\lambda_{0}\bm{H}(\lambda_{0}x), \qquad x \in \R$$
with $\lambda_{0} >0$ described in Theorem \ref{theo:Unique}, we define in the following the linearization $\mathscr{L}_{0}$ of $\Q_{0}$ around $\bm{G}_{0}$ on $L^{1}(\w_{a})$.
\begin{defi}\label{def:spaces} We introduce, for $2 < a <3$ the functional spaces  
\begin{equation*}\begin{split}
\mathbb{X}_{a}=L^{1}(\w_{a}), \qquad {\mathbb{Y}}_{a}&=\left\{f \in \mathbb{X}_{a}\;;\;\int_{\R}f(x)\d x=\int_{\R}f(x) x\d x=0\right\}\\
\text{ and } \qquad {\mathbb{Y}}_{a}^{0}&=\left\{f \in {\mathbb{Y}_{a}}\;;\;\int_{\R}f(x) x^{2}\d x=0\right\}\end{split}
\end{equation*}
with $\|\cdot\|_{\X_{a}}$ denoting the norm in $\X_{a}$.
We define then the linearization of $\Q_{0}$ around the steady state  $\bm{G}_{0}$ given in Theorem~\ref{theo:Unique} as
$$\mathscr{L}_{0}:\mathscr{D}(\mathscr{L}_{0}) \subset \mathbb{X}_{a}\to \mathbb{X}_{a}, \qquad \mathscr{D}(\mathscr{L}_{0})=\left\{f \in \mathbb{X}_{a}\;;\;\partial_{x}(x f) \in \mathbb{X}_{a}\right\}$$
with $$\mathscr{L}_{0}(h)=2\Q_{0}(h,\bm{G}_{0})-\frac{1}{4}\partial_{x}(xh), \qquad \forall h \in \mathscr{D}(\mathscr{L}_{0})\,.$$
\end{defi}
The stability of the profile $\bm{G}_{0}$ for the Maxwell molecules case is established through the following result which provides a spectral gap for the linear operator $\mathscr{L}_{0}$ in the space $\mathbb{Y}_{a}^{0}$ (see \cite[Proposition 4.8]{maxwel} for a proof of this proposition). 
\begin{prp}\label{restrict0}
  Let $2<a<3$. The operator $\left(\mathscr{L}_{0},\mathscr{D}(\mathscr{L}_{0})\right)$ on $\mathbb{X}_{a}$ is such that, for any $ {\nu \in(0,1-\frac{a}{4} -2^{1-a})}$, there exists $C(\nu)>0$ such that
\begin{equation}\label{eq:invert0} 
 \|\mathscr{L}_{0}h\|_{\X_{a}}\ge \frac{\nu}{C(\nu)} \|h\|_{\X_{a}}, \qquad \forall h \in \mathscr{D}(\mathscr{L}_{0}) \cap \mathbb{Y}_{a}^{0}.\end{equation}
In particular, the restriction $\widetilde{\mathscr{L}}_{0}$ of $\mathscr{L}_{0}$ to the space $\mathbb{Y}_{a}^{0}$  
is invertible with
\begin{equation}\label{eq:invert}\left\|\widetilde{\mathscr{L}}_{0}^{-1}g\right\|_{\X_{a}} \leq \frac{C(\nu)}{\nu}\|g\|_{\X_{a}}, \qquad \forall g \in \mathbb{Y}_{a}^{0}.\end{equation}
\end{prp}
Besides the linearized operator $\mathscr{L}_{0}$, we also need to introduce the functional
\begin{equation}\label{eq:mathlo}
\mathscr{I}_{0}(f,g)=\int_{\R^{2}}f(x)g(y)|x-y|^{2}\log|x-y|\dx\dy, \qquad f,g \in L^{1}(\w_{a}), \qquad a>2\,.\end{equation}
The following properties  play a crucial role in our analysis, we refer to \cite[Lemmas 4.9, 4.10 and 4.11]{maxwel} for a proof. 
\begin{lem}\label{rmk:varphi0} The functional $\mathscr{I}_{0}$ enjoys the following properties:
\begin{enumerate}[a)]
\item It holds
$$\lambda_{0}= {\exp \left(\frac12 \mathscr{I}_{0}(\bm{H},\bm{H})\right)=2\sqrt{ e}}.$$
\item Define
$g_{0}(x)=\frac{2}{\pi}\frac{1-3x^{2}}{(1+x^{2})^{3}},$ and $\varphi_{0}(x)=g_{0}(\lambda_{0}x)$ $(x \in \R)$.
Then, for any $2 < a < 3$, 
$$\varphi_{0} \in  \mathrm{Ker}(\mathscr{L}_{0}) \cap \mathbb{Y}_{a}, \qquad M_{2}(\varphi_{0}) \neq 0 \qquad \text{ and } \quad \mathscr{I}_{0}(\varphi_{0},\bm{G}_{0}) \neq 0.$$
\item If $\varphi \in \mathrm{Ker}(\mathscr{L}_{0}) \cap \mathbb{Y}_{a}$ then 
$$\mathscr{I}_{0}(\varphi,\bm{G}_{0})=0 \implies  M_{2}(\varphi)=0.$$
In particular, in such a case, $\varphi=0.$
\end{enumerate}
\end{lem}

\subsection{Stability of the profile -- upgrading the convergence}\label{sec:upgrade}
The results of Section \ref{sec:apost} ensure the convergence (in a weak-$\star$ sense) of $\bm{G}_{\g}$ towards $\bm{G}_{0}$ as $\g \to 0.$ We upgrade here the convergence to the (strong) $L^{1}(\w_{a})$ topology and, more importantly, provide also a \emph{quantitative} estimate of $\|\Gg-\bm{G}_{0}\|_{L^{1}(\w_{a})}.$ To do so, we {will resort to a comparison between the collision operator $\Q_{\g}$ and the operator $\Q_{0}$ (corresponding to Maxwellian interactions) given in Proposition \ref{diff_Q_v2} in  Appendix \ref{app:QgQ0}.} 
Let us denote by ${\mathcal N}_0(f)$ the self-similar operator associated to the Maxwellian case $\gamma=0$, that is
$$  {\mathcal N}_0(f) = -\frac14 \partial_x(xf)+\Q_{0}(f,f).$$
Let us denote by ${\mathcal N}_\gamma(f)$ the self-similar operator associated to the general case $\gamma>0$, that is
$$  {\mathcal N}_\gamma(f) = -\frac{1}{4} \partial_x(xf)+\Q_\gamma(f,f).$$

\begin{lem}\phantomsection\label{N0_Ggamma}
  Let $2<a<3$ and $\delta>0$ such that $a<3-\delta$. Let $\g_{\star} \in (0,1)$ be defined in Corollary \ref{L2-weighted} (notice that $\g_{\star}$ depends on $\delta$ and thus on $a$). For any $\gamma \in (0,\g_{\star}),s>0$ satisfying  $s+\gamma+a<3-\delta$, there exists $C_{0} >0$ depending only on $s$ {and $\delta$}  such that 
 $$  \|{\mathcal N}_0(\Gg)\|_{L^1(\w_{a})}\le { C_{0}\g\left(1+|\log\g|\right)}$$
 {holds true for any profile $\Gg \in \mathscr{E}_{\g}.$}
 \end{lem}
\begin{proof}
Since ${\mathcal N}_\gamma(\Gg)=0$, one has that
$$
\|{\mathcal N}_0(\Gg)\|_{L^1(\w_{a})}  = \|{\mathcal N}_0(\Gg)-{\mathcal N}_\gamma(\Gg)\|_{L^1(\w_{a})}  \le   \|\Q_{0}(\Gg,\Gg)-\Q_\gamma(\Gg,\Gg)\|_{L^1(\w_{a})}\,.
$$
Noticing that, according to  {\eqref{eq:estimGg} and \eqref{eq:L2-weighted}, there exists $C>0$ such that, for any $\gamma \in (0,\g_{\star}),s>0$ satisfying  $s+\gamma+a<3-\delta$,
  $$\max\left(\|\Gg\|_{L^1(\w_{a+s+\g})},\|\Gg\|_{L^1(\w_{a})}\right) \le C,\qquad \text{ and } \quad \|\Gg\|_{L^2(\w_{a})}\le C,$$}
the result then follows from  Proposition \ref{diff_Q_v2}.\end{proof}

We introduce here the following steady state of $\mathcal{N}_{0}$ with the same mass, momentum and energy of $\Gg$, namely
$${\bm{h}_{\g}}(x)=H_{\lambda_{\g}}(x)=\lambda_{\g} {\bm{H}}(\lambda_{\g}x), \qquad \qquad \lambda_{\g}=\frac{1}{\sqrt{M_{2}(\Gg)}}, \qquad \g \in (0,1).$$
Since $\lim_{\g \to0^{+}}M_{2}(\Gg)=M_{2}(\bm{G}_{0})$, we have that
$$\lim_{\g \to 0^{+}}\lambda_{\g}=\lambda_{0}$$
and, noticing that
$$\left|\bm{h}_{\g}(x)-\bm{G}_{0}(x)\right| \leq C\left|\lambda_{\g}-\lambda_{0}\right|\bm{G}_{0}(x),$$
for some $C$ that can be made independent on $\g$.  Consequently, it follows that
\begin{equation}\label{stab} \|\bm{h}_{\g}-\bm{G}_{0}\|_{L^{1}(\w_{a})} \leq C_{a}\left|\lambda_{\g}-\lambda_{0}\right| \qquad \forall a \in (0,3).
\end{equation}
To compare then $\Gg$ to $\bm{G}_{0}$, it is enough to compare $\Gg$ to $\bm{h}_{\g}.$ This is the object of the following
proposition.
\begin{prp}\phantomsection\label{lem:stabil}
Let $2<a<3$. There exist $\g_\star\in(0,1)$  and  an \emph{explicit} function $\eta=\eta(\gamma)$ depending on $a$, with $\lim_{\g\to0^+}\eta(\gamma)= 0$, such that, for any $\g\in(0,\g_\star)$ and any $\Gg\in\mathscr{E}_{\g}$, $$ \|\Gg-\bm{h}_{\g}\|_{L^1(\w_{a})} \le \eta( \gamma).$$
 \end{prp}
 
\begin{proof} Let us denote by $g(t,x)$ the solution to the Maxwell equation \eqref{eq:IB-selfsim} with initial condition $\Gg$. Then, for every $t\ge 0$, 
\begin{equation}\label{diff_G}
  \|\Gg-\bm{h}_{\g}\|_{L^1(\w_{a})} \le \|\Gg-g(t)\|_{L^1(\w_{a})}+\|g(t)-\bm{h}_{\g}\|_{L^1(\w_{a})}.
  \end{equation}
In order to obtain a bound for $\|g(t)-\bm{h}_{\g}\|_{L^1(\w_{a})}$, we shall use the convergence of $g(t)$ towards $\bm{h}_{\g}$ as $t\to \infty$ given in Fourier norm by Theorem \ref{k-norm-cvgce}. Choosing 
$$ {a<a_{*} <3}, \qquad 0 < \alpha < \frac{2(a_*-a)}{2a_*+1}\,,$$
it follows from Lemma  \ref{L2-knorm} that, for any $\beta >0$ and $0 < r <1$, 
\begin{equation}\begin{split}\phantomsection
\|g(t)-\bm{h}_{\g}\|_{L^1(\w_{a})} &\le  C \|g(t)-\bm{h}_{\g}\|^\alpha_{L^2} \left(\|g(t)\|^{1-\alpha}_{L^1(\w_{a_*})} + \|\bm{h}_{\g}\|^{1-\alpha}_{L^1(\w_{a_*})} \right) \nonumber\\
& \le  C_{r,\beta,\alpha} \vertiii{\widehat{g}(t)-\widehat{\bm{h}_{\g}}}_{{a}}^{\alpha(1-r)}\left(\|g(t)\|^{1-\alpha}_{L^1(\w_{a_*})} + \|\bm{h}_{\g}\|^{1-\alpha}_{L^1(\w_{a_*})} \right)\\
& \phantom{++++}\times  {\left(\|g(t)\|_{H^N}^{r}+\|\bm{h}_{\g}\|_{H^N}^{r}\right)^\alpha} \label{eq:alignL2HM}
\end{split}\end{equation}
for an explicit constant $C_{r,\beta,\alpha}$ depending on $\alpha,r,\beta$ and where $N=\frac{(1-r)(2a+\beta+1)}{2r}$. {Notice that, choosing $r$ as close as desired from $1$, we can assume $N \leq 1$.} {Note that this assumption shall be used to deduce a bound of $\|\Gg\|_{H^N}$ from Theorem \ref{theo:gradient}.} 
Observing  that, for any $m\in\R^+$,
$$\|\bm{h}_{\g}\|_{H^{m}}^{2}=\int_{\R}\left(1+|\xi|^{2}\right)^{m}\left(1+\lambda_{\g}^{-1}|\xi|\right)^{2}\exp\left(-\frac{2}{\lambda_{\g}}|\xi|\right)\d\xi$$
where $\lambda_{\g}$ is bounded from below for $\g$ small enough (recall that $\lim_{\g\to0^{+}}\lambda_{\g}=\lambda_{0}$).  In turn, one easily checks that
$$\sup_{\g \in (0,\g_{\star})}\|\bm{h}_{\g}\|_{H^{m}} < \infty$$
for any $m \in \R^{+}$, since $a_{*} < 3$, $\sup_{\g\in (0,\g_{\star})}\|\bm{h}_{\g}\|_{L^{1}(\w_{a_{*}})} < \infty$. Therefore, there is $C_{0} >0$ (depending on $r,\alpha,\beta$ and $a_{*}$ but not on $\g$) such that
\begin{equation}\label{eq:gtHg}
  \|g(t)-\bm{h}_{\g}\|_{L^{1}(\w_{a})} \leq C_{0} \vertiii{\widehat{g}(t)-\widehat{\bm{h}_{\g}}}_{{a}}^{\alpha(1-r)}\left(1+\|g(t)\|_{L^{1}(\w_{a_{*}})}^{1-\alpha}\right)\left(1+\|g(t)\|_{H^{N}}^{ r}\right).\end{equation}
Let $\delta\in(0,3)$ such that $a_{*}<3-\delta$. Let $\g_\star$ be such that the results of Theorem \ref{theo:Unique}, Corollary \ref{L2-weighted} and Theorem \ref{theo:gradient} hold and such that $a+\g_\star<3-\delta$. By virtue of Theorem \ref{theo:gradient}  {and Lemma \ref{lem:assum}, the initial datum $g(0)=\Gg$ is such that (recall that $N \leq 1$),
  \begin{equation*}\label{eq:GgSobolev}
  \sup_{\g \in (0,\g_\star)}\|\Gg\|_{H^N} < \infty \qquad \text{ as well as } \qquad    \left|\widehat{\Gg}(\xi)\right | \leq (1+c^2|\xi|^2)^{-\frac14} \qquad \xi \in \R\,,\end{equation*}
for some $c>0$ depending only on $\g_\star$.}
 {We may now apply Theorem \ref{theo:Sobolev} with $\ell=N$, $\beta=\frac12$, $k=3-\delta$ and deduce the existence of $C >0$ such that}
\begin{equation}\label{eq:g_sobolev}
  \|g(t)\|_{H^{N}} \leq C, \qquad \qquad \forall t\geq0, \qquad \g \in (0,\g_{\star})\,.\end{equation}
Let us now show that we also have a uniform bound with respect to $t$ and $\g$ for  $\|g(t)\|_{L^1(\w_{a_{*}})}$ .
First, for $2<k<3$, one has
$$  \frac{\d}{\d t}\int_{\R}|x|^k g(t,x)\dx + \left(1-\frac{k}{4}\right) \int_{\R}|x|^k g(t,x)\dx =  \int_{\R^2}g(t,x)g(t,y) \,\left|\frac{x+y}{2}\right|^k \dx \dy.$$
{Using \eqref{ineqxplusy}, }we  deduce  that
\begin{multline*}
\frac{\d}{\d t}\int_{\R}|x|^k g(t,x)\dx + \left(1-\frac{k}{4}-2^{1-k}\right) \int_{\R}|x|^k g(t,x)\dx \\
  \le  3 \times 2^{1-k} \int_{\R}g(t,x) |x|^{\frac{2k}3}\dx  \int_{\R} g(t,y) |y|^{\frac{k}{3}}\ \dy 
    \le 3 \times 2^{1-k} M_2(\Gg)^{\frac{k}{2}},
\end{multline*}
since $k<3$. It thus follows that, for any $t\ge 0$, 
$$\int_{\R}|x|^k g(t,x)\dx \le \min\left( \int_{\R}|x|^k \Gg(x)\dx, \frac{3 \times 2^{1-k} M_2(\Gg)^{\frac{k}{2}}}{1-\frac{k}{4}-2^{1-k}}\right).$$
Consequently,  {since $\|g(t)\|_{L^1}=1$ for any $t\ge 0$, we deduce from Theorem \ref{theo:Unique}} that there exists  $C >0$ such that, for  {$2<a_{*}<3-\delta$ }
\begin{equation}\label{eq:Cgt}
  \|g(t)\|_{L^1(\w_{a_{*}})} \leq C, \qquad \qquad \forall t\geq0,\qquad \g \in (0,\g_{\star})\,.\end{equation}
 {We then obtain from \eqref{eq:gtHg}, \eqref{eq:g_sobolev} and  {Theorem} \ref{k-norm-cvgce},} that, for any $t\geq0$
\begin{equation}\begin{split}\label{L1a-conv}
\|g(t)-\bm{h}_{\g}\|_{L^1(\w_{a})} &\le C_{1} e^{- \alpha\sigma(1-r)t} \vertiii{\widehat{\Gg}-\widehat{\bm{h}_{\g}}}_{a}^{ \alpha(1-r)} \\
&\leq  \tilde{C}_{1} e^{- \alpha\sigma(1-r)t} \|\Gg-\bm{h}_{\g}\|_{L^1(\w_{a})}^{ \alpha(1-r)}\,,\end{split}\end{equation}
 for some positive constants $C_{1},\tilde{C}_{1}$ independent of $\g \in (0,\g_{\star})$, {$\sigma = \sigma_a=1-\frac{a}{4}-2^{1-a}$} and where we used Lemma \ref{lem:mukvertk} to bound the $\vertiii{\cdot}_{a}$ norm  by the $\|\cdot\|_{L^1(\w_{a})}$. 
 
Let us now look for a bound of $\|\Gg-g(t)\|_{L^1(\w_{a})}$. We deduce from \eqref{eq:IB-selfsim}   and \eqref{eq:steadyg} that 
$$
\partial_t(\Gg- g(t))  {+}\frac14 \partial_x(x(\Gg-g(t))) 
=\Q_\gamma(\Gg,\Gg)-\Q_0(g(t),g(t)).$$
Multiplying the above equation with $\sgn(\Gg- g(t))\,\w_{a}$ and integrating over $\R$ we obtain 
\begin{multline*}
  \frac{\d }{\d t} \|\Gg- g(t)\|_{L^1(\w_{a})} 
    -\frac{a}{4} \int_\R |x| \w_{a-1}(x) |\Gg(x)-g(t,x)| \d x \\
\le \|\Q_\gamma(\Gg,\Gg)- \Q_0(\Gg,\Gg)\|_{L^1(\w_{a})}+ \|\Q_{0}(\Gg,\Gg)-\Q_{0}(g(t),g(t))\|_{L^1(\w_{a})}.
\end{multline*}
Now, 
\begin{align*}
  \|\Q_{0}(\Gg,\Gg)-\Q_{0}(g(t),g(t))\|_{L^1(\w_{a})} & =
  \|\Q_{0}(\Gg-g(t),\Gg+g(t))\|_{L^1(\w_{a})} \\
  &\le 2 \|\Gg-g(t)\|_{L^1(\w_{a})} \|\Gg+g(t)\|_{L^1(\w_{a})},
\end{align*} {
and with Proposition \ref{diff_Q_v2} (with $p=2$ and $s$ such that $a+s+\gamma_\star<3-\delta$) together with Theorem \ref{theo:Unique}, Corollary \ref{L2-weighted} and \eqref{eq:Cgt}, it implies that
$$\frac{\d }{\d t} \|\Gg- g(t)\|_{L^1(\w_{a})}  \le C_1  \|\Gg-g(t)\|_{L^1(\w_{a})}  + C_{2}\gamma(1+|\log\gamma|), \qquad \g\in(0,\g_\star),$$
with $C_1>0$  and $C_2>0$.  Gronwall's lemma directly gives
\begin{equation}\label{toto}\|\Gg- g(t)\|_{L^1(\w_{a})}  \le \frac {C_{2}\gamma(1+|\log\gamma|)}{C_1} \exp\left(C_1t\right),\qquad t\ge 0.
\end{equation}
Finally, \eqref{diff_G} together with \eqref{L1a-conv} and \eqref{toto} gives, for any $t\ge 0$,
\begin{equation*}\label{eq:difGhg}
\|\Gg-\bm{h}_{\g}\|_{L^1(\w_{a})}\le \frac{C_{2}\gamma(1+|\log\gamma|)}{C_1} \exp\left(C_1t\right)+  \tilde{C}_1 \exp\left(- \alpha\sigma(1-r)t\right){\|\Gg-\bm{h}_{\g}\|_{L^1(\w_{a})}^{ \alpha(1-r)}} .\end{equation*}
Theorem \ref{theo:Unique} allows to bound the last norm to deduce
 \begin{equation*}
\|\Gg-\bm{h}_{\g}\|_{L^1(\w_{a})}\le \frac{C_{2}\gamma(1+|\log\gamma|)}{C_1} \exp\left(C_1t\right)+  C_{a,\alpha,r}  \exp\left(- \alpha\sigma(1-r)t\right).\end{equation*}
Choosing $t=(C_1+ \alpha\sigma(1-r))^{-1}\log\left(\frac{C_{a,\alpha,r}C_1}{C_2\gamma}\right)$ we get
 \begin{equation*}
\|\Gg-\bm{h}_{\g}\|_{L^1(\w_{a})}\le C_{a,\alpha,r}\left(\frac{C_2 \gamma}{C_{a,\alpha,r} C_1}\right)^{\frac{\alpha \sigma (1-r)}{ {C_1}+\alpha \sigma(1-r)}}(2+|\log\gamma|):=\eta(\gamma)\end{equation*}
which proves the result.}
\end{proof}  
We have now all the ingredients to prove Theorem~\ref{theo:sta}.
\begin{proof}[Proof of Theorem~\ref{theo:sta}]   Let $\delta$ and $\g_\star$ be defined as in the proof of Proposition \ref{lem:stabil}. {Up to reducing $\g_\star$, one may assume that $\g_\star< \frac{a-2}2$.} For such $\delta$ and $\g_\star$, the results of Theorem \ref{theo:Unique}, Corollary \ref{L2-weighted} {and Lemma \ref{lem:IgfgI0}} hold and $a+\g_\star<3-\delta$. From the estimate $\|\Gg-\bm{G}_{0}\|_{L^{1}(\w_{a})} \leq \|\Gg-\bm{h}_{\g}\|_{L^{1}(\w_{a})}+\|\bm{h}_{\g}-\bm{G}_{0}\|_{L^{1}(\w_{a})}$ and using Proposition \ref{lem:stabil}  {and \eqref{stab}}, we deduce that
\begin{equation}\label{eq:Ca}
\|\Gg-\bm{G}_{0}\|_{L^{1}(\w_{a})} \leq \eta(\g) + C_{a}\left|\lambda_{\g}-\lambda_{0}\right|\,,\end{equation}
where we recall that $\bm{G}_{0}(x)=\lambda_{0}\bm{H}(\lambda_{0}x),$ $\bm{h}_{\g}(x)=\lambda_{\g}\bm{H}(\lambda_{\g}x)$ where $\lambda_{\g}=M_{2}(\Gg)^{-\frac{1}{2}}$ is such that $\int_{\R}x^{2}\bm{h}_{\g}(x)\dx=M_{2}(\Gg)$. It is therefore enough to quantify the rate of convergence of $\lambda_{\g}$ to $\lambda_{0}.$ Resuming the computations of Lemma \ref{lem:unique} and recalling the notation   \eqref{eq:mathlo}, we see that
\begin{equation*}
\mathscr{I}_{0}(\bm{h}_{\g},\bm{h}_{\g})=\frac{1}{\lambda_{\g}^{2}}\int_{\R^{2}}\bm{H}(x)\bm{H}(y)|x-y|^{2}\log\frac{|x-y|}{\lambda_{\g}}\dx\dy
=\frac{2}{\lambda_{\g}^{2}} {\log\frac{\lambda_{0}}{\lambda_{\g}}}\end{equation*}
where we used that $\mathscr{I}_{0}(\bm{H},\bm{H})=2\log \lambda_{0}$ and $\int_{\R^{2}}\bm{H}(x)\bm{H}(y)|x-y|^{2}\dx\dy=2$ as established in Lemma \ref{lem:unique}. We introduce also the notation
\begin{equation}\label{eq:mathgam}
\mathscr{I}_{\g}(f,g)=\g^{-1}\int_{\R^{2}}f(x)g(y)|x-y|^{2}\left(|x-y|^{\g}-1\right)\dx\dy, \qquad \qquad f,g \in L^{1}(\w_{2+\g})\end{equation}
and recall (see \eqref{eq1}) that $\mathscr{I}_{\g}(\Gg,\Gg)=0$. One has then the following equalities
\begin{equation*} \frac{2}{\lambda_{\g}^{2}} {\log\frac{\lambda_{0}}{\lambda_{\g}}} =\mathscr{I}_{0}(\bm{h}_{\g},\bm{h}_{\g})
=\mathscr{I}_{0}(\bm{h}_{\g}-\Gg,\bm{h}_{\g}-\Gg)+2\mathscr{I}_{0}(\bm{h}_{\g}-\Gg,\Gg) {+} \mathscr{I}_{0}(\Gg,\Gg).
 \end{equation*}
{ Hence, 
$$\left|\frac{2}{\lambda_{\g}^{2}}\log\frac{\lambda_{0}}{\lambda_{\g}} \right|\leq C_{a}\|\bm{h}_{\g}-\Gg\|_{L^{1}(\w_{a})}\left(2\|\Gg\|_{L^{1}(\w_{a})}+\|\bm{h}_{\g}-\Gg\|_{L^{1}(\w_{a})}\right) + {|\mathscr{I}_{0}(\Gg,\Gg)|}\,,$$}
for $2 < a <3$ and where we used Lemma \ref{lem:I0fg}. We deduce then from Proposition \ref{lem:stabil} and the fact that $\sup_{\g\in(0,\g_{\star})}\|\Gg\|_{L^{1}(\w_{a})} {<\infty} $   by Theorem \ref{theo:Unique}  that
$$\left|\lambda_{\g}^{-2}\log\left(\frac{\lambda_{\g}}{\lambda_{0}}\right)^{2}\right| \leq \tilde{\eta}(\g) + \left|\mathscr{I}_{0}(\Gg,\Gg)\right|$$
where $\tilde{\eta}(\g)=C_{a}\eta(\g)\left(2\sup_{\g}\|\Gg\|_{L^{1}(\w_{a})}+\eta(\g)\right) \to 0$ as $\g \to 0^{+}$ is an explicit function. Since $\mathscr{I}_{\g}(\Gg,\Gg)=0$,
$$\left|\lambda_{\g}^{-2}\log\left(\frac{\lambda_{\g}}{\lambda_{0}}\right)^{2}\right| \leq \tilde{\eta}(\g) + \left|\mathscr{I}_{0}(\Gg,\Gg)-\mathscr{I}_{\g}(\Gg,\Gg)\right|$$
and, using Lemma \ref{lem:IgfgI0} together with the estimates in Theorem \ref{theo:Unique} and Corollary \ref{L2-weighted}, we obtain that,  {for $2<a<3-\delta$,} 
$$\left|\lambda_{\g}^{-2}\log\left(\frac{\lambda_{\g}}{\lambda_{0}}\right)^{2}\right| \leq \tilde{\eta}(\g) + {\bm{C}\g }, \qquad \forall \g \in (0,\g_{\star})\,,$$
for some positive constant $\bm{C}$ depending  {on $a$}.  Noticing that $\lambda_{\g}$ is bounded both from above and below for $\g$ small enough due to $\lambda_{\g} \to \lambda_{0}$ as $\g \to 0$, we get that there is $\bm{C}_{0}$ such that
$$\left|\log \frac{\lambda_{\g}}{\lambda_{0}} \right| \leq \bm{C}_{0}\left(\tilde{\eta}(\g)+  {\g}\right).$$
Since $|\log x| \geq \frac{|1-x|}{\max(1,x)}$, there exists $\bm{C}_{1} >0$  such that
$$\left|\lambda_{\g}-\lambda_{0}\right| \leq \bm{C}_{1}\left(\tilde{\eta}(\g)+  {\g}\right), \qquad \g \in (0,\g_{\star}).$$
Introducing the explicit  function $\overline{\eta}(\g)=C_{a}\bm{C}_{1}\left(\tilde{\eta}(\g)+ {\g}\right)+\eta(\g)$, this, together with \eqref{eq:Ca}, proves the result. \end{proof} 
\begin{rem}\label{rem:nonexplicit} Notice that the constants $\bm{C}_{0}$ and $\bm{C}_{1}$ in the above proof depend on upper and lower bounds on $\lambda_{\g}=\left(M_{2}(\Gg)\right)^{-\frac{1}{2}}$ and $M_{2}(\Gg) \leq \frac{1}{2}$. We describe in Section \ref{sec:quantg} a procedure which allows to make the function $\overline{\eta}( \gamma)$ \emph{completely} explicit.
\end{rem}
\subsection{Uniqueness}\label{sec:uniqueness}  To derive some uniqueness result from the stability estimate derived in Theorem~\ref{theo:sta}, we first need to establish estimates for the difference of solutions $\Gg^{1}-\Gg^{2}$ with $\Gg^{i} \in \mathscr{E}_{\g}$ $(i=1,2)$ and $\g$ sufficiently small. For such stability estimates, we begin to estimate the action of $\mathscr{L}_{0}$ on such differences, where  $\mathscr{L}_{0}$ is the linearized collision operator around $\bm{G}_{0}$ as defined in Definition \ref{def:spaces}. 
 \begin{lem}\label{lem:bareta} Let $2 < a < 3$. There exist
     $\g_\star\in(0,1)$ and a mapping $\tilde{\eta}\;:\;[0,\g_\star] \to \R^{+}$ with $\lim_{\g\to0^{+}}\tilde{\eta}(\g)=0$
and such that
\begin{equation}\label{eq:L0dif}
\|\mathscr{L}_{0}\left(\Gg^{1}-\Gg^{2}\right)\|_{\X_{a}} \leq \tilde{\eta}(\g)\left\|\Gg^{1}-\Gg^{2}\right\|_{\X_{a}}.\end{equation} holds true for any  $\Gg^{1},\Gg^{2} \in \mathscr{E}_{\g} $ and any $\g\in(0,\g_\star)$.
\end{lem}
\begin{proof} Let $\delta$ and $\g_\star$ be defined as in the proof of Theorem \ref{theo:sta}. For such $\delta$ and $\g_\star$, the results of Theorem \ref{theo:Unique} and Corollary \ref{L2-weighted} hold and $a+\g_\star<3-\delta$. Let $\g\in(0,\g_\star)$. Let us consider $\Gg^{1},\Gg^{2} \in \mathscr{E}_{\g}$. We introduce the difference 
$$g_{\g}=\Gg^{2}-\Gg^{1}$$
which satisfies \eqref{eq:diffe}. We write this last identity in an equivalent way:
 {\begin{multline*}
\frac{1}{4}\partial_{x}\left(xg_{\g}(x)\right)=\left[\Q_{\g}(g_{\g},\Gg^{2}-\bm{G}_{0})+\Q_{\g}(\Gg^{1}-\bm{G}_{0},g_{\g})\right]\\
+2 \left[\Q_{\g}(g_{\g},\bm{G}_{0})-\Q_{0}(g_{\g},\bm{G}_{0})\right]+2\Q_{0}(g_{\g},\bm{G}_{0})
\end{multline*}}
which can be written as
 {$$-\mathscr{L}_{0}(g_{\g})=\mathcal{A}_{\g}+\mathcal{B}_{\g}$$ }
where
$$\mathcal{A}_{\g}=\left[\Q_{\g}(g_{\g},\Gg^{2}-\bm{G}_{0})+\Q_{\g}(\Gg^{1}-\bm{G}_{0},g_{\g})\right], \qquad {\mathcal{B}_{\g}=2\left[\Q_{\g}(g_{\g},\bm{G}_{0})-\Q_{0}(g_{\g},\bm{G}_{0})\right]}.$$
Therefore, 
$$\|\mathscr{L}_{0}(g_{\g})\|_{L^{1}(\w_{a})}\leq \|\mathcal{A}_{\g}\|_{L^{1}(\w_{a})}+\|\mathcal{B}_{\g}\|_{L^{1}(\w_{a})}.$$
One estimates separately the norms $\|\mathcal{A}_{\g}\|_{L^{1}(\w_{a})}$, $\|\mathcal{B}_{\g}\|_{L^{1}(\w_{a})}.$ It follows from Proposition \ref{prop:QgL1} that 
\begin{equation*}
\|\mathcal{A}_{\g}\|_{L^{1}(\w_{a})}\leq \eta_{1}(\g)\|g_{\g}\|_{L^{1}(\w_{a+\g})}
\end{equation*}
with 
$$\eta_{1}(\g)={2}\left(\|\Gg^{1}-\bm{G}_{0}\|_{L^{1}(\w_{a+\g})}+\|\Gg^{2}-\bm{G}_{0}\|_{L^{1}(\w_{a+\g})}\right).$$
According to {Theorem \ref{theo:sta}}, the mapping  $\eta_{1}:[0,\g_\star] \to \R^{+}$ is such that 
$$\lim_{\g\to0}\eta_{1}(\g)=0.$$
One deduces then from  Proposition \ref{diff_Q_v2}, with $s\in(0,1)$ {such that $s+\gamma_\star+a<3$} and $p=2$, that
 {\begin{multline*}
\|\mathcal{B}_{\g}\|_{L^{1}(\w_{a})} \leq 16 \g |\log \g| \|g_{\g}\|_{L^1(\w_{a})}\|\bm{G}_{0}\|_{L^1(\w_{a})}\\
    +  \frac{32\g}{s} \|\bm{G}_{0}\|_{L^1(\w_{s+\gamma+a})} \|g_{\g}\|_{L^1(\w_{s+\gamma+a})} + 8\g \|\bm{G}_{0}\|_{ L^{2}(\w_{a})}\|g_{\g}\|_{L^1(\w_{a})}.
\end{multline*}}
Using the known bounds on $\bm{G}_{0}$, one deduces that there exists $C_{s} >0$ (independent of $\g$) such that
{$$\|\mathcal{B}_{\g}\|_{L^{1}(\w_{a})} \leq C_{s}\g{(1+|\log\g|)}\,\|g_{\g}\|_{L^{1}(\w_{a+\g+s})}, \qquad \forall \g \in (0,\g_{\star}).$$}
Gathering all these estimates, we obtain,  {for any $s\in(0,1)$  {such that $s+\gamma_\star+a<3$},
$$\|\mathscr{L}_{0}(g_{\g})\|_{L^{1}(\w_{a})} \leq  \left(\eta_{1}(\g)+C_{s}\g (1+|\log\g|)\right)\,\|g_{\g}\|_{L^{1}(\w_{a+s+\g})}.$$
Let us fix $\bar{\g}\in(0,1)$ such that $\bar{\g}\le \g_\star$ and $\bar{\g}<\frac{a}{3}$.
Let $s\in(0,1)$ be such that $\bar{\g}+\frac23 s\le \frac{a}{3} $  {and $s+\gamma_\star+a<3$}. We now deduce from Lemma \ref{lem:Diffmom} with $k=a+s$ that there exists $\g_\star(a,s)$ such that 
$$\|g_{\g}\|_{L^{1}(\w_{a+s+\g})} \leq C_{a,s}\|g_{\g}\|_{L^{1}(\w_{\g+\frac{2(a+s)}{3}})}\leq C_{a,s}\|g_{\g}\|_{L^{1}(\w_{a})},\qquad \g\in(0,\g_{\star}(a,s)).$$
Consequently, 
$$ \|\mathscr{L}_{0}(g_{\g})\|_{L^{1}(\w_{a})} \leq  C_{a,s}\left(\eta_{1}(\g)+C_{s}\g {(1+|\log\g|)}\right)\,\|g_{\g}\|_{L^{1}(\w_{a})},\qquad \g\in(0,\g_{\star}(a,s)),$$
which gives the result.}
\end{proof}

Combining the above result with Proposition \ref{restrict0} allows to show directly that two solutions to \eqref{eq:steadyg} \emph{with same energy} coincide as already explained in the introduction. 

In order to extend this line of reasoning to general solutions to \eqref{eq:steadyg} with different energy, one somehow follows the same approach but needs a way to compensate the discrepancy of energies to apply a variant of \eqref{eq:invert0}. Typically, let us now consider two solutions $\Gg^{1},\Gg^{2} \in  \mathscr{E}_{\g}$ and let $g_{\g}=\Gg^{1}-\Gg^{2}.$
If one is able to construct $\tilde{g}_{\g} \in \mathbb{Y}_{a}$ such that
\begin{equation}\label{eq:kernL0g}
\mathscr{L}_{0}(\tilde{g}_{\g})=\mathscr{L}_{0}(g_{\g}) \qquad \text{ and } \quad M_{2}(\tilde{g}_{\g})=0 \quad  (\text{i.e. } {\tilde{g}_\g} \in \mathbb{Y}_{a}^{0})\end{equation}
then, as before, one would have
\begin{equation}\label{eq:tildeGg}
\frac{\nu}{C(\nu)} \|\tilde{g}_{\g}\|_{\X_{a}} \leq \|\mathscr{L}_{0}(\tilde{g}_{\g})\|_{\X_{a}}=\|\mathscr{L}_{0}({g}_{\g})\|_{\X_{a}}  \leq \tilde{\eta}(\g)\left\|g_{\g}\right\|_{\X_{a}}\,.\end{equation}
To conclude as before, we also need to check that there is $C >0$ (\emph{independent of $\g$}) such that
\begin{equation}\label{eq:tildeg}
\|g_{\g}\|_{\X_{a}} \leq C \|\tilde{g}_{\g}\|_{\X_{a}}\end{equation}
from which the identity $g_{\g}=0$ would follow  easily, as in the introduction (see end of Section~\ref{Sec:strategy}) for solutions with same energy.

Of course, constructing $\tilde{g}_{\g}$ satisfying \eqref{eq:kernL0g} is easy since $\mathscr{L}_{0}$ is invertible on $\mathbb{Y}_{a}^{0}$. The difficulty is to check \eqref{eq:tildeg}. The crucial tool to achieve this purpose is the  
\emph{``linearized dissipation of energy''} functional 
$$\mathscr{I}_{0}(f,\bm{G}_{0})=\int_{\R^{2}}f(x)\bm{G}_{0}(y)|x-y|^{2}\log |x-y|\d x\d y, \qquad f \in L^{1}(\w_{s}), s >2$$
where we recall the notation introduced in \eqref{eq:mathlo}. 
The functional $\mathscr{I}_{0}(\cdot,\bm{G}_{0})$, when evaluated in the differences of solutions to \eqref{eq:steadyg}, has the following smallness property. 
\begin{lem}\label{lem:lastI0} Let $2 < a <3$. There exist $\g_\star\in(0,1)$ and $\bar{\eta}_{0}(\g)$ with 
$\lim_{\g\to0}\bar{\eta}_{0}(\g)=0$
 such that, for any $\g\in(0,\g_\star)$, any $\Gg^{1},\Gg^{2} \in \mathscr{E}_{\g}$,
\begin{equation}\label{eq:diffGg1-2}
\left|\mathscr{I}_{0}\left(\Gg^{1}-\Gg^{2},\bm{G}_{0}\right)\right| \leq \bar{\eta}_{0}(\g)\,\|\Gg^{1}-\Gg^{2}\|_{\X_{a}}.\end{equation}
\end{lem}

\begin{proof}  Let $\delta$ and $\g_\star$ be defined as in the proof of Theorem \ref{theo:sta}. For such $\delta$ and $\g_\star$, the results of Theorem \ref{theo:Unique} and Corollary \ref{L2-weighted} hold and $a+\g_\star<3-\delta$. For $\g\in(0,\g_\star)$,  $\Gg^{1},\Gg^{2} \in \mathscr{E}_{\g}$, let $g_{\g}=\Gg^{1}-\Gg^{2}.$ One notices that
\begin{multline*}
2\mathscr{I}_{0}({g_{\g}},\bm{G}_{0})=\mathscr{I}_{0}(g_{\g},\bm{G}_{0}-\Gg^{1})+\mathscr{I}_{0}(g_{\g},\bm{G}_{0}-\Gg^{2})+
\mathscr{I}_{0}(g_{\g},\Gg^{1}+\Gg^{2})\\
=\mathscr{I}_{0}(g_{\g},\bm{G}_{0}-\Gg^{1})+\mathscr{I}_{0}(g_{\g},\bm{G}_{0}-\Gg^{2})\\
+\left(\mathscr{I}_{0}(g_{\g},\Gg^{1}+\Gg^{2})-\mathscr{I}_{\g}(g_{\g},\Gg^{1}+\Gg^{2})\right)
\end{multline*}
since, with notation \eqref{eq:mathgam},
$$\mathscr{I}_{\g}(\Gg^{1}-\Gg^{2},\Gg^{1}+\Gg^{2})= {\mathscr{I}_{\g}(\Gg^{1},\Gg^{1})-\mathscr{I}_{\g}(\Gg^{2},\Gg^{2})}=0$$
for $\Gg^{1},\Gg^{2} \in \mathscr{E}_{\g}.$  One invokes then Lemma \ref{lem:I0fg} and \ref{lem:IgfgI0} to deduce that there are $C_{a},\bar{C}_{a} >0$ such that
$$\left|\mathscr{I}_{0}( {g_{\g}},\bm{G}_{0})\right| \leq C_{a}\left(\left\|\bm{G}_{0}-\Gg^{1}\right\|_{\X_{a}}+\left\|\bm{G}_{0}-\Gg^{2}\right\|_{\X_{a}}\right)\|g_{\g}\|_{\X_{a}}\\
+\bar{C}_{a}\g \,\|\Gg^{1}+\Gg^{2}\|_{\X_{a}}\|g_{\g}\|_{\X_{a}}.$$
Thanks to the uniform bounds on {$\|\Gg^{i}\|_{\X_{a}}$ $(i=1,2)$  given by Theorem \ref{theo:Unique} together with Theorem \ref{theo:sta}}, we deduce the result. 
\end{proof}
 We are in position to prove our main result regarding the steady solution to \eqref{eq:steadyg} following the strategy described before.
\begin{proof}[Proof of Theorem \ref{theo:mainUnique}]  Let $\delta$ and $\g_\star$ be defined as in the proof of Lemma \ref{lem:lastI0}. For $\g\in(0,\g_\star)$,  $\Gg^{1},\Gg^{2} \in \mathscr{E}_{\g}$, let $g_{\g}=\Gg^{1}-\Gg^{2}.$ Since $\mathscr{L}_{0}$ is invertible on $\mathbb{Y}_{a}^{0}$, there exists a \emph{unique} $\tilde{g}_{\g} \in \mathbb{Y}^{0}_{a}$ such that
$$\mathscr{L}_{0}(\tilde{g}_{\g})=\mathscr{L}_{0}(g_{\g}).$$
It remains to prove the estimate \eqref{eq:tildeg} 
between $\|\tilde{g}_{\g}\|_{\X_{a}}$ and $\|g_{\g}\|_{\X_{a}}$. To do so, we actually prove that $\tilde{g}_{\g}-g_{\g} \in \mathrm{Span}(\varphi_{0})$, more precisely
\begin{equation}\label{eq:exp}
g_{\g}=\tilde{g}_{\g}+z_{0}\varphi_{0}, \qquad z_{0}=\frac{1}{p_{0}}\mathscr{I}_{0}(g_{\g}-\tilde{g}_{\g},\bm{G}_{0})\end{equation}
where $p_{0}=\mathscr{I}_{0}(\varphi_{0},\bm{G}_{0})\neq 0$ by Lemma \ref{rmk:varphi0}. Indeed, writing $\bar{g}_{\g}=\tilde{g}_{\g}+z_{0}\varphi_{0}$ one sees that, since $\varphi_{0} \in \mathrm{Ker}(\mathscr{L}_{0})$
$$\mathscr{L}_{0}(\bar{g}_{\g})=\mathscr{L}_{0}(\tilde{g}_{\g})=\mathscr{L}_{0}(g_{\g})$$
while, obviously, the choice of $z_{0}$ implies that
$$\mathscr{I}_{0}(\bar{g}_{\g},\bm{G}_{0})=\mathscr{I}_{0}(g_{\g},\bm{G}_{0}).$$
From Lemma \ref{rmk:varphi0}, this implies that $M_{2}(\bar{g}_{\g}-g_{\g})=0$ and $\bar{g}_{\g}-g_{\g}=0$. This proves \eqref{eq:exp}. Consequently,
\begin{equation*}\begin{split}
\|g_{\g}\|_{\X_{a}} \leq \|\tilde{g}_{\g}\|_{\X_{a}} + |z_{0}|\,\|\varphi_{0}\|_{\X_{a}} &\leq \|\tilde{g}_{\g}\|_{\X_{a}}+ \frac{\|\varphi_{0}\|_{\X_{a}}}{|p_{0}|}\left|\mathscr{I}_{0}(g_{\g}-\tilde{g}_{\g},\bm{G}_{0})\right|\\
&\leq \|\tilde{g}_{\g}\|_{\X_{a}}+ \frac{\|\varphi_{0}\|_{\X_{a}}}{|p_{0}|}\left(|\mathscr{I}_{0}(g_{\g},\bm{G}_{0})|+|\mathscr{I}_{0}(\tilde{g}_{\g},\bm{G}_{0})|\right)\end{split}\end{equation*}
by definition of $z_{0}.$ According to Lemma  \ref{lem:I0fg}, there is $C_{0} >0$ such that 
$|\mathscr{I}_{0}(\tilde{g}_{\g},\bm{G}_{0})| \leq C_{0}\|\tilde{g}_{\g}\|_{\X_{a}}.$ Therefore, there are $C_{1}, C_{2}>0$ (independent of $\g$) such that
\begin{equation}\label{eq:ggI0}
\|g_{\g}\|_{\X_{a}} \leq C_{1}\|\tilde{g}_{\g}\|_{\X_{a}}+ C_{2} \left|\mathscr{I}_{0}(g_{\g},\bm{G}_{0})\right|.\end{equation}
Using now \eqref{eq:diffGg1-2}, we deduce that
$$\|g_{\g}\|_{\X_{a}} \leq C_{1}\|\tilde{g}_{\g}\|_{\X_{a}}+C_{2}\bar{\eta}_{0}(\g)\|g_{\g}\|_{\X_{a}}$$
and, since $\lim_{\g\to0}\bar{\eta}_{0}(\g)=0$, we can choose $\g^{\star} >0$ small enough so that $C_{2}\bar{\eta}_{0}(\g) \leq \frac{1}{2}$ for any $\g \in (0,\g^{\star})$ so that
$$\frac{1}{2}\|g_{\g}\|_{\X_{a}} \leq C_{1}\|\tilde{g}_{\g}\|_{\X_{a}}, \qquad \forall \g \in (0,\g^{\star}).$$
With the strategy described before, we deduce that the functions $\tilde{g}_{\g}$ and $g_{\g}$ satisfy \eqref{eq:kernL0g}--\eqref{eq:tildeGg} and \eqref{eq:tildeg} with $C=2C_{1}$. In particular, we deduce from \eqref{eq:tildeGg} that
$$\frac{\nu}{C(\nu)}\,\frac{1}{2C_{1}}\|g_{\g}\|_{\X_{a}} \leq \tilde{\eta}(\g)\,\|g_{\g}\|_{\X_{a}}$$
and, since $\lim_{\g\to0}\tilde{\eta}(\g)=0$, there exists $\g^{\dagger} >0$ small enough so that
$$\|g_{\g}\|_{\X_{a}} < \|g_{\g}\|_{\X_{a}} \qquad \forall \g \in (0,\g^{\dagger})$$
which implies that $g_{\g}=0$ and proves the result. 
 \end{proof}
 
 \subsection{Quantitative estimate on $\g^{\dagger}$}\label{sec:quantg}
In order to make Theorem \ref{theo:mainUnique} completely explicit, we need a quantitative estimate for the threshold parameter $\g^{\dagger}$. From the aforementioned proof, this amounts to a quantitative estimate on the mapping $\tilde{\eta}(\g)$.  As observed in Remark \ref{rem:nonexplicit}, the only non completely quantitative estimate in the definition of $\tilde{\eta}(\g)$ rises from the mapping $\overline{\eta}(\g)$ in Theorem \ref{theo:sta}. In this subsection, we briefly explain how it is possible to derive such a quantitative estimate. We keep the presentation slightly informal here just to stress out the main steps of the estimates. The crucial point is then to estimate the rate of convergence of 
$$\|\Gg-\bm{G}_{0}\|_{L^{1}(\w_{a})}$$
to zero as $\g \to 0.$
To do so, we briefly resume the main steps in our proof of uniqueness and introduce, for $\Gg \in \mathscr{E}_{\g}$,
$$h_{\g}=\bm{G}_{0}-\Gg.$$
One sees that
$$\mathscr{L}_{0}(h_{\g})=\Q_{0}(h_{\g},h_{\g})+\left[\Q_{\g}(\Gg,\Gg)-\Q_{0}(\Gg,\Gg)\right]$$
which results in
\begin{equation*}\label{eq:L0h}
\|\mathscr{L}_{0}(h_{\g})\|_{\X_{a}} \leq C_{0}\|h_{\g}\|_{\X_{a}}^{2}+ {C_{0}\g\left(1+|\log\g|\right)}\end{equation*}
for some positive $C_{0}$ independent of $\g$ (see Lemma \ref{N0_Ggamma} for a similar reasoning). Now, as before, there exists $\tilde{h}_{\g} \in \mathbb{Y}^{0}_{a}$ such that
$$\frac{\nu}{C(\nu)}\|\tilde{h}_{\g}\|_{\X_{a}} \leq \|\mathscr{L}_{0}(\tilde{h}_{\g})\|_{\X_{a}}=\|\mathscr{L}_{0}h_{\g}\|_{\X_{a}}.$$
Therefore, there is $C >0$ independent of $\g$ such that
\begin{equation}\label{eq:tildeh}
\|\tilde{h}_{\g}\|_{\X_{a}} \leq C\|h_{\g}\|_{\X_{a}}^{2} {+C\g\left(1+|\log \g|\right)}\end{equation}
and we need to compare again $\|\tilde{h}_{\g}\|_{\X_{a}}$ to $\|h_{\g}\|_{\X_{a}}.$ As in Eq. \eqref{eq:ggI0}
\begin{equation}\label{eq:hgtildeh}
\|h_{\g}\|_{\X_{a}} \leq C_{1}\|\tilde{h}_{\g}\|_{\X_{a}}+C_{2}\left|\mathscr{I}_{0}(h_{\g},\bm{G}_{0})\right|\end{equation}
for $C_{1},C_{2}$ independent of $\g.$ Now, one checks without major difficulty that
$$2\mathscr{I}_{0}(h_{\g},\bm{G}_{0})=\mathscr{I}_{0}(h_{\g},h_{\g})+\left[\mathscr{I}_{\g}(\Gg,\Gg)-\mathscr{I}_{0}(\Gg,\Gg)\right]$$
where we used that $\mathscr{I}_{0}(\bm{G}_{0},\bm{G}_{0})=\mathscr{I}_{\g}(\Gg,\Gg)=0.$ Thus, with Lemmas \ref{lem:I0fg}, \ref{lem:IgfgI0} and Theorem \ref{theo:Unique}, we deduce that
$$\left|\mathscr{I}_{0}(h_{\g},\bm{G}_{0})\right| \leq C_{3}\|h_{\g}\|_{\X_{a}}^{2}+ {C_{3}\g}$$
for some $C_{3} >0$ independent of $\g$. Summing up this estimate with \eqref{eq:tildeh} and \eqref{eq:hgtildeh} one sees that there exists a positive constant $\bm{c}_{0} >0$ independent of $\g$ such that
$$\|h_{\g}\|_{\X_{a}} \leq \bm{c}_{0}\|h_{\g}\|_{\X_{a}}^{2}  +\bm{c}_{0}\g\left(1+|\log\g|\right).$$
Now, since we know that $\lim_{\g\to0}\|h_{\g}\|_{\X_{a}}=0$ (with no explicit rate at this stage), there exists $\g_{0} >0$ (non explicit) such that
$$\bm{c}_{0}\|h_{\g}\|_{\X_{a}} \leq \frac{1}{2} \qquad \forall \g \in (0,\g_{0})$$
and, therefore,
$$\|h_{\g}\|_{\X_{a}} \leq {2\bm{c}_{0}\g\left(1+|\log\g|\right) }\qquad \forall \g \in (0,\g_{0}).$$
Such an estimate provides actually an explicit estimate for $\g_{0}$ since the optimal parameter becomes clearly the one for which the two last estimates are identity yielding
$$ {\g_{0}\left(1+{|\log\g_{0}|}\right)}=\frac{1}{4\bm{c}_{0}^{2}}.$$
This provides then an explicit rate of convergence of $\Gg$ to $\bm{G}_{0}$ as 
$$\|\Gg-\bm{G}_{0}\|_{\X_{a}} \leq  {2\bm{c}_{0}\g\left(1+|\log\g|\right)} \qquad \forall \g \in (0,\g_{0})$$
for some explicit $\g_{0}.$ This makes explicit the mapping $\overline{\eta}(\g)$ in Theorem \ref{theo:sta} and, in turns, provides some quantitative estimates on the parameter $\g^{\dagger}$ in Theorem \ref{theo:mainUnique}.
\appendix
\numberwithin{equation}{section}                                                     %%
\section{Technical results}\label{app:tech}
We collect here several technical results used in the core of the paper. We begin with results regarding the Fourier norms used in Section \ref{sec:upgrade}.
\subsection{Properties of the Fourier norm and interpolation estimates}\label{app:fourier} We introduce here 
the space of measures
{\begin{equation*}%\label{eq:X0:measure}
X_{k}:=  \left\{ \mu\in {\mathcal M}_{k}(\R) \;\Bigg| \Bigg. \begin{array}{l}
  \displaystyle \int_{\R} \mu(\d x) = \int_{\R} x\, \mu(\d x) = \int_{\R} x^2 \,\mu(\d x) = 0
\end{array}\right\}, \qquad k >2,
\end{equation*}
where we recall that ${\mathcal M}_{k}(\R)$ denotes the set of real Borel measures on $\R$ with finite total variation of order $k$.}
Then, for $\mu\in X_k$ and  {$2<k<3$}, we define the norm 
\begin{equation}\label{eq:fourier:norm:2}
{\vertiii{\hat{\mu}}_{k}} :=  \sup_{\xi\in\R\setminus \{0\}} \frac{|\hat{\mu}(\xi)|}{|\xi|^{k}}. 
\end{equation} 
{By \cite[Proposition 2.6]{MR2355628}, for any {$2< k< 3$}, if $\mu\in X_k$ then $\vertiii{\hat{\mu}}_{k}$ is finite. Endowed with the norm  {$\vertiii{\cdot}_{k}$} with $2< k<3,$
the space $X_{k}$ is a Banach space, see Proposition 2.7 in \cite{MR2355628}.} 
 The following lemma is a consequence of \cite[Lemma 2.5]{MR2355628}.

\begin{lem}\label{lem:mukvertk}
 Let $2<k<3$. There exists a constant $C>0$ depending only on $k$ such that
  $$\vertiii{\hat{\mu}}_{k}\leq C \int_{\R}  (1+|x|)^{k} \,|\mu|(dx),$$
 for any  {$\mu\in X_{k}$}.
\end{lem}
One also has the following interpolation estimates. {Inequality \eqref{eq:interp1} is a consequence of \cite[Theorem 4.1]{CGT}. We refer to \cite[Lemma A.4]{long} for a proof of \eqref{eq:interp2}.}  
\begin{lem}\phantomsection \label{L2-knorm}
For $k> 2$, $\beta >0$ and $0<r<1$, one has that
{\begin{equation}\label{eq:interp1}
    \|f\|^{2}_{L^2}\le 2 \, C(r,\beta) \vertiii{\hat{f}}_{k}^{2(1-r)}\|f\|_{H^N}^{2r},
    \end{equation}
with
$$ N=\frac{(1-r)(2k+\beta+1)}{2r}, \qquad 
C(r,\beta)=\left(2 \left(1+\frac{1}{\beta}\right)\right)^{1-r}. $$}
Moreover, for $a_*>a >0$ and $\alpha \in \left( 0 , \frac{2(a_*-a)}{2a_*+1}\right)$ there exists a constant $C>0$ depending only on $\alpha$, $a$ and $a_*$ such that, for every $f\in L^1(\w_{a_*})\cap L^2$, 
\begin{equation}\label{eq:interp2}
  \|f\|_{L^1(\w_{a})}\le C \|f\|^\alpha_{L^2} \|f\|^{1-\alpha}_{L^1(\w_{a_*})} .
  \end{equation}
\end{lem}
  
 \subsection{Lower bound on the collision frequency} Let us introduce the collision frequency associated to $\Q^{-}_{\g}(\Gg,\Gg)$
\begin{equation}\label{eq:coll:freq}
\Sigma_{\g}(y)=\int_{\R}|x-y|^{\g}\Gg(x)\dx\,. 
\end{equation}
Recall the notation $\w_{s}$ introduced in \eqref{eq:weight}.  The following lemma holds, where $\g_{\star}$ was introduced in Theorem \ref{theo:Unique}.
\begin{lem}\phantomsection\phantomsection\label{lem:Sigmag} For any $\g \in (0,\g_{\star})$  and $\Gg \in \mathscr{E}_{\g}$, there exists $\kappa_{\g} >0$ such that the following holds
\begin{equation}\label{eq:sigma_g}\Sigma_{\g}(y) \geq \kappa_{\g}\,\w_{\g}(y)\,.\end{equation}
Moreover, $\lim_{\g\to0}\kappa_{\g}=1.$
\end{lem}

\begin{proof} Let $\g \in (0,1)$ and $\Gg \in \mathscr{E}_{\g}$ be given. We prove a more precise estimate of the form
\begin{equation}\label{eq:sigma}\Sigma_{\g}(y) \geq \tilde{\kappa}_{\g}\,\w_{\g}(y) - (1-{\delta}^{\g})-\sqrt{2{\delta}}\|\Gg\|_{L^2}, \quad \quad \forall \,{\delta} \in (0,1).\end{equation}
Observe that choosing ${\delta}=\g^{2}$, $-(1-{\delta}^{\g}) {\geq} 2\g\log\g$ and using the $L^{2}$-bound on $\Gg$ in \eqref{eq:estimGg}, we deduce from \eqref{eq:sigma} that
$$\Sigma_{\g}(y) \geq \tilde{\kappa}_{\g}\,\w_{\g}(y) {-2\g|\log\g| -C\g},$$
for some $C >0$ (independent of $\g$).  Since $\w_{\g}(y) >1$, this gives \eqref{eq:sigma_g} with $\kappa_{\g}=\tilde{\kappa}_{\g} {-2\g|\log \g|-C\g}$.  Now, let us show \eqref{eq:sigma}. First, notice that, for any $y \in \R$ and any $ {\delta} \in (0,1)$,
\begin{multline*}
\Sigma_{\g}(y)= \int_{\R}\left(|x-y|^{\g}+\ind_{|x-y|< {\delta}}\right)\,\Gg (x)\,\dx - \int_{\R}\ind_{|x-y|< \delta }\,\Gg (x)\,\dx\\
\geq \int_{\mathbb{R}}\left(|x-y|^{\g}+\ind_{|x-y|< {\delta}}\right)\,\Gg (x)\,\dx  - \sqrt{2 {\delta}}\| \Gg  \|_{L^2}=:\Sigma_{\g}^{( {\delta})}(y)-\sqrt{2 {\delta}}\|\Gg\|_{L^2}
\end{multline*}
thanks to Cauchy-Schwarz inequality. We only need to estimate the first term. To do so, for any $\eta >1$ and $y\in\R$, we introduce the set 
$$I=I(y,\,\eta)=\left\{x \in \R\;;\;\w_{1}(x) \leq \eta^{-1}\w_{1}(y)\right\},$$ and write
$$
\Sigma_{\g}^{( {\delta})}(y)
= \int_{I}\left(|x-y|^{\g}+\ind_{|x-y|< {\delta}}\right)\,\Gg (x)\,\dx + \int_{I^{c}}\left(|x-y|^{\g}+\ind_{|x-y|< {\delta}}\right)\,\Gg (x)\,\dx .
$$
On the set $I$, one has that
$$|x-y|^{\g}\geq \big( (1+|y|) - (1+ |x|) \big)^{\g}\geq   \big( 1-\eta^{-1} \big)^{\g}\w_{\g}(y).$$
Therefore,
\begin{equation*}\begin{split}
\int_{I}\left(|x-y|^{\g}+\ind_{|x-y|< {\delta}}\right)\,\Gg (x)\,\dx  &\geq   \left(\frac{\eta-1}{\eta}\right)^{\g}\w_{\g}(y)\int_{I}\,\Gg (x)\dx \\
&\geq  \left(\frac{\eta-1}{\eta}\right)^{\g}\w_{\g}(y)\int_{I}\,\frac{\Gg (x)}{\w_{\g}(x)}\dx  \,.
\end{split}\end{equation*}
Now, observing that $|x-y|^{\g}+\ind_{|x-y|< {\delta}}\geq {\delta}^{\g}\,,$ $( {\delta} <1)$ it follows that
\begin{equation*}\begin{split}
\int_{I^{c}}\left(|x-y|^{\g}+\ind_{|x-y|< {\delta}}\right)\,&\Gg (x)\,\dx 
\geq  {\delta}^{\g}\int_{I^{c}}\Gg(x)\dx 
\geq \int_{I^{c}}\Gg (x)\,\dx- (1- {\delta}^{\gamma})\|\Gg \|_{L^1}\\
&\geq  \frac{\w_{\g}(y)}{\eta^{\gamma}}\int_{I^{c}}\frac{\Gg (x)}{\w_{\g}(x)}\dx - (1- {\delta}^{\gamma})\,,
\end{split}\end{equation*}
since $1 \geq \frac{1+|y|}{\eta(1+|x|)}$ for any $x \notin I$.   Choosing  then $\eta=2$ one sees that
$$\Sigma_{\g}^{( {\delta})}(y)\geq \frac{\w_{\g}(y)}{{2}^{\g}}\int_{\R}\frac{\Gg(x)}{\w_{\g}(x)}\dx - (1- {\delta}^{\g})\,,$$
which gives \eqref{eq:sigma} with
$$\tilde{\kappa}_{\g}:=\frac{1}{2^{\g}}\int_{\R}\frac{\Gg(x)}{\w_{\g}(x)}\dx, \qquad \g \in (0,1).$$
Let us prove that $\lim_{\g\to0^{+}}\tilde{\kappa}_{\g}=1$ which will, of course, give that $\lim_{\g\to0^{+}}\kappa_{\g}=1$.  One has 
\begin{equation*}\begin{split}
\left|\tilde{\kappa}_{\g}-2^{-\g}\right| & =2^{-\g} \left|\int_{\R} \Gg(x)\frac{1-\w_{\g}(x)}{\w_{\g}(x)}\dx \right|
 = 2^{-\g} \int_{\R} \Gg(x)\left(1-(1+|x|)^{-\g}\right)\dx \\
&=  \g 2^{-\g} \int_{\R} \Gg(x)\dx \int_0^{|x|}  (1+r)^{-\g-1} \d r.
\end{split}\end{equation*}
Since $-\g-1 <0,$ the last integrand is always bounded by $1$ and we deduce that 
$$\left|\tilde{\kappa}_{\g}-2^{-\g}\right| \le \g2^{-\g}\int_{\R}|x|\Gg(x)\dx \le \g 2^{-\g-1} (1+M_2(\Gg))\le 3\ \g\, 2^{-\g-2},$$
where we also used Lemma \ref{lem:energy}. The result then follows letting $\g\to 0$.
\end{proof}
\subsection{Nonlinear estimates for $\Q_{\g}$ and $\mathscr{I}_{\g}$}\label{app:QgQ0}
We gather here nonlinear estimates involving {integrals of the collision operator $\Q_{\g}$ for $\g \geq 0$ in the spirit of \cite{ACG}. The same kind of computations also enables to get nonlinear estimates of the functional $\mathscr{I}_{\g}$ introduced in {Section \ref{sec:upgrade}}.} We begin with the following result for $\Q^{+}_{0}$.
\begin{lem}\phantomsection\label{lem:Q+0} For any measurable $f,g,h$, it follows that
$$\int_{\R}\Q^{+}_{0}(f,g)\,h\dx \leq \sqrt{2}\|h\|_{L^2}\min\left(\|f\|_{L^1}\|g\|_{L^2},\|g\|_{L^1}\|f\|_{L^2}\right).$$\end{lem}
\begin{proof} There is no loss of generality in assuming $f,g,h$ nonnegative. One has then
$$\int_{\R}\Q^{+}_{0}(f,g)\,h\dx=\int_{\R^{2}}f(x)g(y)h\left(\frac{x+y}{2}\right)\dx\dy.$$
Given $x \in \R$, it follows from Cauchy-Schwarz inequality that
$$\int_{\R}g(y)h\left(\frac{x+y}{2}\right)\dy \leq \|g\|_{L^2}\left\|h\left(\frac{x+\cdot}{2}\right)\right\|_{L^2}=\sqrt{2}\|g\|_{L^2}\|h\|_{L^2}$$
from which
$$\int_{\R}\Q^{+}_{0}(f,g)\,h\dx \leq \sqrt{2}\|f\|_{L^1}\|g\|_{L^2}\|h\|_{L^2}.$$
Exchanging the role of $g$ and $f$, one deduces the result.\end{proof}
Let us prove an $L^{1}$ estimate for $\Q_{\g}$ for $\g >0$.
\begin{prp}\phantomsection\label{prop:QgL1} For any  {$k \ge0$} and any $f,g\in L^{1}(\w_{k+\g})$,
\begin{equation*}
\|\Q_{\g}(f,g)\w_{k}\|_{L^1} \leq 2 \|f\w_{k+\g}\|_{L^1}\,\|g\w_{k+\g}\|_{L^1}\,.
\end{equation*}
\end{prp}
\begin{proof}
It follows that
\begin{eqnarray*}
\Q_{\g}(f,g)& = & \int_\R f\left(x + \frac{y}{2}\right) g\left(x - \frac{y}{2}\right) |y|^\gamma \dy \\
& \quad &-  \frac12 f(x) \int_\R g(y) |x-y|^\gamma \dy- \frac12 g(x) \int_\R f(y) |x-y|^\gamma \dy,
\end{eqnarray*}
and, consequently, 
\begin{multline*}
\|\Q_{\g}(f,g)\w_{k}\|_{L^1} \leq \int_{\R}\int_{\R} |f(x)| |g(y)| |x-y|^{\g} \w_{k}\left(\frac{x+y}{2}\right)\dy \dx \\
+ \frac12 \int_{\R} \int_\R|f(x)|  |g(y)| |x-y|^\gamma \w_{k}(x) \dy \dx + \frac12 \int_\R\int_\R |g(x)| |f(y)| |x-y|^\gamma \w_{k}(x) \dy \dx \,.
\end{multline*}
Since $|x-y|^{\g}\leq \w_{\g}(x)\w_{\g}(y)$ and $\w_{k}\left(\frac{x+y}{2}\right) \le \w_{k}(x)\w_{k}(y)$, Proposition \ref{prop:QgL1} readily follows.
\end{proof}
We now turn to $L^{2}$ estimates for $\Q^{+}_{\g}$ and $\Q^{-}_{\g}$ for $\g >0$.
\begin{prp}\phantomsection\label{prop:QgL2} For any  {$k \ge0$}, there is $C=C_{k,\g}>0$ such that
\begin{equation*}
\|\Q_{\g}^{-}(f,f)\w_{k}\|_{L^2} \leq \|f\w_{\g}\|_{L^1}\,\|f\w_{k+\g}\|_{L^2},
\end{equation*}
and 
\begin{equation}\label{eq:Qgffplus}
\|\Q_{\g}^{+}(f,f)\w_{k}\|_{L^2} \leq C\, \|f\|_{L^2}\,\|f\w_{k+\g}\|_{L^1}.
\end{equation}  
\end{prp}
\begin{proof}
Note that
$$\Q_{\g}^{+}(f,f)=\int_\R f\left(x + \frac{y}{2}\right) f\left(x - \frac{y}{2}\right) |y|^\gamma \dy, \quad \quad \Q_{\g}^{-}(f,f)=f(x) \int_\R f(y) |x-y|^\gamma \dy.$$
and, in particular, 
$$\|\Q_{\g}^{-}(f,f)\w_{k}\|_{L^2}^{2}\leq \int_{\R}|f(x)|^{2}\w^{2}_{k}(x)\left[\int_{\R}|f(y)|\,|x-y|^{\g}\dy\right]^{2}\dx \leq \|f\w_{k+\g}\|_{L^2}^{2}\|f\w_{\g}\|_{L^1}^{2}.$$
We used that $|x-y|^{\g}\leq \w_{\g}(x)\w_{\g}(y)$.  For the $\Q^{+}_{\g}(f,f)$ term one can for instance use that
\begin{equation}\label{eq:Q+}
\|\Q_{\g}^{+}(f,f)\w_{k}\|_{L^2}=\sup_{\|\varphi\|_{L^2}\leq 1}\int_{\R}\Q_{\g}^{+}(f,f)\w_{k}(x)\varphi(x)\dx.
\end{equation}
For $\varphi\in L^2(\R)$, we set
$$ I_1^{+}(\varphi)=\int_{\R^2} \mathds{1}_{\{|x-y|\le 2|x|\}}f(x)f(y)|x-y|^{\g}\w_{k}\left(\frac{x+y}{2}\right)\varphi\left(\frac{x+y}{2}\right)\dx\dy,$$
and
$$ I_2^{+}(\varphi)=\int_{\R^2} \mathds{1}_{\{|x-y|>2|x|\}} f(x)f(y)|x-y|^{\g}\w_{k}\left(\frac{x+y}{2}\right)\varphi\left(\frac{x+y}{2}\right)\dx\dy.$$ 
Since 
\begin{eqnarray*}
 \mathds{1}_{\{|x-y|\le 2|x|\}} |x-y|^{\g}\w_{k}\left(\frac{x+y}{2}\right) & \le & 2^{\g} |x|^{\g} \left(1+\frac{|x-y|+2|x|}{2} \right)^k\mathds{1}_{\{|x-y|\le 2|x|\}} \\
& \le & 2^{\g} |x|^{\g} (1+2|x|)^k \le 2^{k+\g}  \w_{k+\g}(x), 
\end{eqnarray*}
we deduce that 
\begin{eqnarray}
|I_1^{+}(\varphi)| & \le & 2^{k+\g}\int_{\R} \int_{|y-x|\le 2|x|} |f(x)| |f(y)| \w_{k+\g}(x) \varphi\left(\frac{x+y}{2}\right)\dy \dx \nonumber\\
& \le &  2^{k+\g} \int_{\R} |f(x)|  \w_{k+\g}(x) \|f\|_{L^2} \left\| \varphi\left(\frac{x+\cdot}{2}\right)\right\|_{L^2} \dx \nonumber\\
& \le &  2^{k+\g+\frac12}  \|f\|_{L^1(\w_{k+\g})} \|f\|_{L^2} \|\varphi\|_{L^2}. \label{I1+}
\end{eqnarray}
Now, since $|x-y|>2|x|$ implies that $|y|>|x|$, we have that
$$ \mathds{1}_{\{|x-y|> 2|x|\}} |x-y|^{\g}\w_{k}\left(\frac{x+y}{2}\right) \le  (|x|+|y|)^{\g} \left(1+\frac{|x|+|y|}{2} \right)^k \ \le 2^{\g} \w_{k+\g}(y).$$
Consequently, 
\begin{eqnarray}
I_2^+(\varphi) & \le & 2^{\g}\int_{\R} \int_{|y-x|> 2|x|} |f(x)| |f(y)| \w_{k+\g}(y) \varphi\left(\frac{x+y}{2}\right)\dx \dy \nonumber\\
 & \le &  2^{\g} \int_{\R} |f(y)|  \w_{k+\g}(y) \|f\|_{L^2} \left\| \varphi\left(\frac{\cdot+y}{2}\right)\right\|_{^2} \dy \nonumber \\
& \le &  2^{\g+\frac12}  \|f\|_{L^1(\w_{k+\g})} \|f\|_{L^2} \|\varphi\|_{^2}.\label{I2+}
\end{eqnarray}
Finally, \eqref{eq:Qgffplus} follows directly from \eqref{eq:Q+}, \eqref{I1+} and \eqref{I2+}.  
\end{proof} 
We establish now a comparison estimate between $\Q_{\g}$ and $\Q_{0}$ in the limit $\g \to 0$.
\begin{prp}\phantomsection\label{diff_Q_v2}
  Let $2<a<3$, $p>1$, $\gamma\in(0,1)$, and $s\in(0,1)$. Then, for any $f \in L^{1}(\w_{s+\g+a})$ and $g\in L^p(\w_a)\cap L^1(\w_{s+\gamma+a})$, it holds  
  \begin{multline*}
\|\Q_{0}(f,g)-\Q_\gamma(f,g)\|_{L^1(\w_{a})} 
\le 4{\g} \|f\|_{L^1(\w_{a})}  \|g\|_{L^p(\w_{a})}\\
+ \frac{4p}{p-1} \g |\log\g| \|f\|_{L^1(\w_{a})} \|g\|_{L^1(\w_{a})} + \frac{16\g}{s} \|f\|_{L^1(\w_{s+\gamma+a})} \|g\|_{L^1(\w_{s+\gamma+a})}.
    \end{multline*}
\end{prp}
\begin{proof} First, a change of variables leads to
\begin{align*}
\|\Q_{0}(f,g)-\Q_\gamma(f,g)\|_{L^1(\w_{a})}& \le\int_\R\int_\R |f(x)| |g(y)| \left|1-|x-y|^\gamma\right|  {\w_{a}\left(\frac{x+y}{2}\right)} \dx\dy \\
& +\int_\R\int_\R|f(x)| |g(y)|  \left|1-|x-y|^\gamma\right|  {\left(\frac{\w_{a}(x)}{2}+\frac{\w_{a}(y)}{2}\right)}\d x\d y.
\end{align*}
Due to the convexity of $r\mapsto r^a$ for $a\in(2,3)$,  $\w_{a}(\frac{x+y}{2})\le \frac{1}{2}\left(\w_{a}(x)+\w_{a}(y)\right)$, it follows that 
\begin{eqnarray}
\|\Q_{0}(f,g)-\Q_\gamma(f,g)\|_{L^1(\w_{a})}& \le & \int_\R\int_\R|f(x)| |g(y)|  \left|1-|x-y|^\gamma\right|  \w_{a}(x)\d x\d y \nonumber\\
& + & \int_\R\int_\R|f(x)| |g(y)|  \left|1-|x-y|^\gamma\right|  \w_{a}(y)\d x\d y\,. \label{diffQ2}
\end{eqnarray}
Let $0<\delta<1$. Splitting the above integrals according to the regions $|x-y|\le \delta$, $\delta<|x-y|<1$, $1\le |x-y|\le 2|x|$, and $|x-y|\geq {\max\{1,2|x|\}}$, we get that
$$\|\Q_{0}(f,g)-\Q_\gamma(f,g)\|_{L^1(\w_{a})} \le I_1+I_2+I_3+I_4+I_5,$$
with
\begin{align*}
I_1 &= \int_\R \int_{|x-y|\le\delta} |f(x)| |g(y)|  \left|1-|x-y|^\gamma\right| \w_{a}(x)\d x\dy, \\
I_2 &= \int_\R \int_{|x-y|\le\delta} |f(x)| |g(y)|  \left|1-|x-y|^\gamma\right| \w_{a}(y)\d x\dy, \\
I_3& = 2 \int_\R\int_{\delta<|x-y|<1}  |f(x)| |g(y)|  \left|1-|x-y|^\gamma\right| \w_{a}(x)\w_{a}(y)\d x \dy, \\
I_4 &= 2\int_\R\int_{1\le|x-y|\le 2|x|} |f(x)| |g(y)|  \left|1-|x-y|^\gamma\right| \w_{a}(x)\w_{a}(y)\d x\dy,  \\
I_5 &= 2 \int_\R \int_{|x-y|> \max\{1,2|x|\}} |f(x)| |g(y)|  \left|1-|x-y|^\gamma\right|\w_{a}(x)\w_{a}(y)\d x\dy.
\end{align*}
Since $\delta<1$, for  $|x-y|\le \delta$, we have $\left|1-|x-y|^\gamma\right|\le1$.  Consequently, we deduce from H\"older's inequality that
  $$ I_1 \le \int_{\R} |f(x)| \left(\int_{|x-y|\le \delta}\dy \right)^{\frac{p-1}{p}}\left(\int_\R |g(y)|^p \dy \right)^\frac{1}{p}\w_{a}(x)\, \dx \le  (2\delta)^{\frac{p-1}{p}} \|g\|_{ L^p} \|f\|_{L^1(\w_{a})},$$
  and
  $$ I_2 \le  \int_{\R} |f(x)| \left(\int_{|x-y|\le \delta}\dy \right)^{\frac{p-1}{p}}\left(\int_\R |g(y)|^p  \w_{a}^p(y)\dy \right)^\frac{1}{p} \dx \le   (2\delta)^{\frac{p-1}{p}} \|g\|_{ L^p(\w_{a})} \|f\|_{L^1} ,$$
  for $p>1$. Next, 
  \begin{equation}\label{ineq1}
  |r^{\g}-1| = \left|\int_0^{\g} r^{z} \log r \d z\right|\le \g |\log r| (1+r)^{\g},\qquad  r>0,
  \end{equation}
  implies that 
  \begin{eqnarray*}
  I_3 & \le & 2\g \int_{\R}\int_{\delta<|x-y|<1} |\log|x-y|| \left(1+|x-y|\right)^{\g}  |f(x)| \,\w_{a}(x) |g(y)| \,\w_{a}(y)\,\dy \dx  \\
  & \le & \g \,|\log \delta|\, 2^{\g+1} \,\|g\|_{L^1(\w_{a})} \|f\|_{L^1(\w_{a})}.
  \end{eqnarray*}
  Now, we also deduce from \eqref{ineq1} that 
  $$I_4 \le 2\g \int_{\R}\int_{1\le|x-y|\le 2|x|} |\log|x-y|| \left(1+|x-y|\right)^{\g}  |f(x)| \w_{a}(x)\,|g(y)| \w_{a}(y)\,\dy \dx .$$
 Then, since $1\le |x-y|\le 2|x|$ implies that $|x|\ge \frac12$ and since $s\log r \le r^s$ for any $r\ge 1$ and $s\in(0,1)$, we obtain that
  \begin{eqnarray*}
  I_4 & \le & 2\g \int_{{2|x|\ge 1}} \log(2|x|)\, (1+2|x|)^{\g} \,|f(x)| \w_{a}(x)\,\dx \|g\|_{L^1(\w_{a})}  \\
  & \le &  \frac{2\g}{s} \int_{\R}(2|x|)^s \,(1+2|x|)^{\g} |f(x)| \w_{a}(x)\,\dx \,\|g\|_{L^1(\w_{a})} \\
  & \le & \frac{\g}{s}\,2^{\g+s+1} \int_{\R} (1+|x|)^{\g+s}\,|f(x)| \w_{a}(x)\,\dx \|g\|_{L^1(\w_{a})} \le \frac{8\g}{s} \,\|f\|_{L^1(\w_{a+s+\g})} \,\|g\|_{L^1(\w_{a})}. 
  \end{eqnarray*}
Similarly, we deduce from \eqref{ineq1} that 
  $$I_5 \le 2 \g \int_{\R} \int_{|x-y|> \max\{1,2|x|\}} |\log|x-y|| \left(1+|x-y|\right)^{\g} |f(x)| \,\w_{a}(x) |g(y)| \,\w_{a}(y)\,\dy \dx.$$
  Then, $|x-y|>2|x|$ implies that $|y|\ge |x-y|-|x|\ge |x|$ and $|x-y|\le |x|+|y|\le 2|y|$. Moreover, since $s\log r \le r^s$ for any $r\ge 1$ and $s\in(0,1)$, we obtain 
    \begin{eqnarray*}
    I_5 & \le & \frac{2\g }{s}  \int_{\R} \int_{|x-y|>2|x|}  (1+|x-y|)^{\g+s} |f(x)| \w_{a}(x)\,|g(y)| \w_{a}(y)\,\dy \dx \\
  & \le & \frac{2\g}{s} \int_{\R} \int_{\R}  (1+2|y|)^{\g+s} |f(x)| \w_{a}(x)\,|g(y)| \w_{a}(y)\,\dy \dx \le \frac{8\g}{s} \; \|f\|_{L^1(\w_{a})}\, \|g\|_{L^1(\w_{a+s+\g})}. 
  \end{eqnarray*}
Choosing $\delta=\g^{\frac{p}{p-1}}\in(0,1)$ and combining the above estimates completes the proof of Proposition \ref{diff_Q_v2}. 
\end{proof} 
We recall here some of the notations introduced in Section \ref{sec:upgrade}. Namely, set
$$\mathscr{I}_{0}(f,g)=\int_{\R^{2}}f(x)g(y)|x-y|^{2}\Lambda_{0}\left(|x-y|\right)\dx\dy, \qquad f,g \in L^{1}(\w_{s}), \quad s>2,$$
and
$$\mathscr{I}_{\g}(f,g)=\int_{\R^{2}}f(x)g(y)|x-y|^{2}\Lambda_{\g}\left(|x-y|\right)\dx\dy, \qquad \quad f,g \in L^{1}(\w_{2+\g})\,,$$ 
where
$$\Lambda_{0}(r)=\log r,\qquad \qquad \Lambda_{\g}(r)=\frac{r^{\g}-1}{\g}, \qquad \g>0,\; r >0.$$
One has, then, the following elementary observation.
\begin{lem}\phantomsection\label{lem:I0fg}
For $a >2,$ $f,g\in L^{1}(\w_{a})$, it holds that
$$\left|\mathscr{I}_{0}(f,g)\right| \leq C_{a}\|f\|_{L^{1}(\w_{a})}\|g\|_{L^{1}(\w_{a})}$$
for some positive constant $C_{a} >0$ depending only on $a.$
\end{lem}
\begin{proof} Up to replacing $f$ with $|f|$ and $g$ with $|g|$, we may assume without loss of generality that both $f$ and $g$ are nonnegative.  Therefore,
\begin{equation*}\begin{split}
{\left|\mathscr{I}_{0}(f,g)\right|}&{\le}\int_{|x-y|>1}f(x)g(y)|x-y|^{2}\log|x-y|\dx\dy\\
&\phantom{++++} +\int_{|x-y|<1}f(x)g(y)|x-y|^{2}\bigl|\log|x-y|\bigr|\dx\dy\\
&\leq c_{a}\int_{|x-y|>1}f(x)g(y)|x-y|^{a}\dx\dy+{\frac{1}{2e}}\|f\|_{L^1}\|g\|_{L^1}\end{split}\end{equation*}
because there is a $c_{a} >0$ such that $\log r \leq c_{a}r^{a-2}$ for any $r >1$ and $a >2$, and 
{$r^2\bigl|\log r\bigr|\leq \frac{1}{2e}$ for $r \in (0,1)$}. This gives the result since {$|x-y|^{a} \leq \w_{a}(x)\w_{a}(y)$} and $\|\cdot\|_{L^1}\leq \|\cdot\|_{L^{1}(\w_a)}$.
\end{proof}
Recall that $\lim_{\g\to0^{+}}\Lambda_{\g}(r)=\Lambda_{0}(r)$ for any $r >0$.  Consequently, the following estimate for the difference between $\mathscr{I}_{0}$ and $\mathscr{I}_{\g}$ follows. 
\begin{lem}\phantomsection\label{lem:IgfgI0}
Let $2<a<3$ and $\gamma\in\left(0,\frac{a-2}{2}\right)$. Then, there exists a constant $C_{a}$ depending only on $a$ such that, for any $f,g \in L^{1}(\w_{a})$, it holds  
$$\left|\mathscr{I}_{\g}(f,g)-\mathscr{I}_{0}(f,g)\right| \leq C_{a}\, \g\, \|f\|_{L^1(\w_{a})}\|g\|_{L^1(\w_{a})}.$$
\end{lem}
\begin{proof} As before, we assume without loss of generality that $f,g$ are nonnegative and observe that
$$\left|\mathscr{I}_{\g}(f,g)-\mathscr{I}_{0}(f,g)\right| \leq \int_{\R^{2}}f(x)g(y)|x-y|^{2}\left|\Lambda_{\g}(|x-y|)-\Lambda_{0}(|x-y|)\right|\dx\dy.$$
By Taylor's expansion, for $0<\g<s<1$, 
\begin{eqnarray*}
\left|\Lambda_{\g}(r)-\Lambda_{0}(r)\right|  = \left|\frac{r^\g-1}{\g}- \log r\right| & = & \frac{(\log r)^2}{\g} \int_0^{\g} (\g-z) r^z \d z\\
& \le &  \frac{(\log r)^2}{\g} {(1+r)^s}\frac{\g^2}{2} = \frac{\g}{2}  {(1+r)^s}(\log r)^2.
\end{eqnarray*}
Therefore, 
$$ \left|\mathscr{I}_{\g}(f,g)-\mathscr{I}_{0}(f,g)\right| \leq \frac{\g}{2} \int_{\R^{2}}f(x)g(y)|x-y|^{2}(\log|x-y|)^2  {(1+|x-y|)^s}\dx\dy.$$
Since there is $c_{\eta} >0$ such that $(\log r)^2 \leq c_{\eta} r^{\eta}$ for any $r >1$ and $\eta\in(0,1)$ and 
$\left(r\log r\right)^2\leq \frac{1}{e^2}$ for any $r \in (0,1)$, one has then  
\begin{eqnarray*}
\left|\mathscr{I}_{\g}(f,g)-\mathscr{I}_{0}(f,g)\right| & \leq & \frac{\g c_\eta}{2} \int_{|x-y|>1} f(x)g(y)|x-y|^{2+\eta}  {(1+|x-y|)^s}\dx\dy \\
& & \qquad + \frac{\g}{2e^2} \int_{|x-y\le 1} f(x)g(y)  {(1+|x-y|)^s}\dx\dy \\
& \le &  {\frac{\g}{2}} \left(c_\eta+  \frac{1}{e^2}\right) \int_{\R^2} f(x)g(y)(1+|x-y|)^{2+\eta+s}\dx\dy \\
&\le &    {\frac{\g}{2}}\left(c_\eta+  \frac{1}{e^2}\right)  \|f\|_{L^1(\w_{2+\eta+s})}  \|g\|_{L^1(\w_{2+\eta+s})}.
\end{eqnarray*}
Since $a\in(2,3)$ and $\g\in\left(0, \frac{a-2}{2}\right)$, taking  $\eta=\frac{a-2}{2}\in(0,1)$ and $s=\frac{a-2}{2}\in(\g,1)$ completes the proof. 
\end{proof}

\bibliographystyle{plainnat-linked}

\end{document}